\documentclass{article}

\usepackage[utf8]{inputenc}

\usepackage{mathrsfs}
\usepackage{amsmath}
\usepackage{hyperref}
\usepackage{amsthm}
\usepackage{amssymb}
\usepackage{parskip}
\usepackage{xcolor}
\usepackage{enumitem}

\hypersetup{
	colorlinks=true,
	linkcolor=blue,
	citecolor=blue,
	filecolor=magenta,      
	urlcolor=cyan,
	linktoc=page
}

\theoremstyle{plain}
\newtheorem{theorem}{Theorem}[section]
\newtheorem{corollary}[theorem]{Corollary}
 
\newtheorem{prop}[theorem]{Proposition}
\newtheorem{proposition}[theorem]{Proposition}
\newtheorem{lemma}[theorem]{Lemma}

\theoremstyle{definition}
\newtheorem{definition}[theorem]{Definition}
\newtheorem{remark}[theorem]{Remark}

\def \R{\mathbb R}
\def \N{{\mathbb N}}
\def\E{\mathbb E}
\def\P{\mathbb P}
\def \Z{\mathbb Z}
\newcommand{\cA}{\mathcal{A}}
\newcommand{\cB}{\mathcal{B}}
\newcommand{\cC}{\mathcal{C}}
\newcommand{\cD}{\mathcal{D}}
\newcommand{\cE}{\mathcal{E}}

\newcommand{\cG}{\mathcal{G}}
\newcommand{\cH}{\mathcal{H}}
\newcommand{\cI}{\mathcal{I}}
\newcommand{\cJ}{\mathcal{J}}
\newcommand{\cK}{\mathcal{K}}
\newcommand{\cL}{\mathcal{L}}
\newcommand{\cM}{\mathcal{M}}
\newcommand{\cN}{\mathcal{N}}
\newcommand{\cO}{\mathcal{O}}
\newcommand{\cP}{\mathcal{P}}
\newcommand{\cQ}{\mathcal{Q}}
\newcommand{\cR}{\mathcal{R}}
\newcommand{\cS}{\mathcal{S}}
\newcommand{\cT}{\mathcal{T}}
\newcommand{\cU}{\mathcal{U}}
\newcommand{\cW}{\mathcal{W}}
\newcommand{\cV}{\mathcal{V}}
\newcommand{\cX}{\mathcal{X}}
\newcommand{\cY}{\mathcal{Y}}
\newcommand{\cZ}{\mathcal{Z}}

\newcommand{\floor}[1]{{\left\lfloor #1 \right\rfloor}}
\newcommand{\ceil}[1]{{\left\lceil #1 \right\rceil}}

\newcommand{\fW}{\mathfrak{W}}
\newcommand{\fQ}{\mathfrak{Q}}

\newcommand{\sB}{\mathscr{B}}
\newcommand{\sC}{\mathscr{C}}
\newcommand{\sD}{\mathscr{D}}
\newcommand{\sE}{\mathscr{E}}
\newcommand{\sF}{\mathscr{F}}
\newcommand{\sI}{\mathscr{I}}

\newcommand{\sM}{\mathscr{M}}
\newcommand{\sP}{\mathscr{P}}
\newcommand{\sG}{\mathscr{G}}
\newcommand{\sH}{\mathscr{H}}
\newcommand{\sU}{\mathscr{U}}

\newcommand{\sW}{\mathscr{W}}
\newcommand{\oa}{{\overline{a}}}
\newcommand{\op}{{\overline{p}}}
\newcommand{\oq}{{\overline{q}}}

\newcommand{\del}{\partial}

\newcommand{\sset}{\subset}
\newcommand{\lf}{\left}
\newcommand{\rg}{\right}

\newcommand{\ga}{\gamma}

\newcommand{\ep}{\epsilon}

\newcommand{\ka}{\kappa}
\newcommand{\de}{\delta}

\newcommand{\sig}{\sigma}

\newcommand{\eps}{\epsilon}
\newcommand{\la}{\lambda}
\newcommand{\al}{\alpha}

\newcommand{\smin}{\setminus}

\newcommand{\fR}{\mathfrak{R}}
\newcommand{\fB}{\mathfrak{B}}
\newcommand{\fA}{\mathfrak{A}}
\newcommand{\fC}{\mathfrak{C}}
\newcommand{\fD}{\mathfrak{D}}

\title{The SIR model in a moving population: propagation of infection and herd immunity}
\author{Duncan Dauvergne and Allan Sly}

\begin{document}

	\maketitle
	
	\begin{abstract}
		In a collection of particles performing independent random walks on $\Z^d$ we study the spread of an infection with SIR dynamics.  Susceptible particles become infected when they meet an infected particle. Infected particles heal and are removed at rate $\nu$. We show that when $\nu$ is small, with positive probability the infection survives forever and grows linearly.  Furthermore, after the infection reaches a region, it quickly passes through and leaves behind a \emph{herd immunity} regime consisting of recovered particles, a small positive density of susceptible particles, and no infected particles. One notable feature of this model is the simultaneously existence of supercritical and subcritical phases on either side of an infection front of $O(1)$ width.
	\end{abstract}
	
	
	\section{Introduction}
	
	We study a model for the spread of an infection in a collection of moving interacting particles, based on the well-known epidemiological \textbf{SIR}  (Susceptible-Infected-Removed) model. The process is initiated from a Poisson process of susceptible particles on $\Z^d$ with density $\mu > 0$. All particles perform independent continuous-time simple random walks on $\Z^d$.  At time $0$ a single infected particle is added at the origin $0$.  When an infected particle and a susceptible particle meet at a vertex, the susceptible particle becomes infected.  Infected particles recover at rate $\nu > 0$, after which point they are removed and no longer affect the process. 
	Our goal is to understand the evolution of this process and how the infection spreads over time. 
	
	This model fits within a broader class of dynamical interacting particle systems including Susceptible-Infected (SI also called A/B or X/Y) and Susceptible-Infected-Susceptible (SIS) models of infection spread, multi-particle Diffusion Limited Aggregation (mDLA), frog models, and activated random walks. See Section \ref{SS:related} for discussion of related work on these models. Each of these models can be viewed as an infection spreading through a population. Typically, we are interested in understanding the rate of spread and properties of the infected region.
	In contrast to the other models, the challenge here is that particles remain infected for only $O(1)$ time before being removed, and so particle density decreases as the infection spreads. Therefore, even if the infection survives forever we do not expect to see a growing ball of infected particles. Rather, we expect a growing sphere of particles of $O(1)$ width dividing the plane between an outer supercritical region with a high density of susceptible particles and a low density interior in a subcritical \emph{herd immunity} regime with too few susceptible particles to sustain an epidemic.
	
	The SIR model originates in epidemiology and is generally studied in the mean field setting where any particle can infect any other~\cite{kermack1927contribution,ross1917application}. In large populations its evolution can be studied by deterministic differential equations (see \cite{anderson1991discussion} for a survey of the plethora of variants and applications). It has also been studied in the case where particles occupy fixed vertices in a graph and can infect neighbouring vertices. In the lattice setting, this model was first studied in detail by Kuulasma~\cite{kuulasmaa1982spatial} building on work in simpler models by Mollison~\cite{mollison1977spatial}.
	In the case of random graphs with given degree distributions, Newman~\cite{newman2002spread} characterized the critical threshold for infection spread. 
	
	The case of moving particles is more mathematically challenging as the dynamic environment of particles has a complicated dependence on the infection process.  Moreover, an additional challenge is that, unlike in various SI or SIS models, it lacks any obvious monotonicity in the density and recovery rate parameters $\mu$ and $\nu$. For example, if we couple two SIR models at different densities, then while the higher density model initially has more infections, this may cause particles to recover earlier, decreasing the density and breaking the monotone coupling between the models.
	
	Our first main theorem shows that the SIR model displays qualitatively different behaviour at high and low recovery rates.
	We let ${\bf S_t}$ and ${\bf I_t}$ denote the sets of susceptible and infected particles at time $t$, respectively.  For this next theorem, we let $\P(d, \mu, \nu)$ denote the \textbf{survival probability} for the infection process in dimension $d$, density $\mu$ and recovery rate $\nu$. That is, $\P(d, \mu, \nu)$ is the probability that ${\bf I_t} \ne \emptyset$ for all $t > 0$.
	
	\begin{theorem}
		\label{T:main-1}
		Fix any density $\mu > 0$ and any dimension $d \ge 2$. Then there exist $\nu^- = \nu^-(\mu, d), \nu^+ = \nu^+(\mu, d) \in (0, \infty)$ such that if $\nu > \nu^+$, then $\P(d, \mu, \nu) = 0$ and if $\nu < \nu^-$ then $\P(d, \mu, \nu) > 0$. Moreover, $\P(d, \mu, \nu) \to 1$ as $\nu \to 0$.
	\end{theorem}
	
	We expect that $\P(d, \mu, \nu)$ is a monotone function of $\nu$ and therefore there exists a critical recovery rate $\nu_c \in (0,\infty)$. However, proving this monotonicity appears to be a  difficult problem for the reasons discussed above. Indeed, closely related processes are not monotone on certain graphs, see~\cite{deijfen2006nonmonotonic, candellero2021first}.
	
	The fact that $\P(d, \mu, \nu) = 0$ for all large enough $\nu$ was previously established by Kesten and Sidoravicius in \cite{kesten2006phase} for the SIS model where particles become susceptible again after recovering from the infection. Their result immediately implies ours since the infection process for the SIS model stochastically dominates the infection process for the SIR model and so this part of Theorem~\ref{T:main-1} is not new. However, their proof is quite involved as it has to deal with the possibility of particles becoming reinfected. In the SIR setting, a much more straightforward supermartingale argument works, see Section \ref{S:recovery}. Recently, Grimmett and Li~\cite{grimmett2020brownian} have shown that $\P(d, \mu, \nu) = 0$  for large enough $\nu$ in a Brownian version of our SIR model. Their broad proof strategy is similar to ours but differs in the precise setup (and naturally, there are also different technicalities that come from working in discrete vs. continuous space). The proof that $\P(d, \mu, \nu) \to 1$ as $\nu \to 0$ is much more difficult and is given in Section \ref{S:linear}.
	
	A version of Theorem \ref{T:main-1} also holds if we first fix $\nu$ and allow the density $\mu$ to vary instead. Indeed, the results of Section \ref{S:recovery} show that for fixed $\nu$, if $\mu$ is low enough then $\P(d, \mu, \nu) = 0$. A variant of the arguments in Section \ref{S:linear} would show that $\P(d, \mu, \nu) \to 1$ as $\mu \to \infty$, but the details are sufficiently different that we do not include the proof here.
	
	\begin{remark}[Dimension $1$]
		In dimension $1$, it is not difficult to check that the SIR process dies out exponentially quickly at every recovery rate $\nu$ and every density $\mu$. Indeed, by Theorem~1 in~\cite{kesten2005spread}, there exists a constant $C > 0$ such that the set of infected particles $I_t$ is contained in the interval $[-Ct, Ct]$ with exponentially high probability. Since particles move diffusively and each particle is only infected for an $O(1)$ amount of time, this implies that there is a constant $C' \in \N$ such that at most times $t$, at most $C'$ total particles are infected. At any such time $t$ there is a positive probability depending on $C'$ but not on $t$ that all of these particles recover before encountering any susceptible particles. Hence the infection must die out exponentially quickly. 
	\end{remark}
	
	Theorem \ref{T:main-1} establishes that the SIR process can exhibit both recovery regimes and survival regimes. While the recovery regime is not particularly interesting, there are many possibilities for how the process might behave in the survival regime. The next theorem describes the particular behaviour of the process in the survival regime.

	\begin{theorem}
		\label{T:main-2}
		Fix any density $\mu > 0$, any dimension $d \ge 2$, and $\nu < \nu^-(\mu, d)$. Then there is a positive probability event $\cG_\nu$ with $\P[\cG_\nu] \to 1$ as $\nu \to 0$, such that on $\cG_\nu$ the following events hold. Here $c, C > 0$ are constants that depend on $\nu, \mu, d$.
		\begin{enumerate}
			\item (Survival) ${\bf I_t} \ne \emptyset$ for all $t \ge 0$.
			\item (Linear growth) For all large enough $t$, we have
			$$
			{\bf I_t} \sset B(0, C t) \smin B(0, c t) .
			$$
			\item (Infection Duration) For every site $x \in \Z^d$, let $D_x$ be the length of time between the first appearance of an infected particle at site $x$ and the last appearance. If an infected particle never reaches site $x$, set $D_x = 0$. Then for all $x \in \Z^d$ and $m \ge 3$ we have
		\begin{equation}
		\label{E:Dxm}
			\P(D_x > m \;|\; \cG_\nu) \le C\exp(-m^{c/\log \log m}).
			\end{equation}
			In particular, on $\cG_\nu$ there exists a random constant $D > 0$ such that for all $x \in \Z^d$ we have 
			\begin{equation}
			\label{E:DxDbound}
			D_x \le D + [\log (\|x\|_2 + 3)]^{C \log \log \log (\|x\|_2 + 3)}. 
			\end{equation}
			\item (Herd Immunity in the centre) We have
			$$
			\liminf_{t \to \infty} \frac{|{\bf S}_t \cap B(0, c t) |}{t^d} > 0.
			$$
		\end{enumerate}
	\end{theorem} 
	
	We strongly believe that $\cG_\nu$ can be taken to be the survival event 
	$$
	\{\textbf{I}_t \ne \emptyset \text{ for all } t\}.
	$$
	However, our methods do not suffice to show this. 
	
	\subsection{Proof Sketch}	
	In~\cite{dauvergne2021spread} we studied the evolution of infections in the SI model where particles do not recover.  Previous work of Kesten and Sidoravicius~\cite{kesten2005spread,kesten2008shape} had analyzed this model in the setting where susceptible and infected particles move with the same speed. Under this assumption, unlabeled particle trajectories are independent random walks and are not affected by the infection process.  This work left open the problem of models where the particle environment was dependent on the infection process.
	Our work in~\cite{dauvergne2021spread} developed a technical framework to establish linear growth in general SI models which is flexible enough to study many other models. 
	 In the present work we use it to understand the different phenomena that occur in SIR models, particularly the way in which the infection process divides the plane into subcritical and supercritical regimes divided by an infected region of $O(1)$ width.
	
	We divide the lattice into blocks of fixed side length $L$ chosen based on the recovery rate $\nu$ (see equation~\eqref{E:Lnu-relationship}). We define a colouring process on the set of blocks where a block is coloured the first time an infected particle enters it or if it has been next to a coloured block for a sufficient period of time (see Section~\ref{S:SI-colouring} for the precise definition).  In order to establish sufficient spatial independence, we only observe particles that have previously visited the coloured region.  Outside of the coloured region, the conditional distribution of the particles is conditionally Poisson with a random intensity that can be bounded from below. This description allows us to define the SIR process in a spatially independent way according to the colouring process.
	 To each block we associate an independent Poisson process and use this to generate the particles inside that block at the time it is coloured, as well as their future trajectories and healing times.
	
	
	The lower bound on the intensity implies a high-probability lower bound on the number of particles present in a block $B$ at the random time $\tau_B$ when it is first coloured. This implies that for most blocks, if they are infected at time $\tau_B$ then the infection will quickly be propagated to the neighbouring blocks irrespective of where the infection first enters.  Those that do not have this property are labeled as \emph{blue seeds}. Our block construction of the SIR process using independent Poisson processes guarantees that the blue seed process is stochastically dominated by a highly subcritical percolation in the low recovery rate regime. To establish that the infection survives and grows linearly we couple the colouring process to a competitive growth model called Sidoravicius--Stauffer percolation (SSP) first used to study multi-particle DLA~\cite{sidoravicius2019multi}.  
	
	Our analysis here goes beyond the questions addressed in~\cite{dauvergne2021spread}, and a particular focus is the study of the herd immunity regime left behind after the infection passes through a region. Our aim is to show that in this herd immunity regime, there are no infected particles remaining while a small density of susceptible particles remains uninfected. To show that the infection dies out locally after the wave of infections passes through, we first use random walk estimates to show that the density of unremoved particles is small in regions that were coloured long ago. We then combine this with a more delicate version of the supermartingale argument used to analyze the high recovery rate regime to prove that the infection dies out completely with high probability.
	
	 To show that there is a small density of susceptible particles that never become infected we show that locally each block has a small probability of having a particle that survives for a long period of time after $\tau_B$. After the wave of infections has passed through, it is unlikely that there are ever infected particles in a neighbourhood of this susceptible particle. In order to show that many such particles survive we show that this event can be bounded below by one that is essentially local. 
	 
	 A particular challenge in establishing the features of the herd immunity regime is to show that the events in question can essentially be defined locally using the independent Poisson blocks, i.e. that they do not depend on the relative colouring times $\tau_B$ or on the behaviour of the coupled Sidoravicius--Stauffer percolation, both of which in principle can have long range dependencies.
	
	\subsection{Related work}
	\label{SS:related}
	
	The type of SIR model we study in this paper fits within a broader class of dynamical interacting particle systems
	built from random walks on the lattice. If we take our SIR model and make all susceptible particles stationary, then we recover the frog model on $\Z^d$ (with recovery). Alves, Machado, and Popov~\cite{alves2002phase} showed that 
	this model exhibits a phase transition, obtained asymptotics for the critical recovery rate, and found a shape theorem in the SI setting where recovery is not permitted~\cite{alves2002shape}. The frog model has also been well-studied in the SIS (Susceptible-Infected-Susceptible) setting, where frogs recover from the infection and return to a stationary susceptible state under the name `activated random walks', e.g. see~\cite{dickman2010activated, rolla2020activated, stauffer2018critical} for review articles and recent progress.
	
	The case when both susceptible and infected particles are allowed to move is more mathematically challenging as the dynamic environment of particles has a much more complicated dependence on the infection process. In the SI and SIS settings, as discussed above this model was seriously attacked in a series of papers of Kesten and Sidoravicius~\cite{kesten2005spread, kesten2008shape, kesten2006phase}, culminating in a shape theorem for the infected region in the SI model.  See~\cite{kesten2012asymptotic} for a survey of problems about these types of models and~\cite{baldasso2020local, baldasso2021local} for more recent work.
	
	A version of the SI model on $\Z^d$ where susceptible particles can move but infected particles cannot move and instead form a stationary growing aggregate is known as multi-particle diffusion limited aggregation (mDLA). While the character of mDLA is somewhat different (for example, the aggregate has a fractal-like structure at low densities), the techniques used to study mDLA are useful here. Indeed, this line of work introduced SSP, known there as first passage percolation in a hostile environment (FPPHE), as a tool in the study of mDLA in ~\cite{sidoravicius2019multi}. See~\cite{finn2020non, candellero2021coexistence, candellero2021first} for further developments on FPPHE. A simpler version of the block Poisson description of the SIR process was also first used to study mDLA in one dimension in~\cite{sly2021one}.
	
	\subsection{Outline of the paper}
	
	For simplicity, throughout the paper we assume that the dimension $d=2$. The proofs in the case of general $d \ge 2$ go through essentially verbatim.
	In Section \ref{S:recovery}, we give a martingale argument proving that $\P(d, \mu, \nu) = 0$ for all large enough $\nu$. This section also includes a technical extension of this result that we will use later on to establish local infection recovery in the survival regime. In Section \ref{S:SSP} we introduce SSP and in Section \ref{S:linear} we couple an SSP with a block description of the SIR process to prove that $\P(d, \mu, \nu) \to 1$ as $\nu \to 0$ and that the infection spreads linearly. The remaining sections prove parts $2, 3,$ and $4$ of Theorem \ref{T:main-2}. Section \ref{S:global} defines events that will allow us to separate out local and global effects on the SIR process. Section \ref{S:upper-bd-density} gives a local construction that shows that once the infection passes through a region, the density of unremoved particles in that region quickly decreases with high probability, leaving that region in a subcritical regime. Section \ref{S:survival} provides a converse to the results of Section \ref{S:upper-bd-density}, by defining a low probability local event where a particle survives long after the infection has passed through a region. Section \ref{S:herd} delicately puts together the constructions from all the previous sections to prove Theorem \ref{T:main-2}.

	\subsection{Notational conventions}
	
	Rather unfortunately, the paper is loaded with notation. Moving forward, we use the following conventions to help the reader orient themselves:
	\begin{itemize}[nosep]
		\item Events are typically denoted with calligraphic symbols $\cA, \cB, \cC, \dots$
		\item Large constants are denoted by $C, C', C_1, C_2, \dots$ and small constants are denoted by $c, c', c_1, c_2, \dots$. In sections \ref{S:recovery} and \ref{S:SSP}, all constants are absolute.
		From Section \ref{S:linear} onward, we will fix a density $\mu$ and all constants will depend on $\mu$ but no other parameters. We always allow the meaning of constants to change from line to line.
		\item We always use fraktur notation $\fA, \fB, \dots$ to describe objects arising from the study of Sidoravicius--Stauffer percolation in Section \ref{S:SSP}.
		\item We do not include ceilings and floors unless they materially affect arguments.  
		\item Many numbers, e.g. $4000, 5000$, are arbitrary. They are chosen so that various scales we use in the paper will nest together in the right ways. The reader should use the explicit constants only as a way to help orient themselves when we are considering multiple scales.
		\item For $x, y \in \Z^2$, we write $d(x, y) = |x-y|$ for the $L^1$-distance (or graph distance) from $x$ to $y$. More generally, for sets $A, B$ we write $d(x, A) = \inf\{|x - y| : y \in A\}$ and $d(A, B) = \inf\{|x - y| : x \in A, y \in B\}$.
		\item For a set $A \sset \Z^2$, we let $\del A = \{x \in A : d(x, A^c) = 1\}$ be the interior boundary of $A$.
	\end{itemize}
	
	\subsection*{Acknowledgements}
The authors would like to thank Alexander Stauffer for useful discussions.  AS was supported by NSF grants DMS-1855527 and DMS-1749103, a Simons Investigator grant and a MacArthur Fellowship. DD was supported by an NSERC Discovery Grant.
	
	\section{SIR recovery}
	\label{S:recovery}

	In this section, we show that in the high recovery rate regime, the infection process dies out exponentially quickly. A few modifications to the argument will also allow us to give a criterion for the SIR infection recovering in a local window where particles have an initial low density. This will be used later to establish that in the low recovery rate regime, the infection dies out quickly in a region after the infection front has passed through.
	
	\begin{prop}
		\label{P:high-recovery-rate}
		Consider an SIR process with recovery rate $\nu$, started from a random initial configuration of susceptible particles that is stochastically dominated by a Poisson process of intensity $\nu/8$ with an initial infected particle added at the origin $0 := (0,0)$. Then for all $t > 0$, we have 
		$$
		\P(\mathbf{I}_t \ne \emptyset) \le \E \mathbf{I}_t \le (1 + \nu/8)e^{-\nu t/2}.
		$$
	\end{prop}
	
	To set up the proof of Proposition \ref{P:high-recovery-rate} we introduce a few conventions that we use throughout the paper. First, for a particle $a$, we let $\bar a:\R \to \R$ denote its cadlag trajectory. We also associate to every particle an independent intensity $\nu$ Poisson process $\bar a^h$ on $\R$ which is the \textbf{healing process} for $a$. Each particle $a$ is healed at the first time in its healing process after it becomes infected. Standard results about Poisson processes implies that this gives the same dynamics as associating to each particle a rate-$nu$ exponential recovery clock.
	
	We will analyze the number of infected particles at a particular time by counting certain sequences of potentially infected particles known as active chains. 
	An \textbf{active chain $Z = (\Pi, Q)$ on an interval $[s, t]$} is a partition $\Pi = \{\pi_0 = s < \pi_1 < \dots < \pi_{k(Z)} = t\}$ and a sequence of particles $Q = \{a_i : i \in \{1, \dots, k(Z)\}\}$ such that
	\begin{itemize}
		\item  $\oa_i(\pi_i) = \oa_{i+1}(\pi_i)$ and $\oa_i(\pi_i^-) \ne \oa_{i+1}(\pi_i^-)$ for all $i \in \{1, \dots, k(Z)-1\}$. Here $\oa(t^-) := \lim_{s \to t^-} \oa (s)$. 
		\item For all $i$, we have $\oa_i^h \cap [\pi_{i-1}, \pi_i] = \emptyset$.
		\item $a_i \ne a_j$ for $i \ne j$.
		\item For all $i \ge 2$, the particle $a_i$ is susceptible at time $0$.
	\end{itemize}
	For $r \in [\pi_i, \pi_{i+1}]$, we call $Z(r) := \oa_i (r)$ the \textbf{location} of $Z$ at time $t$ and $a_i$ the \textbf{label} of $Z$ at time $r$. We say that the active chain $Z$ starts at $(s, Z(s))$ and ends at $(t, Z(t))$. We allow the possibility that $s = t$ in the definition of active chains. In this case, an active chain always consists of a single particle. For a potential starting location $u = (s, z) \in [0, \infty) \times \Z^2$, and $t > s$, let $\sC_u(t)$ be the set of active chains in $X$ on the interval  $[s, t]$ starting at $(s, z)$. We have the inequality
	\begin{equation}
	|{\bf I}_t| \le |\sC_{0, 0}(t)|,
	\end{equation}
	so to prove Proposition \ref{P:high-recovery-rate} it suffices to analyze $\sC_{0, 0}$.
	The basic premise behind the proof of Proposition \ref{P:high-recovery-rate} is that $|\sC_{0, 0}|$ behaves like a supermartingale when the particle density is low or the recovery rate is high. Since we would also like a version of Proposition \ref{P:high-recovery-rate} that proves \emph{local} infection recovery even when we are in the global survival regime, we will analyze $|\sC_{0, 0}|$ in conjunction with more general supermartingale-like processes.
	For $u = (s, z) \in [0, \infty) \times \Z^2$ and $A \sset \Z^2$, define the quantity
	$$
	I_{u, A}(t) = \sum_{Z \in \mathcal \sC_u(t)} \exp (-\nu d(Z(t), A)/8). 
	$$
	Note that $I_{0, \Z^2} = |\sC_{0, 0}|$.
	To analyze $I_{u, A}$, we start with two preliminary lemmas. The first establishes that $\E I_{u, A}$ exists and is continuous if we start from a finite initial particle configuration.
	
	\begin{lemma}
		\label{L:IuA-exist-cts} 
		Consider an SIR process started from a possibly random initial configuration with at most $n$ particles, which may be either infected or susceptible.
		For all $u = (z, s), A,$ the quantity $\E I_{u, A}(r)$ is finite and continuous for $r \in [s, \infty)$.
	\end{lemma}
	
	\begin{proof} Fix $t > 0$ and let $r \in [s, t]$. Letting $T(t)$ be the set of times in $[s, t]$ when a particle jumps, any active chain can be encoded as a map from $\{s\} \cup T(t) \to \{1, \dots, n\}$ recording the label of $Z$ at each time in $\{s\} \cup T(t)$. Therefore 
		\begin{equation}
		\label{E:IuA}
		I_{u, A}(r) \le (|T(t)| + 1)^{n}.
		\end{equation}
		The quantity $T(t)$ is a Poisson random variable, so $\E I_{u, A}(r) < \infty$. To establish continuity of $\E I_{u, A}$ at $r \in [s, t)$, observe that the probability that any particle either recovers or jumps in the interval $(r-h, r + h)$ tends to $0$ in probability as $h \to 0$, so almost surely $I_{u, A}(r + h) \to I_{u, A}(r)$ as $h \to 0$. Therefore by the dominated convergence theorem, which we can apply by \eqref{E:IuA}, $\E I_{u, A}$ is continuous at $r$.
	\end{proof}
	
	We will also need the following version of Gr\"onwall's inequality.
	
	\begin{lemma}
		\label{L:gronwall}
		Let $f:[0, t] \to \R$ be a continuous function with $f(0) > 0$. Suppose that for all $s \in (0, t)$, the upper right Dini derivative satisfies
		$$
		D^+f(s) := \limsup_{h \to 0^+} \frac{f(s + h) - f(s)}{h} \le \al f(s)
		$$
		for some $\al \in \R$. Then
		$
		f(s) \le e^{\al s} f(0)
		$
		for all $s \in [0, t]$.
	\end{lemma}
	
	\begin{proof}
		We just need two facts about Dini derivatives of continuous functions $f, g$: a product rule and a mean value theorem.
		$$
		D^+(fg) \le D^+(f)g + f D^+(g), \qquad \sup_{s \in (0, t)} D^+ f \ge \frac{f(t) - f(0)}t.
		$$
		These facts are easy to check and we leave their proofs to the reader. Now, define $g(s) = e^{-\al s} f(s)$ for $s \in [0, t]$. By the product rule and the assumption of the lemma we have
		\begin{equation*}
		\label{E:D^+gg}
		D^+ g(s) \le e^{-\al s} D^+ f(s) - \al f(s) e^{-\al s}\le 0
		\end{equation*}
		for $s \in (0, t)$. Therefore by the mean value theorem, $g(s) \le g(0)$ on $[0, t]$ and so $f(s) \le e^{\al s} f(0)$.
	\end{proof}
	
	We can now establish supermartingale-like behaviour for $I_{u, A}$.

	\begin{lemma}
		\label{L:Poisson-model-1} Consider an SIR process with recovery rate $\nu \le 1$, whose initial configuration of susceptible particles is stochastically dominated by a Poisson process of intensity $\nu/8$ and whose initial configuration of infected particles is at most countable.  Then for any $u = (0, z) \in [0, \infty) \times \Z^2$ and $A \sset \Z^2$, we have 
		$$
		\E I_{u, A}(s) \le \E I_{u, A}(0) \exp(-\nu s/2).
		$$
		If $A = \Z^2$, we can drop the condition that $\nu \le 1$.
	\end{lemma}
	
	\begin{proof}
		We write $I=I_{u, A}$ throughout the proof to simplify notation, let $P$ be the full set of initially susceptible particles, and let $P'$ be the full set of initially infected particles. We first consider the case when $|P| + |P'| \le n$ for some $n \in \N$. For $r < s \in [0, \infty)$ and a sequence of labels $Q$, let $\sC_u^{r, Q}(s)$ be the set of active chains on $[0, s]$ starting at $u$ whose label sequence up to time $r$ is given by $Q$.
		Define $I^{r,Q}(s)$ in the same way as $I(s)$ but with the set $\sC_u^{r, Q}(s)$ in place of $\sC_u(s)$. We first show that for any $r \in [0, \infty)$, almost surely
		\begin{equation}
		\label{E:hEt}
		\limsup_{h \to 0^+} \frac{1}{h} \E\lf(I^{r, Q}(r + h) - I^{r, Q}(r) \mid  I^{r, Q}(r) \rg) \le -\nu I^{r, Q}(r)/2.
		\end{equation}
		First, this bound is trivially true if $\sC_u^{r, Q}(r) = \emptyset$. Now, when $\sC_u^{r, Q}(r)$ is nonempty, we consider how $\sC_u^{r, Q}(r)$ can change by time $r + h$ for small $h$. We enumerate the possibilities. In this enumeration, let $q$ be the final label in $Q$.
		\begin{enumerate}[label=(\roman*)]
			\item $\oq^h \cap (r, r + h] = \emptyset$, the particle $q$ does not move in the interval $(r, r+ h]$, and no particles in $P$ jump onto the square $q(r)$ in the interval $(r, r+ h]$. In this case, every active chain in $\sC_u^{r, Q}(r)$ extends uniquely to an active chain in $\sC_u^{r, Q}(r+h)$ and $I^{r, Q}(r + h) = I^{r, Q}(r)$.
			\item $\oq^h \cap (r, r + h] = \emptyset$, the particle $q$ does not move in the interval $(r, r+ h]$, and exactly one particle $p \in P \smin Q$ jumps onto the square $\oq(r)$ in the interval $(r, r+h]$, does not heal in that interval, and does not jump again in that interval. In this case, $I^{r, Q}(r + h) = 2I^{r, Q}(r)$.
			\item $\oq^h \cap (r, r + h] = \emptyset$ and the particle $q$ jumps exactly once in the interval $(r, r+h]$ from the location $y = \oq(r)$ to a neighbouring location $y'$ with $|(P \smin Q) \cap \{y'\}| = N$. No other particles in $P$ jump onto or off of $y$ or $y'$ in the interval $(r, r+h]$ and no particles heal in the interval $(r, r + h]$. In this case, $I^{r, Q}(r + h)$ equals
			\begin{equation}
			\label{E:IrQ}
			I^{r, Q}(r)(N+1)\exp\lf(\frac{\nu(d(y, A) - d(y', A))}{8}\rg).
			\end{equation}
			Note that in the $A = \Z^2$ setting, this is simply $I^{r, Q}(r) (N+1)$. In the $A \ne \Z^2$ setting, using that $\nu \le 1$, this is bounded above by $I^{r, Q}(r) (N+1)(\nu/7 + 1)$.
			\item The particle $q$ does not jump in the interval $(r, r+h]$, no particles in $P \smin Q$ jump from another square onto $\oq(r)$ in that interval, and $q$ heals once in that interval. In this case $I^{r, Q}(r + h) = 0$.
			\item Other possibilities occur. In all these possibilities, if $E$ denotes the collection of all jump times and healing times of all particles in $P \cup P'$, then $|E \cap (r, r + h]| \ge 2$.
		\end{enumerate}
		To help bound the probability of (i)-(v) conditioned on $I^{r,Q}(t)$, observe that conditioning on $I^{r, Q}(r)$ gives no information about particles $p \in P \smin Q$. Moreover, the location of the particles in $P \smin Q$ at time $r$ is stochastically dominated by a Poisson process of intensity $\nu/8,$ since this holds at time $0$. With this, we can easily see that conditionally on $I^{r, Q}(r)$,
		\begin{equation}
		\label{E:prob-bds}
		\begin{split}
		&\P (i) = 1 - O(h), \qquad \P (ii) \le \nu h/8 + O(h^2), \qquad \P (iii) = h + O(h^2), \\
		\qquad &\P(iv) = \nu h + O(h^2), \qquad \P (v) = O(h^2).
		\end{split}
		\end{equation}
		To get the estimate \eqref{E:hEt}, we also need to understand the conditional expectation of \eqref{E:IrQ} given $I^{r, Q}(r)$ and that the event in (iii) holds. Up to an $O(h)$ term, conditioning on (iii) is the same as conditioning on the particle $q$ to jump in $(r, r + h]$. Under this conditioning, $N$ is stochastically dominated by a Poisson random variable of mean $\nu/8$. Therefore
		$$
		\E\lf(I^{r, Q}(r + h) - I^{r, Q}(r) \mid I^{r, Q}(r), (iii) \rg) \le 3\nu I^{r, Q}(r) /8 + O(h).
		$$
		Also, conditionally on (v) occurring and $I^{r, Q}(r)$, since $|P| \le n$ it is easy to check using ideas similar to the proof of Lemma \ref{L:IuA-exist-cts} that $\E(I^{r, Q}(r + h) - I^{r, Q}(r) \mid I^{r, Q}(r), (v) ) \le c I^{r, Q}(r)$ for some constant $c>0$. Putting this all together with the bounds in \eqref{E:prob-bds}, and the explicit changes to $I^{r, Q}$ on events (i), (ii), and (iv) gives \eqref{E:hEt}. Next, again using that $P$ is finite it is straightforward to check using the decomposition (i)-(v) above that the random variables
		$$
		\frac{1}{h} \E\lf(I^{r, Q}(r + h) - I^{r, Q}(r) \;|\; I^{r, Q}(r)  \rg)
		$$
		are uniformly integrable as we vary $h$ over $(0, 1)$. Therefore the bound \eqref{E:hEt} holds with the relevant random variables replaced with their expectations. Summing over the finitely many sequences $Q$ with labels in $P \cup P'$, we then get that for all $s > 0$, we have the right-hand derivative bound
		\begin{equation}
		\label{E:derivative}
		\limsup_{h \to 0^+} \frac{\E I(s+h) - \E I(s)}h \le -\frac{\nu\E I(s)}2.
		\end{equation}
		The first inequality in the lemma then follows from Lemma \ref{L:gronwall} applied to the function $\E I$. The continuity of $\E I$ required for the lemma holds by Lemma \ref{L:IuA-exist-cts}. 
		
		Now we extend this to the case when $P, P'$ are potentially infinite. Note that $P$ is always countable by the stochastic domination of the initial configuration by a Poisson process. Let $\{p_i : i \in \N\}, \{p_i' : i \in \N\}$ be enumerations of $P, P'$, and let $I^n$ be defined in the same way as $I$ but with the set $\sC_u(s)$ replaced by the set of active chains on $[0, s]$ started at $u$ that only use labels from the set $\{p_1, \dots, p_n, p_1', \dots, p_n'\}$. Then for all $s$, the sequence $I^n(s)$ is monotone increasing, and $I^n(s) \nearrow I(s)$ almost surely, so by the monotone convergence theorem, $\E I^n(s) \to \E I(s)$, yielding the bound in the lemma.
	\end{proof}
	
	\begin{proof}[Proof of Proposition \ref{P:high-recovery-rate}]
		By Markov's inequality and Lemma \ref{L:Poisson-model-1} we have
		$$
		\P({\bf I_t} \ne \emptyset) \le \E \mathbf{I}_t \le \E|\sC_{0,0}(t)| = \E |I_{0, \Z^2}(t)| \le \E |I_{0, \Z^2}(0)| e^{-\nu t/2}.
		$$
		Now, $\E |I_{0, \Z^2}(0)|$ is simply the expected number of particles that are initially at $0$, which is at most $1 + \nu/8$.
	\end{proof}
	
	\subsection{Local recovery in the global survival regime}
	\label{S:recovery-local}
	In the remainder of this section, we give a more technical version of Proposition \ref{P:high-recovery-rate} that will 
	allow us to later establish local recovery in the global survival regime. While we have included this section here since it builds on Proposition \ref{P:high-recovery-rate}, the reader may wish to leave it for now and return to it later when it is applied in Section \ref{S:herd}.
	
	We fix a large spatial parameter $M \in \N$, a set of particles $P$ performing independent continuous-time random walks from an initial configuration $P(0)$, and a collection of particles $Q$ with initial configuration $Q(0)$, where each particle $q \in Q$ follows a (potentially random) cadlag path $\oq : [0, \infty) \to \R$. We start with an arbitrary set of initially infected particles  $S \sset P \cup Q$, and consider the SIR model on the particles $P, Q$ with recovery rate $\nu \in (0, \infty)$.

	We write $P(t), Q(t)$ for the point processes $\{\op(t) : p \in P\},\{\oq(t) : q \in Q\}$. For the remainder of this section, we make the following assumption on the initial distribution $P(0)$:
	\begin{itemize}
		\item There exists $\de \in (0, 1)$ such that for all $z \in \Z^2$, we have 
		\begin{equation}
		\label{E:IC-assumption}
		P(0) \cap [-j M, j M]^2 \le \de j^2 M^2.
		\end{equation}
	\end{itemize}
	
	Now, let $\sF_{P, t}$ be the $\sig$-algebra generated by the trajectories of $P$ up to time $t$ and define the event
	\begin{equation}
	\label{E:AY-event}
	\mathcal D_Q = \{ \text{For all } t \in [0, 4M^2], \text{ we have } Q(t) \cap [-2M, 2M]^2 = \emptyset \}.
	\end{equation}
	In the remainder of this section, we prove the following proposition.
	\begin{proposition}
		\label{P:death-with-specifics}
		There exist absolute constants $c, C > 0$ such that for $C \de \le \nu \le 1$, there exists an $\sF_{P, t}$-measurable event $\mathcal B_P$ such that on $\mathcal D_Q \cap \mathcal B_P$, 
		$$
		\text{there are no infected particles in } [-M, M]^2 \text{ at any } t \in [M^2/4, 4M^2].
		$$
		Moreover, $\P [\cB_P] \ge 1 - C e^{-c \nu M}$.
	\end{proposition}

	Moving forward, we will apply Proposition \ref{P:death-with-specifics} in the following way. In our main SIR model, after the colouring process has passed through a region, then typically that region will have a low density of particles that have not been removed. Any remaining infections will quickly die out because of Proposition \ref{P:death-with-specifics}, applied at various scales $M$. The two sets of particles $P$ and $Q$ in Proposition \ref{P:death-with-specifics} are distinguished in order to separate local and global effects.
	In applications, the set of trajectories $P$ will be measurable given some local $\sig$-algebra, allowing us to establish independence of spatially separated infection recovery events. The contributions from trajectories in $Q$ will be handled with a global union bound. 
	
	We prove Proposition \ref{P:death-with-specifics} by appealing to Lemma \ref{L:Poisson-model-1}. The first step for doing this is to show that in a certain time window, the configuration of $P$-particles near $0$ is stochastically dominated by a Poisson process.
	\begin{lemma}
		\label{L:poisson-domination}
		Under assumption \eqref{E:IC-assumption}, for $M$ sufficiently large and $t \in [M^2/8, 4M^2]$, the configuration $P(t) \cap [-3M, 3M]^2$ is stochastically dominated by a Poisson process $\Pi$ of intensity $C\de$.
	\end{lemma}

	To prove Lemma \ref{L:poisson-domination}, we use the following lemma for establishing stochastic domination. 
	
	\begin{lemma}
		\label{L:russo-lemma}
		Let $P$ be a point process on a finite set $F$. Suppose that for any $x \in F$ and any point configuration $f$ on $F \smin \{x\}$ for which the event $P|_{F \smin \{x\}} = f$ has nonzero probability, the conditional distribution
		\begin{equation}
		\P(|P \cap \{x\}| \in \cdot \; | \; P|_{F \smin \{x\}} = f)
		\end{equation}
		is stochastically dominated by a Poisson random variable of mean $\ga$. Then $P$ is stochastically dominated by a Poisson process with intensity $\ga$.
	\end{lemma}
	
	Results of this form are well-known. For example, Lemma \ref{L:russo-lemma} was shown in \cite{russo1982approximate}, Lemma 1 for Bernoulli random variables and Bernoulli processes in place of Poisson random variables and Poisson processes. Russo's proof works verbatim in the context above, and so we omit it.
	
	We will also use a standard random walk estimate. This estimate will be used throughout the paper so we record it here as a lemma. We leave the proof as an exercise for the reader using Azuma's inequality and basic large deviations.
	
	\begin{lemma}
		\label{L:rw-estimate}
		Let $X:[0, \infty) \to \Z^2$ be a continuous time random walk. Then for all $u \in \Z^2$ and $m > 0$, we have
		$$
		\P(X(t) = u) \le \frac{C}{t} \exp\lf(- \frac{c|u|^2}{|u| + t}\rg), \;\; \P(\max_{0 \le r \le t} |X(t)| \ge m) \le C\exp\lf(- \frac{c m^2}{m + t}\rg).
		$$
	\end{lemma}
	
	\begin{proof}[Proof of Lemma \ref{L:poisson-domination}] By Lemma \ref{L:russo-lemma}, it is enough to show that for any $t \in [M^2/8, 4M^2], z \in [-3M, 3M]^2$ and any finite point configuration $f$ on ${[-3M, 3M]^2 \smin \{z\}}$, the conditional distribution
		\begin{equation}
		\label{E:Xtt}
		\P(|P(t) \cap \{z\}| \in \cdot \mid  P(t)|_{[-3M, 3M]^2 \smin \{z\}} = f)
		\end{equation}
		is stochastically dominated by a mean $C\de$ Poisson distribution.
		Equation \eqref{E:Xtt} follows from the stronger claim that for any set $P' \sset P$ and any function $f:P' \to [-3M, 3M]^2 \smin \{z\}$, the conditional distribution
		\begin{equation*}
		\label{E:Xttt}
		\P\lf(|P(t) \cap \{z\}| \in \cdot \; \Big| \; \op(t) = f(p) \;\forall p \in P', \op(t) \notin [-3M, 3M]^2 \smin \{z\} \;\forall p \in P \smin P' \rg)
		\end{equation*}
		is stochastically dominated by a mean $C\de$ Poisson distribution. With the above conditioning, we have
		\begin{equation}
		\label{E:Xtcapz}
		|P(t) \cap \{z\}| = \sum_{p \in P \smin P'} \mathbf{1}(\op(t) = z).
		\end{equation}
		This is a sum of independent Bernoulli random variables of mean 
		$$
		\mu_p := \P(\op(t) = z \mid \op(t) \notin [-3M, 3M]^2 \smin \{z\}).
		$$ 
		Now, for all $p \in P, t \ge M^2/8$, we have
		$
		\P(\op(t) \notin [-3M, 3M]^2) \ge c.
		$
		Combining this with the bound in Lemma \ref{L:rw-estimate} gives that 
		$$
		\mu_p \le  CM^{-2} e^{-c|\op(0)|/M}.
		$$
		Also, a Bernoulli random variable of mean $\mu$ is stochastically dominated by a Poisson random variable of mean $-\log(1 - \mu)$. Therefore since the sum of independent Poisson random variables is Poisson, \eqref{E:Xtcapz} is stochastically dominated by a Poisson random variable of mean
		$$
		\sum_{p \in P \smin P'} - \log\lf(1 - CM^{-2} e^{-c|p(0)|/M}\rg).
		$$
		For $M$ sufficiently large, every term in the above sum is bounded above by $2CM^{-2} e^{-c|p(0)|/M}$, and so by \eqref{E:IC-assumption}, after increasing $C$ the whole sum is bounded above by $C \de$.
	\end{proof}
	
	We also record a bound that deals with particles moving an unusually large amount in a small time interval. For this next lemma and in the remainder of the section, we  define the sets
		\begin{align*}
		P' &= \{p \in P : |\op(0)| \le M^2\}, \\
		\qquad P'' &= \{p \in P' : \max_{r \in [0, 4 M^2], s \in [0, 2M]} |\op(s + r) - \op(r)| \le M \},\\
			P''' &= \{p \in P : \min_{0 \le r \le 4 M^2} |\op(r)| \le 4 M \}.
		\end{align*}
	
	\begin{lemma}
		\label{L:side-bry}
		Let $\cH$ be the event  $\{P' \ne P''\} \cup \{P''' \not \sset P'\}$.
		Then \[\P[ \cH] \le C e^{- c M}.\] 
	\end{lemma}
	
	\begin{proof}
		Let $P_j = \{p \in P : |\op(0)| \in [jM, j(M+1)]\}$. We then have the union bounds
		\begin{align*}
		\P(P'''  \not\subset P') &\le \sum_{j=M}^\infty \sum_{p \in P_j} \P(\max_{0 \le r \le 4M^2}  |\op(0) - \op(r)| \ge jM/2), \\
		\P( P' \ne P'') &\le \sum_{p \in P'} \sum_{r \in [0, 4M^2] \cap M \Z} \P(\max_{s \in [0, 4M]}  |\op(s + r) - \op(r)| \ge M).
		\end{align*}
		By the right hand sides above can be bounded by $C e^{- c M}$ by Lemma \ref{L:rw-estimate} and \eqref{E:IC-assumption}.
	\end{proof}

	We will want to take a union bound over starting and ending times of active chains so it will be useful to discretize time on a fine mesh of size $e^{-M\nu/100}$.  Let $A$ be the set of times at which a particle in $P'$ takes a step.  Let
	\[
	\cJ=\{\forall 0\leq j\leq 5M^2 e^{M\nu/100}: |A\cap [je^{-M\nu/100},(j+1)e^{-M\nu/100}]| \leq 1\}
	\]
    which says that each interval of time of length $e^{-M\nu/100}$ up to time $5M^2$ has at most one particle step in $P'$.
	
	\begin{lemma}
		\label{l:time.discr}
		We have that,
		\[
		\P[\cJ] \geq 1- C e^{- c M\nu}.
		\]
	\end{lemma}
	\begin{proof}
	    By~\eqref{E:IC-assumption} there are at most $M^2$ particles in $P'$ and so $A$ is stochastically dominated by a rate $M^2$ Poisson process.  Then
	    \[
	    \P[|A\cap [je^{-M\nu/100},(j+1)e^{-M\nu/100}]| > 1] \leq C (M^2 e^{-M\nu/100})^2
	    \]
	    and the result follows by a union bound over $j$.
	\end{proof}
	
	The next lemma will help relate the event in Proposition \ref{P:death-with-specifics} to active chains.
	
	\begin{lemma}
		\label{L:reduce-to-singles}
		Let $T \le M, t \in [M^2/8, 4M^2 - T]$ and let $	\cE_{t, T}$ be the event where 
		$$
	 \text{there exists an infected particle in } [-M, M]^2 \text{ at some time in } [t + T/2, t+ T].
		$$
		Then $\cE_{t, T} \sset \cI_{t, T} \cup\cH \cup \cD_Q^c$, where $\cH$ is as in Lemma \ref{L:side-bry} and $\cI_{t, T}$ is the event where either
		\begin{enumerate}[label=(\roman*)]
			\item there is an active chain $Z$ on $[t, s]$ for some $s \in [t + T/2, t+T]$ with $Z(r) \in [-2M, 2M]^2$ for all $r \in [t, s]$ and final location $Z(s) \in [-M, M]^2$, which only uses particles with labels in $P''$, or
			\item there is an active chain on $[s', s]$ for some $s' \in [t, t + T], s \in [t + T/2, t + T]$ with $Z(r) \in [-2M, 2M]^2$ for all $r \in [t, s]$, starting location $z \in \del [-2M, 2M]^2$, and final location in $[-M, M]^2$ which only uses particles with labels in $P''$.
		\end{enumerate} 
	\end{lemma}
	
	\begin{proof}
		Suppose that there is an infected particle $p$ with  $\op(s) \in [-M, M]^2$ for some $s \in [t+T/2, t+T]$, and that we are working on the event $\cH^c \cap \cD_Q$. On this event, the only particles in $P \cup Q$ that enter $[-2M, 2M]^2$ in the interval $[t, t + T]$ are particles with labels in $P'$, which equals $P''$, since we are working on $\cH^c$.
		
		We can trace the particle $p$ backwards in time until the moment $s_0$ given by the maximum of 
		\begin{itemize}[nosep]
			\item $t$,
			\item the time when $p$ most recently entered the set $[-2M, 2M]^2$,
			\item the time when $p$ became infected by a particle $q$.
		\end{itemize}
		In the former two cases, the particle $p$ and its trajectory form an active chain on $[s_0, s]$ either with starting location in $\del [-2M, 2M]^2$, implying (ii) above, or with starting location in $[-2M, 2M]^2$ and starting time $t$, implying (i). In the latter case, the infection location was contained in $[-2M, 2M]^2$, so $q \in P'$ as well. Moreover
		$$
		\op(s_0) = \oq(s_0), \qquad \op(s_0^-) \ne \oq(s_0^-).
		$$
		We can then continue tracing the particle $q$ back until the time and location when it either became infected or left the box $[-2M, 2M]^2$, and proceed in this way until we reach time $t$, or a location outside of $[-2M, 2M]^2$. This process terminates a.s. since $P'$ is a.s. finite by Lemma \ref{L:poisson-domination}. Either way, we are left with an active chain that implies the event $\cI_{t, T}$.
	\end{proof}

	\begin{lemma}
		\label{L:F-bound}
		There exist constants $c, C > 0$ such that as long as $C \de < \nu/8 < 1$, for any $T \le M, t \in [M^2/8, 4M - T]$ we have
		$$
		\P [\cI_{t, T} \cap \cJ] \le C e^{- c\nu T},
		$$
		where $\cI_{t, T}$ is as in Lemma \ref{L:reduce-to-singles}.
	\end{lemma}
	
	\begin{proof}
		First, we may assume that $M$ is large enough so that Lemma \ref{L:poisson-domination} holds, since otherwise we can guarantee the bound by taking $C$ large enough and $c$ small enough.  On the event $\cJ$ the active chain given in the event $\cI_{t, T}$ takes at most one step in an interval of length $e^{-M\nu/100}$ and so by rounding the starting and ending points we we have the union bound
		\begin{equation}
		\label{E:big-union}
		\P [\cI_{t, T}] \le \sum_{\substack{s' \in e^{-M\nu/100}\Z \cap [t, t + T+1], \\s \in e^{-M\nu/100}\Z \cap [t + T/2, t + T + 1], \\
				z: d(z, \del [-2M, 2M]^2)\leq 1}} \P [\cG_{z, s', s}] + \sum_{\substack{s \in e^{-M\nu/100}\Z \cap [t + T/2, t+ T + 1], \\
				z \in [-2M, 2M]^2}} \P [\cG_{z, t, s}],
		\end{equation}
		where $\cG_{z, r, s}$ is the event where there is an active chain using only particles in $P'$ starting at $(z, r)$  and ending at a time $s$ at a location in $[-(M+1), M+1]^2$. Note that there are $O(M^4 e^{M\nu/50})$ terms in \eqref{E:big-union} so to complete the proof we bound $\P [\cG_{z, r, s}]$. 
		In fact, we will show that whenever $z \in [-2M, 2M]^2$ and $s - r \le \frac32 M$ we have
		\begin{equation}
		\label{E:cFzrs}
		\P [\cG_{z, r, s}] \le C \exp\lf( - \nu\lf(\frac{d(z, [-(M+1), M+1]^2)}{8} + \frac{s-r}{2}\rg)\rg).
		\end{equation}
		To prove \eqref{E:cFzrs}, first observe that if $|r -s| \le \frac32 M$ then any active chain $Z$ on $[r, s]$ with $Z(r') \in [-2M, 2M]^2$ for all $r'$ that only uses particles in $P''$ only uses particles in the set 
		$$
		P^* = \{p \in P' : \op(r) \in [-3M, 3M]^2\}.
		$$
		Therefore $\cG_{z, r, s} \sset \cG^*_{z, r, s}$, where $\cG^*_{z, r, s}$ is the event where there exists an active chain $Z$ from $(z, r)$ to $(s, y')$ for some $y' \in [-(M+1), (M+1)]^2$ with labels contained in $P^*$.
		
		Now, by Lemma \ref{L:poisson-domination}, the process $P^*$ is stochastically dominated by a Poisson process of intensity $C \de$ and the trajectories from $P^*$-particles on the interval $[r, \infty)$ are independent random walks. 
		After a time shift by $r$, we are then in the setting of Lemma \ref{L:Poisson-model-1}.
		By Markov's inequality and Lemma \ref{L:Poisson-model-1}, we have
		\begin{align}
		\nonumber
		\P [\cG_{z, r, s}] \le \P [\cG^*_{z, r, s}] &\le \E I_{(z, r), [-(M+1), M+1]^2}(s) \le \E I_{(z, r), [-(M+1), M+1]^2}(r) e^{-\nu (s-r)/2}.
		\end{align}
		The inequality \eqref{E:cFzrs} then follows since 
		$$
		I_{(z, r), [-(M+1), M+1]^2}(r)=
		N\exp\lf(d(z, [-(M+1), M+1]^2)/8\rg),
		$$
		where $N = |\{p \in P' : \op(r) = z\}|$ is stochastically dominated by a Poisson random variable of mean $C \de$ and the result follows by the union bound in~\eqref{E:big-union}.
	\end{proof}
	
	Proposition \ref{P:death-with-specifics} now follows by a quick union bound. 
	
	\begin{proof}[Proof of Proposition \ref{P:death-with-specifics}]
		We let 
		$$
		\cB_X^c = \cH \cup \bigcup_{t \in \Z \cap [M^2/8, 4M^2]} \cI_{t, M/4}.
		$$
		On the event $\cD_Q \cap \cB_X$, Lemma \ref{L:reduce-to-singles} implies that there are no infected particles in the time-space box $[M^2/4, 4M^2] \times [-M, M]^2$. The estimate on $\P \cB_X^c$ then follows from a union bound, Lemma \ref{L:side-bry}, and Lemma \ref{L:F-bound}.
	\end{proof}
	
\section{Sidoravicius--Stauffer percolation}
\label{S:SSP}
	
	Proving the existence of a survival regime for the SIR process is much more involved than showing the existence of a recovery regime. To do this, we will couple the SIR model to a competing growth process that we call Sidoravicius--Stauffer percolation (SSP). The original variant of SSP was first studied in~\cite{sidoravicius2019multi}. We adapted the framework from~\cite{sidoravicius2019multi} in~\cite{dauvergne2021spread} in order to understand infection spread in random walks. Many of the results of~\cite{dauvergne2021spread} will be required here, though we do not need the same level of generality used in that paper. In this section, we introduce SSP and gather the necessary results.
	
	Informally, in SSP on $\Z^2$, a set of red vertices grows outwards from the origin according to a set of edge weights, similarly to first passage percolation. Within $\Z^2$, there is a collection of blue seeds $\mathfrak{B}_*$ which cannot be invaded by the red growth process. Whenever the red process attempts to invade a blue seed, a competing blue growth process is activated. Once activated, the blue process will grow outward from a blue seed at a slow, constant speed. The red and blue processes can only invade squares that are not already part of one of the two coloured processes. 
	
	Within this framework, we also allow for squares adjacent to the blue process to become activated at arbitrary times. When these squares are activated, they will turn red and join the red process unless they belong to the set of blue seeds, in which case they will turn blue.
	
	\subsection{Constructing the process}
	
	We will define a pair of coupled growth processes $\fR, \fB:[0, \infty) \to \{S: S \sset \Z^2\}$. The growth of these processes will be governed according to weights on the set of directed edges
	$$
	E := \{(u, v) \in \Z^2 : |u - v|
	= 1 \}.
	$$ 
	We will need the following data to define $\fR, \fB$.
	\begin{itemize}
		\item A constant $\ka > 0$, referred to as the \textbf{parameter} of the process. In \cite{dauvergne2021spread}, we let the parameter $\ka$ be a function from $E \to (0, \infty)$. In the current paper, we are only concerned with the case where $\ka$ is constant.
		\item A collection of blue seeds $\fB_* \sset \Z^2$.
		\item Edge weight functions $X_\fR:E \to [0, 1]$ and $X_\fB:E \to [0, \ka]$. We will think of $X_\fR, X_\fB$ as clocks governing the growth of the red and blue processes. Edge weights $X_\fB(e)= \ka$ will encourage the spread of the blue process, whereas other edge weights will encourage the spread of the red process.
		\item We have allowed for the caveat that clocks may have a value of $0$. This corresponds to instantaneous invasion. To ensure that the process is still well-defined even when instantaneous invasion is allowed, we require that there are no directed cycles $(u_0, u_1, \dots u_n = u_0)$ with
		$$
		X_\fR(u_{i-1}, u_{i}) \wedge X_\fB(u_{i-1}, u_{i}) = 0
		$$
		for all $i = 1, \dots, n$. 
	\end{itemize}
	We define $\fR$ and $\fB$ by recording a time $T(u)$ when each vertex $u$ gets added to the process, and a colour $C(u)$ for that vertex. We initialize the process by setting $T(0) = 0$, and setting $C(0) = \fR$ if $0 \notin \fB_*$ and $C(0) = \fB$ if $0 \in \fB_*$. The rules for assigning the other times and colours are as follows.

	For every edge $(u, v) \in E$, at time $T(u) + X_{C(u)}(u, v)$, the edge $(u, v)$ will ring, and the process will update accordingly. If $T(v) < T(u) + X_{C(u)}(u, v)$, then nothing happens. This corresponds to the case when $v$ has already been added to one of the processes by this time. Otherwise, we set $T(v) = T(u) + X_{C(u)}(u, v)$, and colour $v$ according the following four rules.
	\begin{enumerate}
		\item If $v \in \fB_*$, set $C(v) = \fB$.
		\item If $C(u) = \fR$ and $v \notin \fB_*$, then set $C(v) = \fR$.
		\item If $C(u) = \fB, v \notin \fB_*$, and $X_{\fB}(u, v) < \ka$, then set $C(v) = \fR$.
		\item If $C(u) = \fB, v \notin \fB_*$ and $X_{\fB}(u, v) = \ka$, then set $C(v) = \fB$.
	\end{enumerate}
	We must specify what happens if $T(u) + X_{C(u)}(u, v) = T(u') + X_{C(u')}(u', v)$ for two different vertices $u, u'$. In this case, if the colours assigned to $v$ via the two edges $(u, v)$ and $(u, v')$ are different, we set $C(v) = \fB$. For $t \in [0, \infty)$, we set
	\begin{equation}
	\label{E:BR-def}
	\begin{split}
	\fR(t) &= \{u \in \Z^2 : T(u) \le t, C(u) = \fR\} \quad \text{and} \\
	\quad \fB(t) &= \{u \in \Z^2 : T(u) \le t, C(u) = \fB\}. 
	\end{split}
	\end{equation}
	When working with SSP, we will always assume that for any $t$, only finitely many squares $u$ satisfy $T(u)  < t$.
	Under this \textbf{finite speed} assumption, the colouring process is well-defined. From now on, all processes we consider have finite speed.
	
	In~\cite{dauvergne2021spread}, we showed that as long as $\ka$ is sufficiently large, we can find a well-controlled set $\mathfrak{C}$ depending only on the initial configuration of blue seeds $\fB_*$, such that under certain conditions on $\fB_*$, we have $\fB(\infty) \sset \mathfrak{C}$. 
	
The set $\mathfrak{C}$ is built via a multi-scale construction. 
Let $r_0 = 1 \le r_1 < r_2 < \dots$ be a sequence of integer scales satisfying
\begin{equation}
\label{E:gammabound}
\ga := \prod_{i=0}^\infty \left(1 + \frac{10^{12} r_i}{r_{i+1}} \right) < 2,
\end{equation}
and such that $r_{i+1} = (i+1)^2 r_i$ for all large enough $i$. For closed disks and annuli, we write
$$
D(x, r) = \{y \in \Z^2: |y - x| \le r \}, \qquad A(x, r, R) = \{y \in \Z^2 : r \le |y - x| \le R\}.
$$
Here and throughout the paper, we use the union-of-disks notation
$$
D(A, r) = \{y \in \Z^2: d(y, A) \le r \}= \bigcup_{x \in A} D(x, r).
$$
Now, for a set of blue seeds $\fB_* \sset \Z^2$, define
$$
A_1(\fB_*) := \{x \in \fB_*: \fB_* \cap D(x, r_1/3) = \{x\} \},
$$
We then recursively define $\fB_{*, k} = \fB_* \smin \cup_{i=1}^{k-1} A_i(\fB_*)$, and let 
$$
A_k(\fB_*) := \{x \in \fB_{*, k}: \fB_{*, k} \cap A(x, r_{k-1},  r_k/3) = \emptyset \}.
$$
Finally, for a set $A \sset \Z^2$, we let $[A]$ be the union of $A$ and all bounded components of $A^c$.
\begin{theorem}
	\label{T:BssetI}
	Consider an SSP $(\fR, \fB)$ on $\Z^2$ initiated from a set of blue seeds $\fB_*$ with parameter $\ka > 4000$. Suppose that
	\begin{enumerate}
		\item $\fB_* = \bigcup_{k=1}^\infty A_k(\fB_*)$, and
		\item $0 \notin [\fD]$, where $\fD := \bigcup_{k=1}^\infty D(A_k(\fB_*), r_k/100)$.
	\end{enumerate} 
	Then $\fB(\infty) \sset [\fC]$, where $\fC :=\bigcup_{k=1}^\infty D(A_k(\fB_*), 100 r_{k-1})$.
\end{theorem}

Theorem \ref{T:BssetI} is a special case of Theorem 2.1 from~\cite{dauvergne2021spread} when $V = \Z^2$ (see Remark 2.2 from that paper). 

In our applications of Theorem \ref{T:BssetI}, the blue seed set $\fB_*$ will always be stochastically dominated by a low-density Bernoulli process. The next series of lemmas from~\cite{dauvergne2021spread} bound the behaviour of $\fC, \fD$ in this setting. We start with a simple monotonicity lemma.

\begin{lemma}[Lemma 2.11, \cite{dauvergne2021spread}, special case when $V = \Z^2$]
	\label{L:dom-seeds}
	Let $\fB_*' \sset \fB_* \sset \Z^2$ be two collections of blue seeds. Then for every $k \ge 1$, we have $\fB_{*, k}' \sset \fB_{*, k}$. Moreover, if the set $\fB_*$ satisfies the assumptions
	of Theorem \ref{T:BssetI}, then $\fB_*'$ also satisfies these assumptions, and with $\fC, \fD, \fC', \fD'$ as in that theorem grown from the sets $\fB_*, \fB_*'$, we have
	\begin{equation}
	\label{E:IbIB}
	\fC' \sset \fC \qquad \text{ and } \qquad \fD' \sset \fD.
	\end{equation}
\end{lemma}

	Note that the monotonicity in Lemma \ref{L:dom-seeds} does not necessarily hold at the level of the blue processes $\fB, \fB'$ generated from $\fB_*, \fB'_*$.
	
	\begin{lemma} [Lemma 2.12, \cite{dauvergne2021spread}]
		\label{L:Ak-estimate} 
		Consider any collection of blue seeds $\fB_{*}$. 
		Then for any set $X \sset \Z^2$, the set $X \cap \fB_{*, k}$ is a function of $\fB \cap D(X, r_{k-1}/2)$. 
		
		Moreover, suppose that $\fB_{*}$ is stochastically dominated by an i.i.d.\ Bernoulli process with mean $p> 0$.  Then for any $x \in \Z^2$,
		$$
		\mathbb P(x \in \fB_{*, k}) \le (C p)^{2^{k-1}}.
		$$
	\end{lemma}

Lemma \ref{L:Ak-estimate} naturally yields bounds on the engulfing sets $\fC, \fD$, given by the next lemma. To state this lemma, for $k \in \N$ we also define
\begin{equation*}
\fD_k := \bigcup_{i=1}^k D(A_i(\fB_*), r_i/100) \qquad \text{ and } \qquad \fC_k := \bigcup_{i=1}^k D(A_i(\fB_*), 100 r_{i-1}).
\end{equation*}
	\begin{lemma}[Lemma 2.13, \cite{dauvergne2021spread}]
		\label{L:largest-scale} Assume that $\fB_{*}$ is stochastically dominated by an i.i.d.\ Bernoulli process with parameter $p> 0$.
		Fix a ball $D(x, r)$ for some $r \in \N$, and let $M(x, r)$ be the diameter of the largest component of $[\fD]$ that intersects $D(x, r)$. Then
		\begin{equation}
		\label{E:Mrrk}
		\P\lf( M(x, r) > r_k \rg) \le r^2 (Cp)^{2^{k-1}} \qquad \text{ and } \qquad \P\lf( M(x, r) > 0 \rg) \le C r^2 p.
		\end{equation}
		In particular, for every $n_0 \in \N$ with $n_0 \ge 3$ we have
		\begin{equation}
		\label{E:Mn-BC}
		\P\lf( M(x, n) \ge n/2 \text{ for some } n \ge n_0\rg) \le  \exp\lf(\log(Cp) \exp (\log n/(C\log \log n))\rg).
		\end{equation}
		Moreover, for all $k \ge 1$,
		\begin{equation}
		\label{E:fDk-estimate}
		\P([\fD] \cap D(x, r) \ne [\fD_k] \cap D(x, r)) \le  r^2 (Cp)^{2^{k-1}},
		\end{equation}
		and for all sufficiently small $p$, we have
		\begin{equation}
		\label{E:PDxr}
		\P\lf(|[\fD] \cap D(x, r)| \ge \frac{1}{3}|D(x, r)| \rg) \le \exp\lf(\log(Cp) \exp (\log r/(C\log \log r))\rg).
		\end{equation}
		The same bounds hold with $\fC, \fC_k$ in place of $\fD, \fD_k$.
	\end{lemma}
	
The next theorem is a version of Theorem \ref{T:BssetI} for random SSPs.

\begin{theorem}[Theorem 2.14, \cite{dauvergne2021spread}]
	\label{T:random-red-blue}
	Consider a random SSP $(\fR, \fB)$ on $\Z^2$ driven by potentially random clocks $X_\fR, X_\fB$, a collection of blue seeds $\fB_*$ and parameter $\ka > 4000$.
	Suppose additionally that $\fB_*$ is stochastically dominated by an i.i.d.\ Bernoulli process with parameter $p > 0$ sufficiently small.		
	Let $\fC$ be defined from $\fB_*$ as in Theorem \ref{T:BssetI}. Then the event where
	$$
	\fB_* = \bigcup_{i=1}^\infty A_i(\fB_*) \quad \text{ and } \quad [\fC] \text{ has only bounded components}
	$$
	is almost sure. Moreover, the probability of the event $\cE$ given by
	\begin{equation}
	\label{E:CD}
	\begin{split}
	\Big\{\fB_* = \bigcup_{i=1}^\infty A_i(\fB_*), 0 \in [\fD]^c, [\fC]^c &\sset \fR(\infty), [\fC] \text{ has only bounded components}\Big\}, 
	\end{split}
	\end{equation}
	is at least $1 - Cp$. Finally, on the event $\cE$, we have $D(0, t/(2\ka)) \sset [\fR(t)]$ for all large enough $t$. More precisely, for any $t_0 > 0$ we have
	\begin{equation}
	\label{E:D0kk}
	\begin{split}
	\P(\{D(0, t/(2\ka)) &\sset [\fR(t)] \text{ for all } t \ge t_0\} \mid \cE) \\
	&\ge 1 - \exp\lf(\log(Cp) \exp (\log (t_0/\ka)/(C\log \log (t_0 /\ka))\rg).
	\end{split}
	\end{equation}
\end{theorem}
\section{SIR survival}

In this section we give a coupling of the SIR process $X_t$ with initial Poisson density $\mu$
 with an SSP satisfying the conditions of Theorem \ref{T:random-red-blue}. In this coupling the red sites in the SSP will correspond to spatial blocks in $\Z^2$ where the SIR infection spreads efficiently. The initial configuration of blue seeds will be dominated by a Bernoulli process with parameter $p = p(\nu)$ where $p \to 0$ as $\nu \to 0$, allowing us to conclude Theorem \ref{T:main-1} by applying Theorem \ref{T:random-red-blue}. Throughout this section and in the remainder of the paper the density $\mu$ is fixed and all constants $c, C, \dots$ are allowed to depend on $\mu$. 

\label{S:linear}
\subsection{A block construction for the SIR process}
\label{S:SI-colouring}

Fix a side length $L = 2^k$ for some large $k \in \N$.  We choose $L$ so that 
\begin{equation}
\label{E:Lnu-relationship}
\nu^{-1/6} \leq L \leq 2 \nu^{-1/6}
\end{equation}
and subdivide $\Z^2$ into a collection of  dyadic blocks
\begin{equation}
\label{E:BL-construct}
\sB=\sB_L:=\Big\{zL +\big\{-L/2, \ldots ,L/2-1\big\}^2:z\in \Z^2\Big\}.
\end{equation}
This induces a natural map $f = f_L:\Z^2 \to \sB_L$. For $B, B' \in \sB$, we also define the \emph{block distance}
$$
d_\sB(B, B') = |f^{-1}(B) - f^{-1}(B')|
$$
and for $B \ne B'$ note the inequalities
\begin{equation}
\label{E:block-distance}
d(B, B') < L d_\sB(B, B'), \qquad d_\sB(B, B') < 4 \ceil{d(B, B')/L}.
\end{equation}
Both $d$ and $d_\sB$ will be useful measurements in the paper for analyzing distance. 
For the arguments in the remainder of the paper, $L$ is the more natural parameter than $\nu$ and so we work exclusively with a sufficiently large $L$ rather than sufficiently small $\nu$, invoking the relationship \eqref{E:Lnu-relationship} when necessary.

Given a side length $L$, for some small $\alpha\in(0,1)$, let $\xi = \xi_L = \alpha L^2$. This will be the minimum speed at which the colouring of blocks spreads.  We will now define a \textbf{colouring process} for blocks $B \in \sB_L$ from the SIR process.

\begin{definition}
	\label{D:SI-colouring}
	We will let $\tau_B$ denote the (stopping) time when a block $B$ is first coloured. For a vertex $u\in B$ we will use $\tau_u$ as shorthand for $\tau_B$. A block $B$ becomes coloured at time $t$ if one of the following events happens:
	\begin{enumerate}[label=(\alph*)]
		\item A neighbouring block $B'$ (i.e. $d_\sB(B, B') = 1$) was coloured at time $\tau_{B'} = t - \xi$.
		\item An infected particle enters $B$ for the first time at time $t$.
		\item A block $B'$ becomes coloured according to rule (b) at time~$t$ by an infected particle entering at location $x \in B'$, and $\inf \{\|x - u \|_\infty : u \in B \} \le \al L$.
	\end{enumerate}
\end{definition}
In case (c) we call this a multi-colouring, where up to three blocks may be coloured simultaneously. 
In case (b) we call the infected particle that entered $B$ the \textbf{ignition particle} of $B$.  In case (c) the ignition particle is the ignition particle for block $B'$. In cases (b) and (c) we say that the block $B$ is \textbf{ignited} at the location $x$ where an infected particle first entered $B$ (for case (b)) or first entered the relevant neighbour of $B$ (for case (c)). All ignition locations are within $L^\infty$-distance $\al L$ of $B$.

We say that a particle $a$ is \textbf{coloured} at the first time it enters a coloured block, write $B(a)$ for the block in which $a$ was coloured, and write $T(a)$ for the time when $a$ was coloured. This can happen either because the block became coloured or because the particle entered a coloured block. We let 
$
X_t^* 
$ denote the process of all particles up to time $t$ that satisfy $T(a) \le t$, and let $\sF_t^*$ denote the $\sigma$-algebra generated by $X_t^*$. We let $\iota_a$ be the stopping time when particle $a$ becomes infected.

A key element of our analysis is to regard the collection of random walks as given by a Poisson process.  We let $\fW$ be the space of cadlag sample paths $w(t):\R\to\Z^2$ such that 
\begin{equation}\label{eq:sublinearGrowth}
\frac1{t}|w(t)|\to 0
\end{equation}
as $|t| \to \infty$.  Elements of $\fW$ will represent the trajectories of particles.  To represent healing times of particles we let we let $\fQ$  be the set of simple point measures on $\R$ and let $\fW^+=\fW \times \fQ$. An element of $\fW^+$ will represent the trajectory of a particle together with the set of times at which it receives a healing event.
Let $\sW^+_u$ denote the measure on $\fW^+$ given by the pair of 
\begin{enumerate}[nosep, label=(\alph*)]
	\item a continuous time random walk on $\Z^2$ over all time $t\in \R$ that is at $u\in \Z^2$ at time $0$ and
	\item an independent rate $\nu$ Poisson process on $\R$.
\end{enumerate}
Note that a simple random walk satisfies~\eqref{eq:sublinearGrowth} almost surely.  Furthermore, let $\sW^+ = \sum_{u\in \Z^2} \sW^+_u$.  We let $\sP$ denote a Poisson process on $\fW^+$ with intensity measure $\mu\sW^+$.  If we remove the initial infected particle from the origin, we can interpret all remaining particles and their trajectories as being given by a sample from $\sP$.  

In order to obtain spatial independence of various events we will give an alternative construction of the SIR process with the same law based on a collection of Poisson processes $\sP_B, B \in \sB$ which are IID and equal in distribution to $\sP$. We call a particle $a$ in $\sP_B$ \textbf{simple}, let $a(t)$ denote its trajectory, and let $a^h \subset \R$ its set of healing times.

Let $\sM_B$ denote the $\sig$-algebra generated by
\begin{enumerate}[label=(\roman*)]
	\item the independent Poisson process $\sP_B$,
	\item a pair $(W_{B,\operatorname{ig}}, W_{B,\operatorname{ig}}^h)$ sampled independently from $\sW^+_0$, which will encode the trajectory and healing process of an ignition particle. 
\end{enumerate} 
Note that the $(\sP_B,W_{B,\operatorname{ig}}, W_{B,\operatorname{ig}}^h), B \in \sB$ are IID.

We will build the SIR process $X_t$ according to the $\sM_B$ in such a way that particles coloured in a block $B$ correspond to particles in $\sP_B$ but with a time shift of length $\tau_B$. We will abuse notation somewhat by conflating a particle $a$ in some $\sP_B$ with a particle in $X_t$ matched according to our construction. Define
\begin{equation}
\label{E:atti}
\oa(t):= a(t-\tau_B), \qquad \oa^h = a^h - \tau_B.
\end{equation}
These will be, respectively, the trajectory and healing process of $a$ in $X_t$. Particle locations in $\sP_B$ at time $0$ tell us the particle locations in $X_t$ at time $\tau_B$.

It will be  enough to construct the process of previously coloured particles $X_t^*$ at every time $t$, since every particle is eventually coloured. This holds since the colouring process spreads at least linearly by Definition \ref{D:SI-colouring}(a), whereas particles spread diffusively.
Let $B_0$ be the block containing the origin so that $\tau_{B_0}=0$.  At time $0$ we add all simple $\sP_{B_0}$-particles that are in $B_0$ at time $0$ to $X_t^*$ plus an infected ignition particle at the origin. The simple particles evolve according to their paths $a(t)$ while the initially infected particle evolves according to $W_{B_0, \operatorname{ig}}(t)$. Particles become infected if they enter the same vertex as another infected particle.  New particles can enter $X_t^*$ in two ways, either when a new block becomes coloured or when an uncoloured particle enters a coloured block for the first time.

{\bf Case 1:  A newly coloured block:} When a block $B$ is coloured for the first time according to Definition \ref{D:SI-colouring} we add  particles to $X_t^*$ as follows.  If a particle $a$ from $\sP_B$ is in $B$ at time $0$, then we add it to $X_t^*$ at time $t=\tau_B$ if for all $0\leq t < \tau_B$  we have that $\tau_{a(t-\tau_B)} > t$.  This condition is equivalent to saying that a particle with trajectory $a(t-\tau_B)$ first hits a coloured block at time $\tau_B$. The trajectory and healing process of this particle is given by $\oa(t), \oa^h$.

If $B$ is ignited at a vertex $x \in \del B$, the ignition particle follows special rules.  Instead of continuing to follow the trajectory given by the block it was initially coloured by, when it becomes the ignition particle of $B$ at time $\tau_B$ we alter its future trajectory to $x+ W_{B,\operatorname{ig}}(t-\tau_B)$ and its future healing process to $\tau_B + W^h_{B, ig}$. Note that some particles may ignite multiple boxes at distinct times, in which case their future trajectories and healing processes will change more than once.

{\bf Case 2:  Particle first entering a coloured block:} For times $t\in(\tau_B,\tau_B+\xi]$ new particles are revealed in $B$ in the process $X_t^*$ according to the following rules.  If $a$ is a particle in $\sP_B$ that enters $B$ at time $t-\tau_B$, then we add it to $X_t^*$ at time $t$ if for all $0\leq s < t$  we have that $\tau_{a(s-\tau_B)} > t$.  This condition is equivalent to saying that a particle with trajectory $a(s-\tau_B)$ first hits a coloured block at time $t$.  Again, the trajectory and healing process of this particle are given by $\oa(s), \oa^h$.  Particles can only join $X^*_t$ in this way during the time interval $(\tau_B,\tau_B+\xi]$ since at time $\tau_B+\xi$ the neighbouring blocks of $B$ are already coloured by construction.

With this construction, a particle $a$ from $\sP_B$ is added with trajectory $\oa$ if and only if $B(a) = B$. It follows from independence properties of Poisson processes that the construction above indeed describes an SIR process. This is shown in Section 5 of \cite{dauvergne2021spread} in the setting when particles cannot heal. The proof goes through verbatim with the minor addition of healing processes. 

We would like to partition the set of all particles in the SIR process $X_t$ into sets $H_B, B \in \sB$ where particle trajectories in each $H_B$ are entirely determined by $\sM_B$. However, this is not quite possible since some particles that are coloured in $B$ may become ignition particles for other blocks and hence their trajectories will not be entirely $\sM_B$-measurable. To deal with this, it will be convenient to think of a particle $a \in H_B$ as healing when it becomes the ignition particle for a new block $B'$, and a new infected particle $W_{B', \operatorname{ig}} \in H_{B'}$ appearing instantaneously in its place.

Let $H_B$ be the set of all particles that are first coloured in $B$. Definition~\ref{D:SI-colouring}(a) guarantees that our colouring process is moving at a linear speed depending only on $L$, which helps to ensure that there are typically many particles in $H_B$. With this in mind, define $H^-_B$ to be all particles $a\in \sP_B$ that satisfy the following additional constraint:
\begin{equation}
\label{eq:principal}
\sup_{t\leq 0} d(a(t),B) - L \lfloor |t\xi^{-1}|/4 \rfloor = 0.
\end{equation}
We call the particles in $H_B^-$ \textbf{principal particles} and claim that $H_B^- \sset H_B$. Indeed, let $a \in H_B^-$. Equation \eqref{eq:principal} implies that $a(0) \in B$. Now suppose $B'$ gets coloured before $B$. To show that $H_B^- \sset H_B$, we must check that $a(t) \notin B'$ for all $t \ge \tau_{B'} - \tau_{B}$.
By Definition \ref{D:SI-colouring}(a) and the bound \eqref{E:block-distance}, we have the bound
\begin{equation}
\label{E:first-lipschitz}
|\tau_B-\tau_{B'}| \leq d_\sB(B, B')\xi < 4\ceil{d(B, B')/L} \xi.
\end{equation}
Condition~\eqref{eq:principal} then implies that $d(a(t), B) < d(B, B')$ for all $t \ge \tau_{B'} - \tau_{B}$, so $a(t) \notin B'$.

Also let $H_B^+$ be the set of all particles $a \in \sP_B$ such that $a(t) \in B$ for some $t \in [0, \xi]$. We then have $H_B \sset H_B^{++} := H_B^+ \cup \{W_{B, \operatorname{ig}}\}$. The $\sM_B$-measurable set of trajectories in $H_B^{++}$ describes all randomness contributed to the SIR process by $\sM_B$.

\subsection{Blue Seeds}
\label{S:blue-seeds}

In this section we give a set of conditions on $\sM_B$ which will ensure the efficient spread of the infection to the neighbouring blocks if the block $B$ is ignited, i.e. coloured according to rule (b) or (c).  The complement of this event will correspond to marking the block as a blue seed in our SSP coupling. For the remainder of Section \ref{S:linear}, we fix an arbitrary $\ka > 4000$.

Setting some notation, let 
$$
U_x = \{B \in \sB_L : \inf_{y \in B} \|x - y\|_\infty \le \al L \}.
$$
When the meaning is clear from context, we will also let $U_x$ denote the set of vertices contained in the union of the blocks in $U_x$. By construction, $U_x$ is a rectangle in the plane made of either 1, 2, or 4 blocks. It includes the block containing $x$. 

Furthermore, if $x$ is the ignition site for the block $B$, all blocks in $U_x$ are coloured at or before time $\tau_B$. Set 
\[
\partial B^\#= \bigcup_{B'\in \sB_L}\{x\in \partial B': B\in U_x\},
\]
which is the set of possible ignition locations for the ignition particle of $B$ (the shape of the set $\partial B^\#$ resembles a hashtag $\#$).

The ignition particle for $B$ is infected, starts at $x\in\partial B^\#$ and follows the path $x+ W_{B_x,\operatorname{ig}}((t-\tau_B))$ for $t \ge \tau_B$, where $B_x$ is the block containing $x$. Define the event
\[
\cA^{(1)}_{B}=\bigcap_{B':d(B,B')\leq 2}\Big\{\sup_{0\leq s\leq \xi/ \log \log L} |W_{B',\operatorname{ig}}(s)| \leq \frac12 \alpha L\Big\}.
\]
Since $d(B_x, B) \le 2$ for all $x \in \del B^\#$, this event ensures that the ignition particle for $B$ remains in $U_x$ until time $\tau_B + \xi/ \log \log L$, and hence cannot become the ignition particle of another block prior to this time.

For $x \in \del B^\#$ and each block $B' \not\in U_x$ with $d(B, B') = 1$, define the event
\begin{align*}
\cA^{(2)}_{B,B',x}=\bigg\{\exists a \in H_B^-, \; &\exists \; 0\leq s \leq s' < \tfrac{\xi}{\log\log L} \text{ s.t.} \\  a(s) = x+ W_{B_x,\operatorname{ig}}(s);\quad 
&a(s')\in B'; \quad \forall\;  0\leq s''<s', a(s'')\in U_x\bigg\}.
\end{align*}
This event asks for at least one principal particle in $B$ to have a trajectory that intersects the ignition particle's trajectory at some time $s$, and to stay within $U_x$ until some time $s'$, at which point it enters $B'$. 
Finally let $\cA^{(3)}_{B}$ be the event that the ignition particles of all blocks $B'$ with $d(B, B') \le 2$ have no healing events in the interval $[0, L^3]$ and let $\cA^{(4)}_{B}$ be the event that no particle in $H_B^-$ has a healing event in the interval $[0, L^3]$.  Finally we want that many principal particles which do not become ignition particles become infected. Thus we define
\begin{align*}
Q_{B,x}=\Big\{a\in H_B^- &: d(a(0),B^c) \geq \tfrac1{10} L, \exists s\in[0,\tfrac{\xi}{\log\log L}], a(s) = x+ W_{B_x,\operatorname{ig}}(s) \Big\},
\end{align*}
and
\begin{align*}
\cA^{(5)}_{B,x}=\Big\{|Q_{B,x}| \geq L^{2-\log^{-1/2} L} \Big\}.
\end{align*}
We combine these events and define
\begin{equation}
\label{E:cA123}
\cA_B=\cA_B^{(1)}\cap\cA^{(3)}_{B}\cap \cA^{(4)}_{B} \cap  \bigcap_{x\in\partial B^\#} \bigg(\cA^{(5)}_{B,x} \cap \bigcap_{\substack{B'\in \sB \smin U_x \\
		d(B,B')\le 1}}\cA^{(2)}_{B,B',x} \bigg)
\end{equation}
We call the block $B$ a \textbf{blue seed} if $(\cA_B)^c$ holds. It follows from the definition that $\cA_B$ is measurable with respect to the $\sigma$-algebra generated by $\{\sM_{B'}, d(B',B) \leq 2\}$. 

We have asked for many things from our event $\cA_B$ that together will ensure the efficient spread of the infection. The simplest consequence of the definition is that if $\cA_B$ holds and $B$ is ignited, then all its neighbouring blocks in $\sB_L$ will be coloured before time $\tau_B + \xi/\kappa$.

\begin{lemma}
	\label{L:lower-scale-ignition}
	Suppose that a block $B \in \sB_L$ is ignited at time $\tau_B$ at location $x$  and that $\cA_B$ holds. Then any neighbouring block $B'$ is coloured by time $\tau_B + \xi/\kappa$.
\end{lemma}

\begin{proof}
	First, we may assume that $B' \notin U_x$, since else $B'$ is coloured by time $\tau_B$.
	Select some $a\in H_B^-$ and  $s,s'$ that make the event $\cA^{(2)}_{B,B',x}$ hold. The particle $a$ must be susceptible at time $\tau_B$ and at time $\tau_B + s'$ will meet the ignition particle which by $\cA^{(3)}_{B}$ is still infected at that time. Therefore $a$ will become infected at or prior to to time $\tau_B + s'$. Moreover, it will reach $B'$ at time $\tau_B+s \leq \tau_B + \xi/\kappa$ and by $\cA^{(4)}_{B}$ it will not heal in the interval $[\tau_B, \tau_B + \xi/\kappa]$. Therefore this particle will ignite $B'$ at time $\tau_B + s$ if $B'$ has not already been previously coloured.
\end{proof}

With these colouring rules, we can associate an SSP to the given SIR process and then use the machinery of Section \ref{S:SSP}. To do this, it is more natural to use the time scaling of the current section, which will allow the red and blue clocks to take values in $[0, \xi/\ka]$ and $[0, \xi]$, rather than in $[0, 1]$ and $[0, \ka]$, respectively. We will also define the SSP directly on the block set $\sB$, rather than on $\Z^2$. Clearly this will map to an SSP on $\Z^2$ via the correspondence from \eqref{E:BL-construct}.

We let $B<_{\operatorname{SIR}} B'$ for two blocks $B, B' \in \sB$ if both they are ignited by the same ignition particle at location $x$, and $d(x, B) < d(x, B')$. Note that our ignition rules guarantee that either $d(x, B) < d(x, B')$ or $d(x, B) < d(x, B')$.
The random directed graph on $\sB$ with edges given by pairs $(B, B') \in \sB^2$ with $B <_{\operatorname{SIR}} B'$ is acyclic.
For each directed edge $(B, B')$ between adjacent vertices in $\sB$, define
\begin{equation*}
\begin{split}
X_\fR(B, B') = \begin{cases}
\lf(\tau_{B'} - \tau_B \rg) \wedge \frac{\xi}{\ka}, \qquad &\tau_{B'} - \tau_B > 0,
\\
0, \qquad &\tau_{B'} = \tau_B \text{ and } B <_{\operatorname{SIR}} {B'}, \\
\frac{\xi}{\ka} \qquad &\text{ else.}
\end{cases}
\end{split}
\end{equation*}
Similarly define the blue clock by
\begin{equation*}
\begin{split}
X_\fB(B, B') = \begin{cases}
\tau_{B'} - \tau_B, \qquad &\tau_{B'} - \tau_B > 0,
\\
0, \qquad &\tau_{B} = \tau_{{B'}} \text{ and } B <_{\operatorname{SIR}} B', \\
\xi \qquad &\text{ else.}
\end{cases}
\end{split}
\end{equation*}
\begin{prop}
	\label{P:SI-to-BR}
	The clocks above along with the set of blue seeds $\fB_* = \{B \in \sB: \cA_B^c \text{ holds}\}$ a.s.\ define a finite speed SSP on $\sB$. Moreover, a.s.\ for every $B \in \sB$, the colouring time $T(B)$ equals $\tau_B$. That is,
	\begin{equation}
	\label{E:BR-tau-eqn}
	\mathfrak{B}(t) \cup \mathfrak{R} (t) = \{B \in \sB : \tau_B \le t \}.
	\end{equation}
	Finally, let $\operatorname{IGN} : =\{B \in \sB: B \text{ is ignited or } B = B_0\}$. Then a.s.\ 
$
	\fR(\infty) \sset \operatorname{IGN} \smin \fB_*.$
\end{prop}

A version of Proposition \ref{P:SI-to-BR} for the SI model was shown in \cite{dauvergne2021spread}, Proposition 3.9. As the proof goes through verbatim for the SIR model, we omit it.

Next, to apply the framework of Section \ref{S:SSP} to understand the SIR process, we need to show that the blue seed probability is small.

\begin{proposition}\label{p:blueSeed}
	There exists $\al > 0$ such that for any $L$ satisfying $L^5 \le \nu^{-1} \le L^6$, we have
	\[
	\P[(\cA_B)^c] \leq \epsilon_\nu,
	\]
	where $\eps_\nu \to 0$ with $\nu$.
\end{proposition}

When analyzing rare events in the final section of the paper, we will also require the following conditional estimate which will fall out immediately from the proof of Proposition \ref{p:blueSeed}.

\begin{lemma}\label{L:blueSeedconditional}
	Let $\al, L$ be as in Proposition \ref{p:blueSeed}. Fix a block $B$, a subset $\sC \sset \{B' : d(B, B') \le 2\}$, and let $\cW$ be the event
	$$
	\bigcap_{B' \in \sC} \{ W_{B', \operatorname{ig}}(t) = W_{B', \operatorname{ig}}(0) \; \forall t \in [0, L^4], \; W^h_{B', \operatorname{ig}} \cap [0, L^3] = \emptyset, \;W^h_{B', \operatorname{ig}} \cap [L^3, L^4] \ne \emptyset\}. 
	$$
	Then 
	$$
	\P(\cA_B^c \;|\; \cW) \le \ep_\nu.
	$$
\end{lemma}

Before proving Proposition \ref{p:blueSeed} and Lemma \ref{L:blueSeedconditional}, we will see how they imply Theorem \ref{T:main-1}. First, we have the following corollary of Proposition \ref{p:blueSeed}. 

\begin{corollary}
	\label{C:blue-domination}
	The process $\fB$ of blue seeds on $\sB_L$ is stochastically dominated by a Bernoulli process of mean $\de_\nu$, where $\de_\nu \to 0$ with $\nu$.
\end{corollary}

\begin{proof}
	We appeal to Theorem 0.0(i)\footnote{It really is Theorem 0.0, this is not a typo} in \cite{liggett1997domination}, which states the following. Let $d \ge 1$, and suppose that $X:\Z^d \to \{0, 1\}$ is a random process such that for any vertex $v \in \Z^d$, the conditional probability that $X(v) = 1$ given all the values of $X$ on vertices at $\ell^\infty$-distance at least $k$ away from $v$ is at most $\ep$. Then $X$ is stochastically by an i.i.d.\ Bernoulli process $Y$ on $\Z^d$ such that $\E Y(0) = f(\ep)$, where $f(\ep) \to 0$ with $\ep$. Here the function $f$ depends on $k$ and $d$. 
	In our setting, we let $X = \{ B \in \sB_L : \cA_B^c \text{ holds}\}$, and use that the event $\cA_B^c$ is measurable given the $\sig$-algebras $\sM_{B'}, d(B, B') \le 2$ and hence is independent of the collection of events $\cA_{B'}^c, d(B, B') \ge 3$. Proposition \ref{p:blueSeed} then implies the result.
\end{proof}

\begin{proof}[Proof of Theorem \ref{T:main-1}]
	Fix $\mu > 0$. By Proposition \ref{P:high-recovery-rate}, we can take $\nu^+ = 8 \mu$. Now, the SIR process survives forever if there exists $L$ for which the coupled SSP $(\fR, \fB)$ in Proposition \ref{P:SI-to-BR} satisfies $|\fR(\infty)| = \infty$. Choosing $L$ as in \eqref{E:Lnu-relationship},  by Corollary \ref{C:blue-domination} and Theorem \ref{T:random-red-blue} we have
	$$
	\P(d, \mu, \nu) \ge \P(|\fR(\infty)| = \infty) \ge 1 - C \de_\nu,
	$$
	which tends to $1$ as $\nu \to 0$.
\end{proof}

\subsection{Tools for analyzing random walks}

In this part we gather a couple of preliminary estimates used in the proof of Proposition \ref{p:blueSeed}.  With $W_t$ a continuous time random walk, for $x\in B$ define
\[
p_x = \P\Big[\forall t\geq 0: d(x+W_t,B^c) \leq L \lfloor |t\xi^{-1}|/4 \rfloor \Big]
\]
At time $\tau_B$ the principal particles $H_B^-$ are Poisson distributed with intensity $p_x \mu$.
For any $x$ such that $d(x,B^c) \geq \frac1{10} L$ we have that
\begin{align}\label{eq:principalLB}
p_x &\geq 1 - \P[\sup_{0\leq t \leq 3\alpha L^2}  |W_t| > \frac1{10} L^2 ] - \sum_{j\geq 1} \P[\sup_{0\leq t \leq 3(j+1)\alpha L^2}  |W_t| > Lj ]   \geq \frac34,
\end{align}
for all large enough $L$ and $\alpha \leq c$ by Lemma \ref{L:rw-estimate}. The next lemma then follows from a standard concentration estimate on Poisson random variables.
\begin{lemma}\label{l:HBsize}
	Let $L$ be sufficiently large. The size of $|H_B^+|, |Q_B|$ and $|H_B^{-}|$ are all Poisson distributed with the following means:
	\begin{align*}
	\E |H_B^{-}| &= \sum_{x\in B} p_x \mu \geq  \frac13 \mu L^2\\
	\E |Q_B| &= \sum_{x\in B,d(u,B^c)\geq L/10}  p_x \mu \geq \frac1{3} \mu  L^2\\
	\E |H_B^+| &= \sum_{x\in \Z^2} \mu  \P[\exists t\in[0,\xi] : W_t+x \in B ] \leq  \frac32 \mu L^2.
	\end{align*}
	In particular, for some $c>0$ we have that
	\begin{equation}\label{eq:particleSizeBounds}
	\P[|H_B^+|\leq 2\mu L^2, |H_B^{-}|\geq |Q_B|\geq \frac14 \mu L^2] \geq 1 - \exp(-c \mu L^2).
	\end{equation}
\end{lemma}

Next, let $H$ be a set of particles and let $x:[s,s+T]\to\Z^2$ be an arbitrary cadlag path.
The following lemma will be useful for counting how many particles in $H$ hit $x$. With this in mind, let
\[
N(H,s,T,x):= \sum_{a\in H} \mathbf{1}(\{t\in [s,s+T]: a(t) = x(t)\}\neq \emptyset)
\]
be the number of $H$-particles that intersect the path $x$ and let
\[
R(H,s,T,x):= \sum_{a\in H} \int_s^{s+T} \mathbf{1}(a(t) = x(t))dt.
\]
be the aggregate intersection time of the $H$-particles with $x(t)$.
\begin{lemma}\label{l:hittingNumberSimple}
	There exists a constant $C>0$ such that for any set of particles $H$ performing continuous-time simple random walks and any path $x:[s, s + T] \to \Z^2$, for $T \ge 2$ we have
	\[
	\E[N(H,s,T,x) \mid \sF_s] \geq \frac{1}{C\log T} \E[R(H,s,T,x) \mid \sF_s].
	\]
	Here $\sF_s$ is the $\sig$-algebra generated by the walks in $H$ up to time $s$.
	The same results holds if the walks in $H$ instead perform a different Markov chain with transition matrix $Q_{x,y}^{t,t'}$ satisfying the following bound for all $t_1 > 0$, $T\geq 2,$ and $x \in \Z^2$:
	\begin{equation}\label{eq:transitionProbReturnCondition}
	    \int_{t_1}^{t_1 +T} \sup_{y\in \Z^2} Q_{x,y}^{t_1,t} dt \leq C\log T.
	\end{equation}
\end{lemma}

\begin{proof}
First note that the transition probability for a simple random walk satisfies \eqref{eq:transitionProbReturnCondition}.
Now, for each particle $a\in H$ let
	\begin{align*}
	n_a &= \P[\{t\in [s,s+T]:  a(t) = x(t)\}\neq \emptyset\mid \sF_s],\\
	r_a &= \E\Big[\int_s^{s+T} \mathbf{1}(a(t) = x(t)) dt \Bigm| \sF_s\Big].
	\end{align*}  
	Setting $\varsigma_a = \inf\{t\geq s: a(t) = x(t)\}$ we have
	\begin{align*}
	r_a &= \E\Big[\mathbf{1}(\varsigma_a \leq s+T) \int_{\varsigma_a}^{s+T} Q^{\varsigma_a,t}_{x(\varsigma_a),x(t)} dt \Bigm| \sF_s\Big]\\
	&\leq C\log (T) \E[\mathbf{1}(\varsigma_a \leq s+T) \mid \sF_s] \leq  C\log (T)n_a.
	\end{align*}
	Summing over $a \in H$ yields the result.
\end{proof}

We will also need a slight variant of Lemma~\ref{l:hittingNumberSimple}. Define
\begin{align*}
N'(H,s,T,x)&:= \sum_{a\in H} \mathbf{1}(\{t\in [s,s+T]: a(t') = x(t) \text{ for all } t'\in[t,t+1]\} \neq \emptyset) \\
R'(H,s,T,x)&:= \sum_{a\in H} \int_s^{s+T} \mathbf{1}(a(t') = x(t) \text{ for all } t'\in[t,t+1])dt
\end{align*}

\begin{lemma}\label{l:hittingNumber}
	There exist constants $C_1,C_2>0$ such that for any set of particles $H$ performing continuous time simple random walks and any path $x:[s, s + T] \to \Z^2$, for $T \ge 2$ we have
	\[
	\E[N'(H,s,T,x) \mid \sF_s] \geq \frac{1}{C\log T} \E[R'(H,s,T,x) \mid \sF_s].
	\]
\end{lemma}

\begin{proof}
For each particle $a\in H$ let
	\begin{align*}
	n_a &= \P[\{t\in [s,s+T]: a(t') = x(t) \text{ for all } t'\in[t,t+1]\} \neq \emptyset \mid \sF_s],\\
	r_a &= \E\Big[\int_s^{s+T} \mathbf{1}(a(t') = x(t) \text{ for all } t'\in[t,t+1])dt \Bigm| \sF_s \Big].
	\end{align*}  
	Setting $\varsigma_a = \inf\{t\geq s: a(t) = x(t)\}$ we have
	\begin{align*}
	r_a &= e^{-1}\E\Big[\mathbf{1}(\varsigma_a \leq s+T) \int_{\varsigma_a}^{s+T} Q^{\varsigma_a,s+T}_{x(\varsigma_a),x(t)} dt \Bigm| \sF_s\Big]\\
	&\leq Ce^{-1}\log (T) \E[\mathbf{1}(\varsigma_a \leq s+T) \mid \sF_s] \leq  C\log (T)n_a.
	\end{align*}
which yields the result after summing over $a$.
\end{proof}

\begin{remark}\label{rem:hittingNumberConcentration}
	The counting function $N = N(H,s,T,x)$ is a sum of independent indicator random variables so by Bernstein's inequality,
	\begin{align}
	\nonumber
	\P\Big[N \leq \frac{1}{2} \E[N \mid \sF_s] \;\Big|\; \sF_s\Big] \leq \exp(- c \E[N \mid \sF_s]).
	\end{align}
	The same bound holds with $N' = N'(H,s,T,x)$ by identical reasoning.
\end{remark}

\subsection{Proof of Proposition~\ref{p:blueSeed} and Lemma \ref{L:blueSeedconditional}}

The proofs of these two statements are essentially the same. We focus on proving Proposition~\ref{p:blueSeed} and will indicate when a minor difference is required for the proof of Lemma \ref{L:blueSeedconditional}. 

By standard random walk estimates (Lemma \ref{L:rw-estimate}),
\begin{equation}\label{eq:a1Bound}
\P[\cA^{(1)}_{B}] \to 1    
\end{equation}
as $L\to \infty$. Note that in the context of Lemma \ref{L:blueSeedconditional}, $\P[\cA^{(1)}_{B} \mid \cW] \ge \P[\cA^{(1)}_{B}]$ since the events for $\cA_B^{(1)}$ involving particles that are affected by $\cW$ are necessarily satisfied on the event $\cW$.

 Moreover, the relationship \eqref{E:Lnu-relationship} implies that any particle has a healing event in the interval $[0, L^3]$ with probability at most $64 L^{-3}$.  By this bound, Lemma~\ref{l:HBsize}, and a union bound we have that
\begin{equation}\label{eq:a34Bound}
\P[\cA^{(3)}_{B} \cup \cA^{(4)}_{B}] \geq 1 - CL^{-1}.
\end{equation}
Again, we have $\P[\cA^{(3)}_{B} \cup \cA^{(4)}_{B} \mid \cW] \ge \P[\cA^{(3)}_{B} \cup \cA^{(4)}_{B}]$ since the event $\cW$ guarantees that some ignition particles receive no healing events in the interval $[0, L^3]$.

It remains to bound the probabilities of $\cA^{(2)}_{B,B',x}$ and $\cA^{(5)}_{B,x}$.  We define the following sets:
\begin{align*}
U_x^- &= \{y\in B:d(y,U_x^c)\geq \frac14 \alpha L\},\quad U_x^* = \{y\in B:d(y,U_x^c)\geq \frac12 \alpha L\},\\
B^- &= \{y\in B: d(y, B^c) \geq \tfrac1{10} L\},
\end{align*}
and 
\begin{align*}
Q''_{B,x}=\Big\{a\in H_B^- &: a(0)\in B^-, \forall 0 \le s'\leq \tfrac{\xi}{2\log\log L} \quad a(s')\in U_x^- \Big\}.
\end{align*}
At time $0$, $Q''_{B,x}$ is a Poisson process of particles on the set $B^-$ with density $\mu p_y q_{x,y}$ where $p_y$ was defined in \eqref{eq:principalLB} and
\[
q_{x,y}=\P[W_s \in U_x^- \text{ for all } s \in [0, \tfrac{\xi}{2\log\log L}] ],
\]
where $W_s$ is a simple random walk started at $y$. When $x \in \del B^\#, y \in  B^-$, then by Lemma \ref{L:rw-estimate} we have that $q_{x,y} \ge 1 - o(1)$ where the $o(1)$ term tends to $0$ as $L \to \infty$ uniformly over $x \in \del B^\#, y \in  B^-$.
By equation~\eqref{eq:principalLB}, $p_y\geq \frac34$  and so for sufficiently large $L$,
\[
\E[|Q''_{B,x}|]\geq \frac12 \mu |B^-| \geq \frac14\mu L^2
\]
for all $x \in \del B^\#$.
The future trajectories of the particles in $Q''_{B,x}$ are random walks conditioned to stay inside $U_x^-$ until time $\tfrac{\xi}{2\log\log L}$. These conditioned walks are instances of a Markov chain.  Moreover, the set $U_x^-$ is of the form $[a, b] \times [c, d]$ where $b - a \ge L/2, d-c \ge L/2$. Therefore conditioning a random walk to stay inside $U_x^-$ up in the interval $[0, \tfrac{\al L^2}{2\log\log L}]$ will not greatly increase the transition probabilities between points and the transition matrix of this Markov chain will satisfy
\[
P^{t,t'}_{y,y'} \leq 1\wedge \frac{C}{t'-t}
\]
for all $0\leq t'-t$. In particular, this Markov chain satisfies equation~\eqref{eq:transitionProbReturnCondition}.  Furthermore, for all $y\in B^-$, all $y'\in U_x^*$ and all $s\in[\tfrac{\xi}{4\log\log L},\tfrac{\xi}{2\log\log L}]$ we have that
\[
P^{0,s}_{y,y'} \geq \frac1{L^2} \exp(-C(\log \log L)),
\]
since $y,y'$ are both distance $\al L/4$ away from the boundary of $U_x^-$ and are distance at most $4L$ from each other.
If we let $x(t)$ be a path in $U_x^*$ and let
\[
R(x)=\sum_{a\in Q''_{B,x}}\int_{\tfrac{\xi}{4\log\log L}}^{\tfrac{\xi}{2\log\log L}} \mathbf{1}(a(t)=x(t)) dt
\]
then
\begin{align}
\E R(x)&=\sum_{a\in Q''_{B,x}}\int_{\tfrac{\xi}{4\log\log L}}^{\tfrac{\xi}{2\log\log L}} P^{0,t}_{a(0),x(t)} dt\nonumber\\
&\geq \E[Q''_{B,x}] \inf_{y \in B^-} \int_{\tfrac{\xi}{4\log\log L}}^{\tfrac{\xi}{2\log\log L}} P^{0,t}_{y,x(t)} dt\nonumber\\
& \geq \tfrac{cL^2}{4\log\log L} \exp(-C_1(\log \log L))\nonumber\\
&\geq L^2 \exp(-C_2(\log \log L)).
\end{align}
Now let
\begin{align*}
Q'_{B,x}=\Big\{a\in Q^{''}_{B, x} &: \exists s\in[0,\tfrac{\xi}{2\log\log L}] \text{ such that } a(s) = x+ W_{B_x,\operatorname{ig}}(s)\Big\}.
\end{align*}
Since $x+ W_{B_x,\operatorname{ig}}(s) \in U_x^*$ for $s\in[0,\tfrac{\xi}{2\log\log L}]$ on the event $\cA^{(1)}_B$, by Lemma~\ref{l:hittingNumberSimple} we have that
\begin{equation}
\E[|Q'_{B,x}|\mid \cA^{(1)}_B] \geq \frac{C}{\log L} \E[R(x+ W_{B_x,\operatorname{ig}})\mid \cA_B^{(1)}] \geq L^2 \exp(-C_3(\log \log L)).
\end{equation}
Now, conditional on $\cA^{(1)}_B$ and all particle trajectories from $\sP_B$ up to time $\xi/(2 \log \log L)$, each particle in $Q'_{B,x}$ performs a simple random walk after time $\xi/(2 \log \log L)$. Therefore under this conditioning, each particle in $Q'_{B,x}$ has probability at least $\exp(-C_4 (\log\log L))$ of exiting $U_x$ through $B'$ for any neighbouring block $B' \not\subset U_x$ by time $\tfrac{\xi}{2\log\log L}$. Calling the set of particles that does this $Q_{B, B', x}'$, we therefore have
\begin{align*}
\E[|Q'_{B, B', x}|  \mid \cA^{(1)}_B] &\geq \exp(-C_4 (\log\log L))\E[|Q'_{B,x}|  \mid \cA^{(1)}_B] \\
&\geq L^2 \exp(-(C_3 + C_4)(\log \log L)).
\end{align*}
Conditional on the trajectory of $W_{B_x,\operatorname{ig}}$, we also have that $Q'_{B, B', x}$ is a Poisson random variable and so
\begin{align*}
\P\Big[|Q'_{B, B', x}| \Bigm| \cA^{(1)}_B\Big] &\geq \P\Big[\hbox{Poisson}\big(L^2 \exp(-(C_3 + C_4)(\log \log L))\big) \geq L^{2-\log^{-1/2} L}]\nonumber\\
&\geq 1-L^{-100}
\end{align*}
for $L$ sufficiently large. Now, $Q'_{B, B', x} \sset Q_{B, x}$ for any $B'$, so
\begin{align}
\label{eq:a5Bound}
\P\Big[\cA^{(5)}_{B,x} \Bigm| \cA^{(1)}_B\Big]\geq 1-L^{-100}.
\end{align}
Also, $\cA^{(2)}_{B,B', x}$ is contained in the event where $|Q'_{B, B', x}| \ne \emptyset$ for all $B' \notin U_x$ with $d(B, B') = 1$, so similarly
\begin{align}
\label{eq:a2Bound}
\P\Big[\cA^{(2)}_{B,B', x} \Bigm| \cA^{(1)}_B\Big]\geq 1-4L^{-100}.
\end{align}
Note that the argument for \eqref{eq:a5Bound} and \eqref{eq:a2Bound} goes through verbatim if we condition on $\cA^{(1)}_B \cap \cW$ rather than just on $\cA^{(1)}_B$, since this stricter conditioning just determines some of the ignition trajectories. 
Combining~\eqref{eq:a1Bound}, \eqref{eq:a34Bound}, \eqref{eq:a5Bound} and~\eqref{eq:a2Bound} together with a union bound we have that
\[
\P[(\cA_B)^c]\to 0
\]
as $\nu\to 0$ (or equivalently as $L \to \infty$), completing the proof of Proposition~\ref{p:blueSeed}. Similarly, $\P[(\cA_B)^c \mid \cW]\to 0$, yielding Lemma \ref{L:blueSeedconditional}.

\section{Local and global events}
\label{S:global}
The remainder of the paper is devoted to the proof of the more refined Theorem~\ref{T:main-2}. One of the difficulties in proving this theorem is to separate local and global effects on the process. In this section we define events which will help us to deal with this. Here and throughout the remainder of the paper, for a set of blocks $\sD \sset \sB$, we write $\sM_{\sD}$ for the $\sig$-algebra generated by the $\sM_B, B \in \sD$. 

Suppose that we wish to understand the behaviour of the SIR process $X_t$ in a spatial window near a block $B$ and in a window of time centered at the colouring time $\tau_B$. In this window, we would like to say that $X_t$ is essentially governed only by the $\sig$-algebras $\sM_{B'}$ for $B'$ close to $B$. However, there are three global effects that can disrupt this:
\begin{enumerate}
	\item The behaviour of the coupled SSP from Proposition \ref{P:SI-to-BR} in blocks near $B$.
	\item The presence of particles coloured in far away blocks coming close to $B$. 
	\item The differences $\tau_{B'} - \tau_B$, for all blocks $B'$ close to $B$.
\end{enumerate}
We handle the first of these effects by looking not at the exact behaviour of the SSP, but rather only the encapsulating sets $[\fC]$ and $[\fD]$. First, throughout the proof we assume that $L$ is sufficiently large so that Theorem \ref{T:random-red-blue} holds in the coupled SSP. Next, we have no control over the SSP or the SIR process unless we are working on the event
\begin{equation}
\label{E-event}
\cC = \Big\{\fB_* = \bigcup_{i=1}^\infty A_i(\fB_*), 0 \in [\fD]^c, [\fD]^c \sset \fR(\infty), [\fD] \text{ has only bounded components}\Big\}.
\end{equation}
This is essentially the event $\cE$ in \eqref{E:CD}, with $[\fC]$ replaced by $[\fD]$ in two places. 
 It satisfies the same probability bound as $\cE$ since $\fC \sset \fD$ and $[\fD]$ only has bounded components almost surely for $p$ small enough by Lemma \ref{L:largest-scale}. We choose to only work with $[\fD]$ to simplify things moving forward.

On the event $\cC$, the only information we need about the SSP will be contained in the set $[\fD]$. We can use inequality \eqref{E:fDk-estimate} in Lemma \ref{L:largest-scale} to estimate the probability that $[\fD]$ is not determined locally. For a block $B$ and $r \in (0, \infty]$, let 
$$
D_\sB(B, r) = \{B' \in \sB: d_\sB(B, B') \le r\}.
$$
Define the encapsulating set $[\fD(B, r)]$ to be the version of $[\fD]$ generated using only blue seeds from blocks in $D_\sB(B, r)$. Note that these sets are nondecreasing in $r$ by Lemma \ref{L:dom-seeds}. The set $[\fD(B, r)]$ is a local approximation of $[\fD]$. The two sets are typically equal away from the boundary of $D_\sB(B, r)$; in particular they typically coincide on the block $B$. Indeed, we have the following estimate.
\begin{lemma}
\label{L:global-local-ests}
For $r < r' \in (0, \infty]$, define the event
\begin{equation}
\cD_B(r, r') = \{ \{B\} \cap [\fD(B, r)] = \{B\} \cap [\fD(B, r')]\}.
\end{equation}
Then as long as $\nu$ is sufficiently small given $\mu$, we have $\P [\cD_B(r, r')] \ge 1 - \exp(-c r^{c/\log \log r})$ for some $c>0$.
\end{lemma}

Moving forward, we will write $\cD_B(r) = \cD_B(r, \infty)$. It will also be useful to note the transitivity relationship
\begin{equation}
\label{E:transitivity}
\cD_B(r, r') \cap \cD_B(r', r'') = \cD_B(r, r'')
\end{equation}
for $r < r' < r''$, which follows from the monotonicity of the sets $[\fD(B, r)]$ in $r$. We will typically use this with $r'' = \infty$. In this case, \eqref{E:transitivity} splits the global event $\cD_B(r)$ into a local part $\cD_B(r, r')$ which is measurable given the $\sig$-algebra $\sM_{D_\sB(B, r' + 2)}$, and a global part $\cD_B(r')$ whose probability decays quickly with $r'$.
\begin{proof}
If $\{B\} \cap [\fD(B, r')] \ne \{B\} \cap [\fD(B, r)]$, then by monotonicity $B \notin [\fD(B, r)]$ but $B \in [\fD(B, r')]$.
We claim that then $B$ is in a component of $[\fD(B, r')]$ of radius at least $r/2$, and hence is also in a component of $[\fD]$ of radius at least $r/2$. This will imply the lemma via \eqref{E:Mn-BC} in Lemma \ref{L:largest-scale}.

Setting some notation, let 
$$
\fB_*^1 = \fB_* \cap D(B, r), \qquad \fB_*^2 = \fB_* \cap D(B, r'),
$$
and for each blue seed $B'$ and $i = 1, 2$, let $k_i(B')\in \N$ be the unique level with $B' \in A_k(\fB_*^i)$. Suppose that the component of $B$ in $[\fD(B, r')]$ has radius less than $r/2$. Then this component only contains blue seeds in the smaller disk $D_\sB(B, r/2)$. One of these blue seeds  $B'$ must satisfy $k_1(B') < k_2(B')$ and so $B' \in \fB^2_{*, k_1(B') + 1} \smin \fB^1_{*, k_1(B') + 1}$. Therefore by applying the first part of Lemma \ref{L:Ak-estimate}, with $X = D_\sB(B, r/2)$, we must have that 
$$
D_\sB(D_\sB(B, r/2) , r_{k_1(B')}/2) \not \sset D_\sB(B, r), 
$$
and hence $r_{k_1(B')} > r$. Therefore $r_{k_2(B')} > 10^{12} r_{k_1(B')}> 100r$ by \eqref{E:gammabound} and so by the definition of $\fD$ the component of $B'$ in $[\fD(B, r')]$ has radius at least $r$. This is a contradiction.
\end{proof}

We can handle the second effect with straightforward bounds on the deviations of random walks. 
For this next lemma, recall that $H_B^{++}$ consists of all trajectories that may be contributed to the SIR process from $\sP_B$ plus the ignition trajectory for $B$.
\begin{lemma}
\label{L:far-particles}
For $t \in \R, s, m > 0$, and a block $B$, define the event
\begin{equation}
\label{E:deviation}
\cE_B(t, s, m) = \bigcap_{a \in H_B^{++}} \{\sup_{r \in [-s, s]} |a(t) - a(t + r)| \le m\}.
\end{equation}
Then $\P [\cE_B(t, s, m)] \ge 1 - C L^2 e^{- \frac{cm^2}{m+s}}$, where $C, c > 0$ are constants depending on $\mu$.
\end{lemma}

\begin{proof}
This bound is immediate from Lemma \ref{L:rw-estimate}, Lemma \ref{l:HBsize}, and a union bound.
\end{proof}

The third effect is more difficult to understand, as we have not established any control over the correlations between different $\tau_B$'s. We will therefore deal with the issue of the random differences $\tau_B - \tau_{B'}$ with a mixture of different methods. Our main tool will be the bound from \eqref{E:first-lipschitz} which crudely guarantees the Lipschitz estimate
\begin{equation}
\label{E:Lipschitz}
|\tau_B - \tau_B'| \le L^2 d(B, B'). 
\end{equation}
As an immediate application of \eqref{E:Lipschitz} we use Lemma \ref{L:far-particles} to control particles that come close to $B$ from far away. 
\begin{corollary}
	\label{C:no-far-particles}
For blocks $B, B'$ and $s > 0$, define the event 
\begin{equation*}
\cL_B(B', s) := \cE_{B'}(0, s + L^2 d(B, B'), d(B, B')/4).
\end{equation*}
Then on $\cL_B(B', s)$, we have that
\begin{equation}
\label{E:LB-event}
\inf \{ d(\bar a (t), B) : t \in [\tau_B - s, \tau_B+ s], a \in H_{B'}^{++} \} \ge d(B, B')/2.
\end{equation}
Moreover,
\begin{align*}
\P [\cL_B(B', s)] &\ge 1 - C L^2 \exp \lf( -\frac{c d(B, B')^2}{s + L^2 d(B, B')}\rg).
\end{align*}
\end{corollary}

\begin{proof}
By \eqref{E:Lipschitz}, we have
$$
[\tau_B - s, \tau_B + s] \sset [\tau_B' - s - L^2 d(B, B'), \tau_B' + s + L^2 d(B, B')].
$$
Any particle $a \in H_{B'}^{++}$ is contained in the block $B'$ at some point in the interval $[\tau_{B'}, \tau_{B'} + \xi]$.
Therefore for \eqref{E:LB-event} to fail, there must be some particle $a \in H_{B'}^{++}$ whose path in the time interval $[-s - L^2 d(B, B'), s + L^2 d(B, B')]$ has diameter at least $d(B, B')/2$. This particle would have to travel distance $d(B, B')/4$ from its location at time $0$ in the time interval $[-s - L^2 d(B, B'), s + L^2 d(B, B')]$, implying $\cL_B(B', s)^c$.
Lemma \ref{L:far-particles} gives the quantitative bound.
\end{proof}

While Lemma \ref{L:global-local-ests} and Corollary \ref{C:no-far-particles} do not allow us to rule out all global effects on the local behaviour of the SIR process $X_t$, they do allow us to simplify our study to only consider blocks within a $(d_\sB(B, 0)^{1/6} + m)$-radius of $B$ for some constant $m$ (here $0$ is the block containing $(0,0)$). The following corollary is immediate from these results, and sets up the global event $\cG_\nu$ in Theorem \ref{T:main-2}.

\begin{corollary}\label{L:global-is-rare}
For any density $\mu > 0$ and $\nu < \nu^-$, as long as $m =m_\nu \in \N$ is large enough the following event has positive probability.
	$$
	\cG_\nu := \cC \cap \bigcap_{B \in \sB} \lf(\cD_B(d_\sB(0, B)^{1/6} + m) \cap \bigcap_{B' : d_\sB(B, B') \ge m + d_\sB(0, B)^{1/6}}\cL_B(B', [d(B, B')]^{3/2}) \rg).
	$$
	Moreover, by letting $m_\nu \to \infty$ as $\nu \to 0$ we can guarantee that $\P[ \cG_\nu] \to 1$ as $\nu \to 0$.
\end{corollary}

Moving forward, we set $m_B = m + d_\sB(0, B)^{1/6}$ to be the locality radius for $B$ and assume $\nu$ is small enough so that $\P[\cG_\nu] \ge 1/2$. 

\section{Upper bound on the density of remaining susceptible particles}
\label{S:upper-bd-density}

In this section, we show that the density of infected and susceptible particles in a region decreasing dramatically shortly after the colouring process passes through. When combined with the estimate in Proposition \ref{P:death-with-specifics}, this will later allow us to conclude that the infection dies out locally after the colouring process passes through (see Section \ref{S:herd}).

\subsection{Local infection spread}

Fix any block $B$.
We first establish that the following holds for some small fixed $\de > 0$: on the event $\cG_\nu$,
\begin{equation}
\label{E:goal-delta}
\text{all particles in $H_B$ are infected by time $L^{2+\delta} + \tau_B$ and heal by time $L^7$}
\end{equation}
with probability superpolynomial in $L$.

We will construct a collection of events whose intersection with $\cG_\nu$ implies \eqref{E:goal-delta}. The first of these events, $\cJ_B$, will be a long-range event controlling the coupled SSP, and the remaining events will have short range, only depending on the $\sig$-algebras $\sM_{B'}$ for $B'$ in a small radius of $B$. Throughout this section, all bounds hold for sufficiently large $L$ given a fixed $\de > 0$.  

We let the block $B$ be our frame of reference. Setting some notation for the section, let $\sC = \sC_B := D_\sB(B, 4 L^\de), \sC^- = \sC^-_B := D_\sB(B, 2 L^\de)$.
%

Let $\cJ_B$ be the event
\[
\cJ_B = \Big\{\forall B'\in \sC^-: \{B'\} \cap [\fD(B', L^{\delta/20})] = \{B'\} \cap [\fD(B', m_{B'})]\Big\}.
\]
Using the notation of Lemma \ref{L:global-local-ests}, we have $\cJ_B = \bigcap_{B'\in \sC^-} \cD_{B'}(L^{\delta/20}, m_{B'})$. By equation~\eqref{E:transitivity}, on the event $\cG_\nu \cap \cJ_B$ the components of $\sC^- \cap [\fD]$ agree with a local approximation of radius $L^{\delta/20}$.  
By Lemma~\ref{L:global-local-ests} we have that
\begin{equation}
\label{E:cJB-bound}
\P[\cJ_B] \geq 1 - \exp(-L^{c\delta/\log\log L}).
\end{equation}
We move on to defining the short-range events.
Our next event implies that every local neighbourhood of radius $L^{\delta/15}$ blocks has density at least $\frac23$ of red blocks. We set
\[
\cL_B = \Big\{\forall B'\in \sC^-: |\{B''\subset D_\sB(B', L^{\delta/15}) : B''\in [\fD(B'', L^{\delta/20})]\}| \leq \frac13 L^{2\delta/15}\Big\}
\]
and have the following strong estimate.
\begin{lemma}
\label{E:cLB-bound}
We have $
\P[\cL_B] \geq 1 - \exp(-cL^{\delta/30})$.
\end{lemma}

\begin{proof}
	For each block $B''$ we have that $\P(B''\in [\fD(B'', L^{\delta/20})]) \leq \frac14$ for large enough $L$ by Corollary \ref{C:blue-domination} and \eqref{E:Mrrk} in Lemma \ref{L:largest-scale}. Moreover, any collection of events 
$$
\{B_i \in [\fD(B_i, L^{\delta/20})]\}, \quad i = 1, \dots, k
$$
are independent if $d_\sB(B_i, B_j) > 2L^{\delta/20}$ for all $i \ne j$.
Applying a standard concentration estimate, see Lemma~\ref{l:dependentPerc} immediately below, yields the bound.
\end{proof}

\begin{lemma}\label{l:dependentPerc}
	Let $p > 0$, and let $\{Y_x\}_{x\in \Z^2}$ be a field of Bernoulli$(p_x)$ random variables where $p_x \le p$ for all $x$. Suppose that $\{Y_x\}_{x\in A}$ are independent for any finite set $A\subset \Z^2$ whose pairwise distances are at least $R$.  Then for any set $\Lambda$ and any $y>2p|\Lambda|$,
	\[
	\P[\sum_{x\in \Lambda} Y_x \geq y] \leq R^2 \exp(-c y^2/(R^2|\Lambda|)).
	\]
	Similarly, if $p_x \ge p$ for all $x$ then for any set $\Lambda$ and any $y< p|\Lambda|/2$ we have
	\[
	\P[\sum_{x\in \Lambda} Y_x \leq y] \leq R^2 \exp(-c y^2/(R^2|\Lambda|)).
	\]
\end{lemma}

\begin{proof}
		We only prove the first inequality as the second is similar. For $(j_1,j_2)\in \{0,\ldots,R-1\}^2$ let 
	\[
	\Lambda_{j_1,j_2} = \{(x_1,x_2)\in\Lambda: x_1 \equiv j_1 \ \hbox{mod} \ R, x_2 \equiv j_2 \  \hbox{mod} \ R\}
	\]
	If $\sum_{x\in \Lambda} Y_x \geq y$ then there must exist some $(j_1,j_2)$ such that
	\[
	\sum_{x\in \Lambda_{j_1,j_2}} Y_x \geq y\lf(\frac{1}{4R^2} + \frac{3|\Lambda_{j_1,j_2}|}{4 |\Lambda|}\rg).
	\]
	By Hoeffding's inequality and using that $\E \sum_{x\in \Lambda_{j_1,j_2}} Y_x \le |\Lambda_{j_1,j_2}| p \le \tfrac{|\Lambda_{j_1,j_2}|y}{2 |\Lambda|} $ we have
	\begin{align*}
	\P[\sum_{x\in \Lambda_{j_1,j_2}} Y_x &\geq |\Lambda_{j_1,j_2}|y\lf(\frac{1}{4R^2 |\Lambda_{j_1,j_2}|} + \frac{3}{4 |\Lambda|} \rg)] \\
	& \le \exp\lf(-2 |\Lambda_{j_1,j_2}| y^2\lf(\frac{1}{4R^2 |\Lambda_{j_1,j_2}|} + \frac{1}{4 |\Lambda|} \rg)^2 \rg) \le \exp\lf(- c \frac{y^2}{R^2 |\Lambda|}\rg).
	\end{align*}
	The final inequality follows by minimizing the exponent in $|\Lambda_{j_1,j_2}|$. The lemma follows by a union bound.
\end{proof}
On the event $\cG_\nu \cap\cJ_B\cap\cL_B$, for every $B' \in \sC^-$, each ball $D_\sB(B', L^{\de/15})$ is coloured at most one-third blue. All these balls are completely contained in the larger set $\sC$. Our next event controls the movement of particles from this larger set.
Define
\[
\cN_B:= \bigcap_{ B'\in \sC} \bigcap_{ a\in H_{B'}^{++}} \Big\{  |a(t)-a(0)| \leq L^{1 +\frac{\delta}{10}}\lceil t/L^2 \rceil^{\frac12+ \frac{\delta}{5}} \quad \forall t \in [0, L^{3}]\Big\}.
\]
With notation as in Lemma \ref{L:far-particles}, we have
\begin{equation*}
\cN_B \supset \bigcap_{ B'\in \sC} \bigcap_{i=1}^L \cE_B(0,i L^2, (iL^2)^{1/2 + \de/10}),
\end{equation*}
so by that lemma and a union bound we have
\begin{equation}
\label{E:NB-bound}
\P[\cN_B] \geq 1 - \frac12 \exp(-cL^{\delta/3})
\end{equation}
for large enough $L$.
If $B' \in \sC$ and $B'' \notin D_\sB(B', L^{3\de/5})$ then $d(B, B'') \ge L^{1 + 3\delta/5}/5$ by \eqref{E:block-distance}. The Lipschitz bound \eqref{E:Lipschitz} then implies that on the event $\cN_B$ a particle in $H_B^{++}$ will not reach $B''$ before time $\tau_B'' \wedge (\tau_{B'} + L^3)$ and thus cannot become the ignition particle for $B''$ prior to time $\tau_{B'} + L^3$.  It follows that at most $L^{3\de}$ particles in $\bigcup \{H_{B'}^{++} : B' \in \sC\}$ become ignition particles prior to time $\tau_{B'} + L^3$.

A significant difficulty in checking when particles near $B$ become infected given only the information in in $\sM_{\sC}$ is that this does not tell us the relative times of the $\tau_{B'}$.  To accommodate this we will discretize the relative times $\tau_B-\tau_{B'}$ and take a union bound over this discretization.  Similarly, we don't know exactly which blocks are blue so we will take a union bound over the possibilities.

For a block $B$  let $\sU=\sU_B=\{0,1\}^{\sC}\times \Z^{\sC}\times (\Z^2)^{\sC}$. 
 Let
\begin{align*}
\sU^* = \sU^*_B =\bigg\{(r,s,x)\in \sU&: \max_{B'\in \sC}  |s(B')| \leq 9L^{\delta}\xi, \qquad x(B') \in (B')^{\#} \; \forall B' \in \sC, \\
&\sum_{B'' \in D_\sB(B', L^{\delta/15})} r(B'') \geq \frac{1}{3} L^{2\delta/15} \qquad \forall B'\in \sC^-
\bigg\}.
\end{align*}
We will be interested in the random variable $U_B\in \sU$ defined by $U_B=(r_B,s_B,x_B)$ where $r_B(B'')$ is the indicator that block $B'$ is red, $x_B(B')$ is the ignition location for $B'$ and $s_B(B') = \lfloor \tau_{B} - \tau_{B'} \rfloor$.  

Since the colouring process spreads at least at rate $\xi$ (see \eqref{E:first-lipschitz}) we have that $\max_{B'\in \sC^-}  |s_{B}(B')| \leq 9L^{\delta}\xi$ and $x_B(B') \in (B')^\#$ for all $B'$ by construction. Moreover, the condition on $r$ required for membership in $\sU^*$ holds for $r_B$ on the event $\cG_\nu \cap \cJ_B\cap\cL_B$. Therefore on this event, $U_B\in \sU^*$. Note that
\begin{equation}
\label{E:cLde}
|\sU^*| \leq L^{c L^{\delta}}.
\end{equation}  
For $u=(r,s,x)\in \sU_B^*$ we say a particle $a\in H_{B}^+$ is \textbf{covered} by a particle $a'\in Q_{B',x(B')}$ with respect to $u$ if $r(B')=1$ and there exists $t\in[\frac12 L^{2+\delta}, L^{2+\delta}-1]$ and $y\in\Z^2$ such that $a(t)=y$ and $a'(t' + s(B'))=y$ for all $t'\in [t,t+1]$.  By construction, if $(r,s,x)=(r_B,s_B,x_B)$ and $a$ is covered by $a'$ for $t$ and $x$ then $\oa(t+\tau_B)=\oa'(t+\tau_B)=y$ since $a'(t')=y$ for all 
\[
t'\in \Big[t  + \lfloor \tau_{B} - \tau_{B'} \rfloor + \tau_{B'},t  + \lfloor \tau_{B} - \tau_{B'} \rfloor + \tau_{B'} +1\Big],
\]
which includes $t+\tau_B$.  We will write $\cW_{B,u,a}$ for the event that $a\in H_B^+$ is covered by at least $L$ particles in $\bigcup_{B' \in \sC} Q_{B', x(B')}$  with respect to $u\in \sU_B^*$.  We ask that $L$ particles cover $a$ rather than just 1 since some of the particles may become ignition particles and then not follow their original path.  However, recall on the event $\cN_B$ at most $L^{3 \delta}$ particles will become ignition particles prior to time 
$$
\inf_{B' \in \sC} \tau_{B'} + L^3 \ge \tau_B + L^3/2.
$$
Therefore some of the $L$ particles that cover $a$ will not change their original path prior to time $\tau_B + L^3/2$.
\begin{lemma}
\label{L:WBua-bd}
Defining
$$
\cW_B = \bigcap_{u=(r,s,x)\in \sU^*} \Big(\bigcap_{a\in H_B^+} \cW_{B,u,a} \Big) \cup \cT_{B, u}^c,
$$
where $\cT_{B, u} := \{\forall B'\in \sC : r(B') \leq \mathbf{1}(\cA_{B'})\}$,
we have
$
\P[\cW_B^c, \cN_B] \leq \exp(-c L^2)$.
\end{lemma}

\begin{proof}
We will establish the result by a union bound over $u$ and $a$.  Let $N$ be the number of particles $a' \in \bigcup_{B' \in \sC} Q_{B', x(B')}$ which cover a fixed $a \in H_B^+$ and define
\[
R= \sum_{\substack{B'\in \sC:\\r(B')=1}} \sum_{a'\in Q_{B',x(B')}} \int_{\frac12 L^{2+\delta}}^{L^{2+\delta}-1} \mathbf{1}\Big(a'(t' + s(B'))=a(t) \;\; \forall t'\in[t,t+1] \Big) dt.
\]
Let $\sI$ denote the $\sigma$-algebra generated by all $\{\cA_B\}_{B'\in \sC^-}$, all particles in $\cup_{B'\in \sC^-} H_{B'}^+ \cup W_{B, \operatorname{ig}}$ up to time $L^2$ and the particle $a(t)$ up to time $L^{2+\delta}$:
\[
\sI=\sigma\Big\{\{\cA_B\}_{B'\in \sC^-}, \{a(t)\}_{t\in [0,L^{2+\delta}]}, \{a'(t)\}_{a'
	\in \cup_{B'\in \sC^-} H_{B'}^+ \cup W_{B, \operatorname{ig}}, t\in [0,L^{2}]} \Big\}.
\]
The future trajectories of the $a'(t)$ after time $L^2$ are simple random walks independent of $\sI$, so conditional on $\sI$, we are in the setting of Lemma~\ref{l:hittingNumber} and Remark \ref{rem:hittingNumberConcentration}. This gives that
\begin{equation}
\label{E:RcI}
\P[N \leq \frac{c}{\log L} \E[R \mid \sI]\mid \sI] \leq \exp(-\frac{c'}{\log L} \E[R \mid \sI])
\end{equation}
for constants $c, c' > 0$. Next, we estimate $\E[R \mid \sI]$.

For this, let
\[
Z_{B,u,a} = \mathbf{1}(\cT_{B, u} \cap \cN_B^*).
\]
where $\cN_B^* := \{\P[\cN_B \mid \cI] > 0\}$ is the part of the probability space where information from $\cI$ does not preclude the set $\cN_B$. The indicator $Z_{B,u,a}$ is $\sI$-measurable.

Now fix $t\in [\frac12 L^{2+\delta}, L^{2+\delta}-1]$.
Let $B_t$ denote the $\sI$-measurable block containing the location $a(t)$. On the event $\cN_B^*$, the movement of $a$ is restricted and so we necessarily have $B_t \in \sC^-$.
Therefore on $\cN_B^*$, we have that $D_\sB(B_t, L^{\de/15}) \sset \sC$, and so for any $B' \in D_\sB(B_t, L^{\de/15})$ any $a' \in Q_{B',x'(B')}$, on $\cN_B^*$ the value of $a'(L^2)$ must be within distance $2L^{1 + \de/5}$ of $B_t$. Hence
\[
d(a'(L^2), a(t)) \le 3L^{1 + \de/5}.
\]
Therefore for any such $a'$ we have 
\begin{align*}
\P[a'(t' + s(B'))=a(t) \quad \forall t'\in[t,t+1] \mid \sI] &= e^{-1} \P[a'(t + s(B'))=a(t) \mid \sI] \\
&\ge c L^{-2-\delta}Z_{B,u,a}.
\end{align*}
Since $u \in \sU^*$ and $B_t \in \sC^-$, at least one-third of the blocks $B' \in D_\sB(B_t, L^{\de/15})$ have $r(B')=1$.  If we also have that $\cT_{B, u}$ holds then $\mathbf{1}(\cA_{B'})=1$ for all these blocks.
Since $|Q_{B',x(B')}|\geq L^{2-(\log L)^{-1/2}}$ for such blocks, we have
\begin{align*}
\E\Big[\sum_{\substack{B'\in \sC:\\r(B')=1}} \sum_{a'\in Q_{B',x}}  &\mathbf{1}(a'(t' + s(B'))=a(t) \quad \forall t'\in[t,t+1] ) \mid \sI\Big] \\
&\geq c' L^{-\de + 2\de/15 -(\log L)^{-1/2}} Z_{B,u, a}
\end{align*}
and so upon integrating we get that
\[
\E[R\mid \sI] \geq c L^{2+2\delta/15 -(\log L)^{-1/2} }Z_{B,u, a} \geq L^{2} \log(L) \ Z_{B,u, a}.
\]
Therefore by \eqref{E:RcI}, we have
\[
\P[\cW_{B,u,a}^c,Z_{B,u, a}=1 \mid \sI]\leq \exp(-c L^2)
\]
and so unconditionally
\[
\P[\cW_{B,u,a}^c \cap \cT_{B, u} \cap \cN_B^*]\leq \exp(-c L^2).
\]
Taking a union bound over $u, a$ and using that $\cN_B \sset \cN_B^*$ we have that
$
\P[\cW_B^c, \cN_B] \leq \exp(-c L^2)
$
which completes the proof.
\end{proof}

On the event $\cW_B$ we must have $\cW_{B, U_B, a}$ for all ${a\in H_B^+}$ and hence all such $a$ are covered. 

Next, we write $\cK_{B}$ for the event that every particle in $H_B^+ \cup W_{B, \operatorname{ig}}$ gets a healing event in the interval $[\frac12 L^7, L^7]$. This is the final event we need to ensure all particles in $H_B$ are removed. The following lemma then summarizes the results of this section.

\begin{lemma}
\label{L:final-event}
On the event $\cG_\nu \cap\cJ_B\cap\cL_B\cap\cN_B \cap \cW_B \cap \cK_B$ every particle in $H_B \cup W_{B, \operatorname{ig}}$ is removed by time $\tau_B+L^7$. Moreover, we have the following probability bounds, in addition to the bound \eqref{E:cJB-bound} on $\P [\cJ_B]$.
\begin{align}
\label{eq:cLNWbound}
\P[\cL_B\cap\cN_B \cap \cW_B] &\geq  1 - \exp(-cL^{\delta/30}) \\
\label{eq:cKbound}
\P[\cK_B] &\geq 1-\exp(-\frac13 L).
\end{align}
\end{lemma}

\begin{proof}
On the event $\cG_\nu\cap \cJ_B\cap\cL_B\cap\cN_B \cap \cW_B$ every particle in $B$ is infected by time $\tau_B + L^{2+\delta}$ since each particle in $H_B^+$ will collide with a particle in some $Q_{B',x_B(B')}$ from a red block $B'$ while it is still infected by the discussion prior to Lemma \ref{L:WBua-bd}, so if we also intersect with $\cK_B$ then every particle in $H_B \cup W_{B, \operatorname{ig}}$ will be removed by time $\tau_B+L^7$. The bound \eqref{eq:cLNWbound} follows from Lemma \ref{E:cLB-bound}, \eqref{E:NB-bound}, and Lemma \ref{L:WBua-bd}. Finally, $|H_B^+|$ is stochastically dominated by a Poisson with mean $\frac32 \mu L^2$ by Lemma~\ref{l:HBsize} and each particle receives a healing event in $[\frac12 L^7, L^7]$ with probability at least $1-\exp(-\frac12 L)$ by the relationship \eqref{E:Lnu-relationship}. This yields \eqref{eq:cKbound}.
\end{proof}

\subsection{Events to bound the number of susceptible particles}
\label{S:upper-bound-s}
Having established that particles in most blocks $B$ become infected soon after $\tau_B$ we next want to establish that the density of susceptible or infected particles in a given region near $B$ remains low. Concretely, we will prove the following estimate.

\begin{prop}
	\label{P:big-estimate}
	Fix a block $B$ and let $s \ge L^{10}$. Let $z_B$ be the square closest to the center of $B$ and for $z \in \Z^2$ define $\Lambda_z := D(z + z_B, s^{49/100})$. Then there is an event $\cX = \cX_{B, s}$ such that:
	\begin{itemize}
		\item  $\P[\cX] \ge 1 - \exp(-s^{c/\log \log s})$ and $\cX$ is $\sM_{D_{\sB}(B, m_B + 3s^{3/4})}$-measurable. 
		\item On the event $\cX \cap \cG_\nu$, we have
		\begin{equation}
		\label{E:bara-sum}
		\frac{1}{|\Lambda_z|} \sum_{B' \in D_\sB(B, s^{3/4})} \sum_{a \in H_{B'}^{++}} \mathbf{1}(\bar a(\tau_B + s') \in \Lambda_z, \bar a \notin \mathbf{R}_{\tau_B + s'}) \le L^{-20}
		\end{equation}
		for all $s' \in [s-1, s]$ and $z \in \Z^2$.
	\end{itemize}
\end{prop}

To shorten notation, we write $\sD = D_\sB(B, s^{3/4})$ and define
\[
\sD_1=\{B'\in \sD: \cJ_{B'}^c\}
\]
and
\[
\sD_2 = \{B'\in \sD: (\cL_{B'}\cap\cN_{B'} \cap \cW_{B'}\cap \cK_{B'})^c\}.
\]
Lemma \ref{L:final-event} shows that for  $B'\in \sD \setminus(\sD_1 \cup \sD_2)$ all the particles in $H_{B'} \cup W_{B, \operatorname{ig}}$ are removed by time $\tau_{B'} + L^{7} \leq \tau_B + s$.  Thus in the sum in \eqref{E:bara-sum} we need only consider particles from blocks in $\sD_1 \cup \sD_2$. For small $s$, this observation already allows us to prove Proposition \ref{P:big-estimate}.

\begin{proof}[Proof of Proposition \ref{P:big-estimate} for $L^{10} \le s \le L^{5000}$]
In this case, we let $\cX = \{\sD_1 \cup \sD_2 = \emptyset\}$. On the event $\cX \cap \cG_\nu$, the left hand side of \eqref{E:bara-sum} is $0$ for all $z$.  For $B' \in D_\sB(B, s^{3/4})$ the events $\{B' \in \sD_1\}, \{B' \in \sD_2\}$ are measurable given $$
\sM_{D_\sB(B', m_{B'} + 10 L^\de)} \sset \sM_{D_\sB(B, m_{B'} + 2s^{3/4})} \sset \sM_{D_\sB(B, m_{B} + 3s^{3/4})}
$$
by the definitions of $\cJ_B, \cL_B, \cN_B, \cW_B,$ and $\cK_B$. It remains to check the probability bound on $\cX$. By \eqref{E:cJB-bound}, Lemma \ref{L:final-event}, and a union bound, we have 
\begin{equation*}
\P[|\sD_1 \cup \sD_2|  \geq 1] \leq L^{10000}\exp(-L^{c'\delta/\log \log L}) \leq  \exp(-s^{c/\log \log s})
\end{equation*}
as long as $c$ is sufficiently small given $c'$, where in the final inequality we have used that $s \le L^{5000}$.
\end{proof}

For the remainder of this section we assume $s \ge L^{5000}$. In this case, Proposition \ref{P:big-estimate} will be proven by using the estimates in \eqref{E:cJB-bound} and Lemma \ref{L:final-event} applied to different blocks $B'$ in a moderate radius around $B$. 
To have a concentration bound that improves with $s$, 
we will use the fact that all the events in Lemma~\ref{L:final-event} (aside from $\cG_\nu$ itself) are defined locally, along with the concentration estimate in Lemma \ref{l:dependentPerc}.

First, with notation as in Lemma \ref{L:far-particles}, define
\begin{equation}
\label{E:cX'}
\cX' = \bigcap_{B' \in D_\sB(B, s^{3/4})} \cE_{B'}(0, 2s, s/4).
\end{equation}
By Lemma \ref{L:far-particles} and a union bound we have
\begin{equation}
\label{E:X'-bound}
\P[\cX'] \ge 1- Cs^{3/2} L^2 \exp (- cs) \ge 1 - \exp(-s^{c/\log \log s})
\end{equation}
and by construction and the Lipschitz bound \eqref{E:first-lipschitz}, on $\cX'$ we have
$$
\sum_{B' \in D_\sB(B, s^{3/4})} \sum_{a \in H_{B'}^{++}} \mathbf{1}(d(\bar a(\tau_B + s'), B) > s) = 0
$$
for all $s' \in [s-1, s]$.
In particular, on $\cX'$, the inequality \eqref{E:bara-sum} holds for all $z \ge 2s$. Therefore by a union bound, to complete the proof it suffices to show that for every $z$ with $|z| \le 2s$, we can define a $\sM_{D_{\sB}(B, m_B + s^{3/4})}$-measurable event $\cQ_z$ satisfying $
\P[\cQ_z] \ge 1 - \exp(-s^{c/\log \log s})
$ and such that \eqref{E:bara-sum} holds on $\cQ_z \cap \cG_\nu$ for all $s' \in [s-1, s]$ for that particular $z$.

To count how many particles that originated in blocks contained in $\sD_1 \cup \sD_2$ are close to $\Lambda_z$ at time $\tau_B + s$, we will obtain concentration estimates on the density of $\sD_1 \cup \sD_2$ on different geometric scales. Define $j_{{\max}} = \frac32 \log_2 (sL^{-1})+2$ and for $i\in\{1,2\}$ and $1\leq j \leq j_{{\max}}$ define
\[
D_{i,j}= \{B' \in \sD_i: B' \in D_\sB(B_z ,  2^j s^{1/2} L^{-1}) \},
\]
where $B_z$ is the block containing the reference vertex $z + z_B$.
We will let
\[
\cR_{i,j} = \{|D_{i,j}| \leq L^{-1000} 4^j s\}, \qquad \cR =\bigcap_{j=1}^{j_{{\max}}} (\cR_{1,j}\cap \cR_{2,j})
\]
and will establish the following bound.
\begin{lemma}
\label{eq:sumCSbound}
For any $s \ge L^{5000}$ and any $z \in \Z^2$ with $|z| \le 2s$ we have 
$$
\P[\cR] \geq 1 - \exp(-s^{c'/\log \log s}).
$$
\end{lemma}

\begin{proof}  
Events $\{B_i'\in \sD_2\}, i = 1, \dots, k$ are independent as long as $d_\sB(B_i', B_j') \ge 10 L^\de$ whenever $i \ne j$. Hence by Lemma~\ref{l:dependentPerc} we have
\begin{align}\label{eq:sumCS2bound}
\P[\cR_{2,j}^c] &\leq 100 L^{2 \de} \exp(-c L^{-2000}  4^j s L^{-2\delta}).
\end{align}
To estimate $\P[\cR_{1,j}^c]$ we can first use \eqref{E:transitivity} and the definition of $\cJ_B, \sD_1$ to write
\begin{align}
\nonumber
&\{B'\in \sD_1\} = \bigcup_{B'' \in \sC_{B'}^-} [\cD_{B''}(L^{\de/20}, m_{B''})]^c = \cK^1_{B'} \cup \cK^2_{B'} \text{ where }\\
\nonumber
\qquad &\cK_{B'}^1 = \bigcup_{B'' \in \sC_{B'}^-} [\cD_{B''}(s^{1/10}, m_{B''})]^c,  \quad \cK_{B'}^2 = \bigcup_{B'' \in \sC_{B'}^-} [\cD_{B''}(L^{\de/20}, s^{1/10})]^c.
\end{align}
From here we can use Lemma \ref{L:global-local-ests} and the fact that the events $\cK_{B'}^2$ are independent for blocks that are distance at least $3 s^{1/10}$ apart to bound $\P[\cR^c_{2, j}]$. Indeed, by combining Lemma \ref{L:global-local-ests}, Lemma~\ref{l:dependentPerc}, and a union bound we have
\begin{align*}
\P[\cR^c_{2, j}] &\le \P\Big(\sum_{B' \in D_\sB(B_z ,  2^j s^{1/2} L^{-1})} \mathbf{1}\lf(\cK_{B'}^1\rg) > 0 \Big) +  \P\Big(\sum_{B' \in D_\sB(B_z ,  2^j s^{1/2} L^{-1})} \mathbf{1}\lf(\cK_{B'}^2\rg) > L^{-2000} 4^j s \Big) \\
&\le 4^j s L^{-2} \exp(- c s^{c/\log \log s}) + 9s^{1/5} \exp(-c L^{-2000}4^j s/s^{1/5}).
\end{align*}
Combining this bound with \eqref{eq:sumCS2bound}, summing over $j$ and using that $s \ge L^{5000}$ yields the desired bound.
\end{proof}

\begin{corollary}
	\label{C:cY}
Fix $s \ge L^{5000}$ and $z \in \Z^2$ with $|z| \le 2s$ and define 
\begin{equation}
\label{E:Ydef}
\cY_{j} = \{\sum_{B'\in \sD_{1,j}\cup \sD_{2,j}} |H_{B'}^{++}| \leq L^{-800} 4^j s\}, \qquad \cY =\bigcap_{j=1}^{j_{{\max}}} \cY_{j}.
\end{equation}
Then
\begin{equation*}
\P[\cY^c] \leq \exp(-s^{c/\log \log s}).
\end{equation*} 
\end{corollary}

\begin{proof}
We will estimate the probability of $\cY_{j}$ given $\cR_{1,j}\cap \cR_{2,j}$.  
On $\cR_{1,j}\cap \cR_{2,j}$ there are at most $L^{-1000} 4^j s$ blocks $B'$ in $D_{1,j}\cup D_{2,j}$.  There are at most ${4^j s \choose L^{-1000} 4^j s}$ choices of these blocks from the blocks in $D_\sB(B_z ,  2^j s^{1/2} L^{-1})$.  Given a deterministic selection of $L^{-1000} 4^j s$ blocks $B'$, the total number of particles in $H_{B'}^{++}$ is stochastically dominated by Poisson random variable with mean $C L^{2-1000} 4^j s$ by Lemma \ref{l:HBsize} and so
\begin{align*}
\P[\cY_{j}^c,\cR_{1,j}\cap \cR_{2,j}] &\leq {4^j s \choose L^{-1000} 4^j s} \P[\hbox{Pois}(C L^{2-1000} 4^j s) > L^{-800} 4^j s]\\
&\leq \exp(- L^{-800} 4^j s).
\end{align*}
Hence by a union bound,
\begin{equation}\label{eq:sumCYbound}
\P[\cY^c,\cR] \leq \exp(-s^{c/\log \log s}).
\end{equation}
The result then follows from \eqref{eq:sumCYbound} and Lemma \ref{eq:sumCSbound}.
\end{proof}

We will use Corollary \ref{C:cY} to control the event
\[
\cQ'_z =\bigg\{ \sup_{s' \in [s-1, s]} \sum_{B'\in \sD_1 \cup \sD_2}\sum_{a\in H_{B'}^{++}} \mathbf{1}(\oa(\tau_B+s') \in \Lambda_z) \leq L^{-20} |\Lambda_z| \bigg \}.
\]
Here we again face the issue that the functions $\oa$ depends implicitly on the non-local differences $\tau_B - \tau_{B'}$.  By \eqref{E:first-lipschitz} we know that $|\tau_B - \tau_{B'}|\leq \xi s^{3/4}$ so we let $G_{a}$ be the event that for a particle $a \in H_{B'}^{++}$ we have $a(t')\in \Lambda_z$ for some time $t'$ with $|t'-s|\leq  \xi s^{3/4} + 1$. The event $G_{a}$ must hold if a particle is in $\Lambda_z$ at time $\tau_B+s'$ for some $s' \in [s-1, s]$.  Hence we define
\begin{equation}
\label{E:Qdef}
\cQ_z =\bigg\{\sum_{B'\in \sD_1 \cup \sD_2} \sum_{a\in H_{B'}^{++}} \mathbf{1}(G_{a}) \leq L^{-20} |\Lambda_z| \bigg \}.
\end{equation}
With this definition, $\cQ'_z \subset \cQ_z$, and we have the following estimate.
\begin{lemma}
\label{L:Q'-lemma}
For $s \ge L^{5000}$ we have $\P[\cQ_z] \ge 1 - \exp(-s^{c/\log \log s}).$
\end{lemma}  

\begin{proof}
	First, let $\cO$ be the event that no particle from $\sD$ moves more than distance $s^{5/12}$ by time $s^{4/5}$:
	\[
	\cO:=\bigcap_{B'\in \sD} \cE_{B'}(0, s^{4/5}, s^{5/12})
	\]
	which by Lemma \ref{L:far-particles} and a union bound holds with probability
	\begin{align}\label{eq:CObound}
	\P[\cO] &\geq 1 - C s^{3/2} L^2 \exp(-c s^{1/30}).
	\end{align}
By \eqref{eq:CObound} and Corollary \ref{C:cY}, it suffices to show that 
\begin{equation}\label{eq:cQ}
\P[\cQ^c,\cY,\cO] \leq \exp(-s^{c/\log \log s}).
\end{equation}
Now let $a \in H_{B'}^{++}$ for some $B' \in \sD$. It follows from our construction that $a(s^{4/5} + t)-a(s^{4/5}), t \ge 0$, the trajectory of $a$ after time $s^{4/5}$, is independent of $\sD_1, \sD_2, \cY$ and $\cO$. In particular, if we define the annulus 
$$
V_j=  D_\sB(B_z ,  2^j s^{1/2} L^{-1}) \smin D_\sB(B_z ,  2^{j-1} s^{1/2} L^{-1})
$$
then for $a\in B' \subset  V_j$, by a random walk estimate similar to Lemma \ref{L:rw-estimate} we have
\[
\P[G_{a}\mid \sD_1, \sD_2, \cY,\cO] \leq \exp(-c 2^{j})|\Lambda_z|s^{-1}
\]
since conditional on $\cO$ the path of $a$ travels at most distance $s^{5/12}$ by time $s^{4/5}$ and after that it is a random walk that needs to travel distance $c2^j s^{1/2}$ in time $s-o(s)$ to reach $\Lambda_z$.  Note that the window of time $\xi s^{3/4}+1$ in the definition of $G_{a}$ is
lower order compared to the square of the side-length of $\Lambda_z$ so this window of time will only affect the random walk estimate by a constant factor.
On the event $\cY$ there are at most $L^{-800} 4^j s$ particles in these blocks so
\begin{align*}
\P&[\sum_{B' \subset   V_j }\mathbf{1}(B'\in \sD_1 \cup \sD_2) \sum_{a\in H_{B'}^+} \mathbf{1}(G_{a}) \geq 2^{-j} L^{-20} |\Lambda_z|,\cY,\cO]\\
&\leq \P[\hbox{Bin}(L^{-1500} 4^j s,\exp(-c 2^{j})|\Lambda_z|s^{-1}) \geq 2^{-j} L^{-20} |\Lambda_z|]\\
&\leq \exp(-s^{c/\log\log s}),
\end{align*}
which yields \eqref{eq:cQ} after a union bound over $j$.
\end{proof}

\begin{proof}[Proof of Proposition \ref{P:big-estimate} when $s \ge L^{5000}$]
We set
$$
\cX = \cX' \cap \bigcap_{z \in \Z^2, |z| \le 2s} \cQ_z.
$$
This satisfies the bound in the first bullet of Proposition \ref{P:big-estimate} by Lemma \ref{L:Q'-lemma} and satisfies \eqref{E:bara-sum} for $s' \in [s-1, s]$ and $|z| \ge 2s$ by construction. The $\sM_{D_\sB(B, m_B + 3s^{3/4})}$-measurability of  $\cX$ follows by tracing back the various definitions:
\begin{itemize}[nosep]
	\item Each event $\cE_{B'}, B' \in D_\sB(B, s^{3/4})$ used in the definition of $\cX'$ only depends on $\sM_{B'}$.
	\item Given $\sD_1, \sD_2$, the event $\cQ_z$ only depends on $\sM_{B'}, B' \in D_\sB(B, s^{3/4})$.
	\item For $B' \in \sD$ the events $\{B' \in \sD_1\}, \{B' \in \sD_2\}$ are measurable given $\sM_{D_\sB(B, m_{B'} + 3s^{3/4})}$ as in the $s \le L^{5000}$ case.
\end{itemize}
Finally, on the event $\cG_\nu$, for $|z| \le 2s$ we have
\begin{align*}
	\sum_{B' \in \sD} \sum_{a \in H_{B'}^{++}} \mathbf{1}(\bar a(\tau_B + s') \in \Lambda_z, \bar a \notin \mathbf{R}_{\tau_B + s}) \le &\sum_{B'\in \sD_1 \cup \sD_2}\sum_{a\in H_{B'}^{++}} \mathbf{1}(\oa(\tau_B+s') \in \Lambda_z),
\end{align*}
which on $\cQ_z \sset \cQ'_z$ is at most $L^{-20}|\Lambda_z|$ for all $s' \in [s-1, s]$. This yields \eqref{E:bara-sum} for $|z| \le 2s$.
\end{proof}

\section{Surviving particles}
\label{S:survival}

In this section we define a \emph{survival event} which will guarantee that a certain particle in a block $B$ survives long past when its block gets coloured. This is the first step in showing that with positive probability, the infection survives forever but a positive density of particles never become infected. 

For this section and the next, in addition to examining measurability of events with respect to the block $\sig$-algebras $\sM_B$, we will also examine measurability with respect to $\sig$-algebra $\sF_t$ generated by trajectories $\bar a(s), s \in (-\infty, t]$ for all particles $a \in H_B^{++}$ for some $B \in \sB$ with $\tau_B \le t$. In other words, $\sF_t$ is the natural filtration for the process of coloured particles.

The most important component of our survival event is an extremely rare event $\cU_B$ which is $\sM_{D_\sB(B, 2)}$-measurable. We have three main goals when defining this event:
\begin{itemize}[nosep]
	\item To specify the movements of a particle $a^* \in H_B$ that will have a good chance of surviving forever.
	\item To use the other particles in $H_B$ to build a protective wall of infected particles around $a^*$ that will ensure it is difficult for  particles from other blocks to reach $a^*$ before they have recovered.
	\item To ensure that when doing this, $B$ does not become a blue seed.
\end{itemize}
Naturally, the precise definition of $\cU_B$ must be somewhat technical. For the definition, we need to distinguish a square in $B$. Rather arbitrarily, we let $v_B$ be the square closest to the center of $B$, where ties are broken using the lexicographic order.

\begin{definition}
	\label{D:survivor}
For a block $B$ let $\cU_{B}$ be the event where the following things occur:
\begin{enumerate}
	\item \emph{Particle count:} $H^+_B = H^-_B$.
	\item \emph{Healing:} For every particle $a \in H^+_B$, we have $a^h \cap [0, L^{20}-2] = \emptyset$ and $a^h \cap [L^{20}-2, L^{20}-1] \ne \emptyset$.
	\item \emph{Ignition particles:} For every block $B'$ with $d(B', B) \le 2$, the ignition trajectory $W_{B', \operatorname{ig}}$ satisfies $W_{B', \operatorname{ig}}(s) = W_{B', \operatorname{ig}}(0)$ for all $0 \le s \le L^4$ and the ignition healing clock satisfies $W^h_{B', \operatorname{ig}} \cap [0, L^3] = \emptyset$ and $W^h_{B', \operatorname{ig}} \cap [L^3, L^4] \ne \emptyset$.
	\item \emph{The survivor:}
	There is some particle $a^* \in H^+_B$ with $a^*(t) = v_B$ for all $0 \le t \le L^{20}$.
	\item  \emph{Other trajectories:} We can partition $H_B^+ \smin \{a^*\}$ into sets 
	$$
	\{H_{B, B', x}^+ : x \in \del B^\#, d(B, B') = 1, B' \not\in U_x \}
	$$
	such that:
	\begin{enumerate}
		\item $|H_{B, B', x}^+| = L^{90}$ for all $x, B'$.
		\item Every particle $a \in H_{B, B', x}^+$ satisfies:
			\begin{align}
			\label{E:66}
			d(a(0), B^c) &\ge L/10 \\
			\label{E:67}
		\{a(t) : 0 \le t \le \xi/(2\log \log L) \} &= U_x \smin \{v_B\}, \\
		\label{E:68}
		\{a(t)  : t \in [0, L^{20}]\} &\sset D(v_B, L^{20}) \smin \{v_B\}.
		\end{align}
		\item The particle $a$ first exits $U_x$ in the interval $[\xi/(2\log \log L), \xi/\log \log L)$. When it does so, it enters $B'$.
	\end{enumerate}
	\item \emph{Protective wall:} For every $b \in D(v_B, L^{20}) \smin \{v_B\}$,
	there are $L^{45}$ particles in each $H_{B, B', x}^+$ that satisfy
	$$
	a(t) = b, \qquad \xi \le t \le L^{20}.
	$$
\end{enumerate}
\end{definition}

The protective wall in the final part of Definition \ref{D:survivor} will guarantee that it is hard for a particle starting in a block far away from $B$ to infect $a^*$. However, the protective wall is not effective at preventing infection from nearby blocks. Because of this, we need to define a different specific event to handle blocks $B'$ close to $B$. To consolidate later proofs, we will build this event in a similar way to $\cU_B$; however, the exact structure here is not nearly as important.

\begin{definition}
	\label{D:survivor-2}
	For blocks $B' \ne B$, let $\cV_{B', B}$ be the event where the following things occur:
	\begin{enumerate}
	\item \emph{Particle count and ignition trajectories:} Events $1$ and $3$ in Definition \ref{D:survivor} hold for the block $B'$.
	\item \emph{Healing:} For every particle $a \in H^+_{B'}$, we have $a^h \cap [0, L^{15}] = \emptyset$ and $a^h \cap [L^{15}, L^{15} + 1] \ne \emptyset$.
	\item \emph{Trajectories:} We can partition $H_{B'}^+$ into sets 
	$$
	\{H_{B', B'', x}^+ : x \in \del B^{\prime \#}, d(B'', B') = 1, B'' \not\in U_x \}
	$$
	such that:
	\begin{enumerate}
		\item $|H_{B', B'', x}^+| = L^{90}$ for all $x, B''$.
		\item Every particle $a \in H_{B', B'', x}^+$ satisfies:
		\begin{align}
		d(a(0), B^{\prime c}) &\ge L/10 \\
		\{a(t) : 0 \le t \le \xi/(2\log \log L) \} &= U_x \smin \{v_B\}, \\
		\label{E:71}
		\{a(t)  : t \in [0, 2L^{20}]\} &\sset D(v_B, L^{20}) \smin \{v_B\}.
		\end{align}
		\item The particle $a$ first exits $U_x$ in the interval $[\xi/(2\log \log L), \xi/\log \log L)$. When it does so, it enters $B'$.
	\end{enumerate}
	Note that this is essentially the same as Definition \ref{D:survivor}.5, except we have no survivor particle $a^*$, and the block $v_B$ is not located in $B'$.
\end{enumerate}
\end{definition}

We now define
$$
\cY_{B} = \cU_B \cap \bigcap_{B' \in D_\sB(B, L^{10})\smin \{B\}} \cV_{B, B'}.
$$
We have the following easy deterministic consequences of the construction of $\cY_B$.
\begin{lemma}
	\label{L:red-boxes}
	On the event $\cY_B$, the following events hold.
	\begin{enumerate}[label=(\roman*)]
		\item Every block $B' \in D_\sB(B, L^{10})$ is not a blue seed.
		
		\item Suppose that the block $B$ is ignited. For all times $t \in [\tau_B + \xi, \tau_B + L^{20} - 2]$ and all locations $b \in D(v_B, L^{20}) \smin \{v_B\}$, there is at least one infected particle satisfying $\bar a(t) = b$.
		
		\item Suppose that the block $B$ is ignited. Then by the time $\tau_B + L^{20}-1$, all particles in the set $S_B = \bigcup \{H_{B'} : B' \in D_\sB(B, L^{10}) \}$ have been removed, other than the survivor particle $a^*$ for the block $B$. (For this point, recall that the ignition particle for $B'$ belongs to $H_{B'}$ and we think of the event where a particle becomes an ignition particle for another block as removal).
		\item Suppose that the block $B$ is ignited. Then the survivor particle $a^*$ does not come into contact with any infected particles in the set $S_B$.
	\end{enumerate}
Moreover, the event $\cY_B$ is measurable with respect to either the $\sig$-algebra $\sM_{D_\sB(B, L^{10} + 2)}$ or the $\sig$-algebra $\sF_{\tau_B + 3L^{20}/2}$. 
\end{lemma}

\begin{proof}
	Throughout the proof, we work on the event $\cY_B$.
We start with (i). We refer the reader back to \eqref{E:cA123} and surrounding discussion for the detailed definition of a blue seed.
 Let $B' \in D_\sB(B, L^{10})$. There are two cases, depending on whether or not $B' = B$ -- their proofs are identical so we only treat the case when $B' = B$. Definition \ref{D:survivor}.3 guarantees the events $\cA_B^{(1)}$ and $\cA_B^{(3)}$.  Definition \ref{D:survivor}.2 guarantees the event $\cA_B^{(4)}$. Definition \ref{D:survivor}.5 and the fact that none of the ignition trajectories in block $B'$ with $D(B', B) \le 2$ move prior to time $L^3$ (Definition \ref{D:survivor}.3) guarantee that each of the events $\cA^{(2)}_{B, B', x}$ and $\cA^{(5)}_{B, x}$ hold for $x \in \del B^\#, B' \notin U_x$. Therefore $\cA_{B}$ holds, and so $B$ is not a blue seed.
	
We now prove (ii). If $B$ is ignited at a location $x \in \del B^\#$, then by  Definition \ref{D:survivor}.$3$ that square contains an infected particle for all times in $[\tau_B, \tau_B + L^3]$. Choose $B'' \notin U_x$ with $d(B, B'') = 1$ arbitrarily. Equation \eqref{E:67} in Definition \ref{D:survivor}.$5$ guarantees that all particles in $H^+_{B, B'', x}$ visit the site $x$ prior to time $\tau_B + \xi/(2 \log \log L)$, and equation \eqref{E:68} in Definition \ref{D:survivor}.$5$ guarantees that no more than $5 L^{40}$ of these particles become ignition particles prior to time $\tau_B + L^{20}$.
Definition \ref{D:survivor}.$2$ and Definition \ref{D:survivor}.$6$ then guarantee that at least $L^{45} - 5 L^{40}$ infected particles are at every site in $D(v_B, L^{20}) \smin \{v_B\}$ at all times $t \in [\tau_B + \xi, \tau_B + L^{20} - 2]$ and all locations $b \in D(v_B, L^{20}) \smin \{v_B\}$.

For part (iii), we first deal with particles from $H_B$. The ignition particle recovers by time $\tau_B + L^4$ by Definition \ref{D:survivor}.3. The remaining particles recover by time $\tau_B + L^{20} - 1$ if they are infected prior to that time by Definition \ref{D:survivor}.2. The fact that they are infected prior to that time follows from \eqref{E:68} and part (ii). Now let $B' \in D_\sB(B, L^{10})$. The ignition particle for $H_{B'}$ recovers prior to time $\tau_{B'} + L^4$, which is bounded above by $\tau_B+ L^{12}$ by \eqref{E:first-lipschitz}. For the remaining particles, observe that by \eqref{E:first-lipschitz}, we have
$$
[\tau_{B'}, \tau_{B'} + L^{12}] \cap [\tau_B + \xi, \tau_B + L^{20}] \ne \emptyset,
$$
and so by (ii) and \eqref{E:71}, all particles in $H_{B'}$ are infected by time $\tau_{B'} + L^{12}$. Hence by Definition \ref{D:survivor-2}.2 all particles have recovered by time   $\tau_{B'} + L^{15} \le \tau_B+L^{20} - 1$. 

For part (iv), by part (iii), we just need to check that $a^*$ does not contact any infected particles in $S_B$ in the interval $[\tau_B, \tau_B + L^{20}]$. For particles in $H_B \smin \{a^*\}$, this is guaranteed by \eqref{E:68} and Definition \ref{D:survivor}.3 for the ignition particle. For particles in $S_B \smin H_B$  this is guaranteed by \eqref{E:71}, Definition \ref{D:survivor-2}.1 for the ignition particle, and the Lipschitz bound \eqref{E:Lipschitz}.

The $\sM_{D_\sB(B, L^{10} + 2)}$-measurability claim is immediate from construction and the measurability given $\sF_{\tau_B + 3 L^{20}/2}$ follows by the Lipschitz bound \eqref{E:Lipschitz}.
\end{proof}

We also need to define events that control the possibility that the particle $a^*$ comes into contact with an infected particle outside of the set $S_B$, and to control the SSP near $B$. Here we use some notation from previous sections.
First, using the notation of Lemma \ref{L:far-particles}, define
$$
\cR_{B, B'} := \cE_{B'}(0, [d(B, B')]^{3/2}, d(B, B')/4).
$$
Also let $\cT_{B'}$ be the event where for every $b\in H_{B'}^{++}$ and every $a \in [0, L^{17}]$, the healing process $b^h$ satisfies
	$$
	b^h \cap (a, a + L^{10}/2] \ne \emptyset
	$$
	and the set
	$$
	\{b(t) : t \in [a, a + L^{10}/2]\}
	$$
	has diameter at most $L^{10}/5$.
Note that $\cR_{B, B'}, \cT_{B'}$ are $\sM_{B'}$-measurable events.

\begin{lemma}
\label{L:other-particles}
Let $b \in H_{B'}^{++}$ for some $B'$ with $L^{10} < d(B, B')$.
\begin{enumerate}
	\item If additionally $L^{15} < d_\sB(B, B')$, then on the event $\cR_{B, B'}$ the particle $b$ never comes within distance $d(B, B')/2$ of the location $v_B$ before time $\tau_B + L^{20}$.
	\item If additionally $L^{10} \le d_\sB(B, B') \le L^{15}$, then on the event $\cT_{B'} \cap \cR_{B, B'} \cap \cY_B$, if the block $B$ was ignited, then the particle $b$ never comes within distance $L^{10}/4$ of the location $v_B$ before time $\tau_B + L^{10}$ and will be removed by that time.
\end{enumerate}
\end{lemma}

\begin{proof}
Part 1 of the lemma follows from Corollary \ref{C:no-far-particles}, noting that when $L^{15} < d_\sB(B,B')$, we have
$$
L^{20} + L^2 d(B, B') \le [d(B, B')]^{3/2}.
$$	
For part 2 of the lemma, by the Lipschitz bound \eqref{E:first-lipschitz} we have	
\begin{equation}
\label{E:tauB-spread}
|\tau_B + \xi - \tau_{B'}| \le \xi (d_\sB(B, B') + 1) \le [d(B,B')]^{3/2}.
\end{equation}
Therefore since we are working on the event $\cR_{B, B'}$, we have $d(B, B')/2 \le d(\bar b(t), B) \le 2d(B, B')$ for all $t \in [\tau_{B'}, \tau_B + \xi]$. 

Therefore Lemma \ref{L:red-boxes}(ii) guarantees that $b$ has been infected prior to time $\tau_B + \xi$. Therefore again using \eqref{E:tauB-spread} and the definition of $\cT_{B'}$, the particle $\bar b$ will be removed prior to time $\tau_B + \xi + L^{10}/2 < \tau_B + L^{10}$ and will never come within distance $L^{10}/4$ of $B$ in that time.
\end{proof}
	
The events $\cT_{B'}$ and $\cR_{B, B'}$ have high probability.

\begin{lemma}
	\label{L:AC-estimates}
For any block $B'$ we have
\begin{align*}
\P [\cT_{B'}] \ge 1 - \exp (-c L^{4}).
\end{align*}
and for any $B'$ with $d(B, B') \ge L^{10}$, we have
$$
\P [\cR_{B, B'}] \ge 1 - \exp (-c d(B, B')^{1/2}).
$$
Moreover, both of these bounds hold conditional on $\cY_B$ whenever $B' \notin D_\sB(B, L^{10})$.
\end{lemma}

\begin{proof}
	The estimate on $\cR_{B, B'}$ follows from Lemma \ref{L:far-particles}. For the estimate on $\cT_{B'}$, let $\cZ_{m}$ be the event where
\begin{equation}
\label{E:Xib}
h^b \cap (m, m + L^{10}/4] \ne \emptyset
\end{equation}
for all $b \in H_{B'}^{++}$. 
Then with notation as in Lemma \ref{L:far-particles} we have
$$
\cT_{B'} \supset \bigcap_{m \in \N \cap [0, L^{18}]} \cZ_m \cap \cE_{B'}(m, 2L^{10}, L^{10}/4).
$$
The probability of \eqref{E:Xib} failing for a fixed $b$ is simply the probability that a Poisson random variable of mean $\nu L^{10}/2$ equals $0$, which is $e^{-\nu L^{10}/2}$. Therefore by a union bound and Lemma \ref{l:HBsize},
$$
\P [\cZ_{m}^c] \le \E (|H_{B'}^+| + 1) e^{-\nu L^{10}/2} \le 2 \mu L^2 e^{-\nu L^{10}/2}.
$$
We can bound $\P [\cE_{B'}(m, 2L^{10}, L^{10}/4)]$ with Lemma \ref{L:far-particles}. Combining these estimates with \eqref{E:Lnu-relationship}, a union bound and simplifying yields the bound on $\cT_{B'}$. 

For the conditional claims, note that both $\cT_{B'}, \cR_{B, B'}$ are $\sM_{B'}$-measurable and so they are independent of $\cY_{B}$ unless $d(B', B'') \le 2$ for some $B'' \in D_\sB(B, L^{10})$. 
In the adjacent case, conditioning on $\cY_B$ changes the behaviour of the ignition particle in $B'$ by forcing it to stay still until time $L^3$, and affecting $W^h_{B', \operatorname{ig}}$ in the interval $[0, L^4]$. It is easy to repeat the computations above for $\P[ \cT_{B'}], \P [\cR_{B, B'}]$ under this conditioning to get the same bounds.
\end{proof}

Next, we analyze the SSP close to $B$, conditional on the event $\cY_B$. 

\begin{lemma}
\label{L:cYB-conditional}
Conditional on the event $\cY_B$, the blue seed process $\fB$ is still $1$-dependent and is stochastically dominated by a Bernoulli process of mean $\de_\nu$ as in the unconditional case (Corollary \ref{C:blue-domination}).
\end{lemma}

\begin{proof}
First, by Lemma \ref{L:red-boxes}(i), on $\cY_B$ every block $B' \in D_\sB(B, L^{10})$ is not a blue seed. Moreover, the remaining process of blue seeds still forms a $1$-dependent process under this conditioning, so as in the proof of Corollary \ref{C:blue-domination} it is enough to show that for every $B' \notin D_\sB(B, L^{10})$, we have 
\begin{equation}
\label{E:PcA}
\P(\cA_{B'}^c \; | \; \cY_B) \le \ep_\nu.
\end{equation}
For $B' \in D_\sB(B, L^{10} + 2)$ we can write $\cY_B = \cW \cap \cX$, where $\cX$ is independent of $\cA_{B'}$ and $\cW$ is of the form of the event in Lemma \ref{L:blueSeedconditional}. That lemma implies the bound \eqref{E:PcA}. For $B' \notin D_\sB(B, L^{10} + 2)$, since $\cY_B$ is $\sM_{D_\sB(B, L^{10}+2)}$-measurable by Lemma \ref{L:red-boxes}, the events $\cY_B$ and $\cA_{B'}$ are independent, so Proposition \ref{p:blueSeed} yields \eqref{E:PcA}.
\end{proof}

Given Lemma \ref{L:AC-estimates}, we have the following.
\begin{corollary}
\label{C:blue-prob}
Recall the definition of the events $\cD_B(r, r')$ in Lemma \ref{L:global-local-ests}. We have the following:
\begin{enumerate}
	\item $\fD(B, L^{10}) = \emptyset$ on the event $\cY_B$. In particular, $B \notin [\fD]$ on the event $\cY_B \cap \cD_B(L^{10})$.
	\item As in Lemma \ref{L:global-local-ests}, for all $0 < r < r'$ and $B'$ we have $\P[\cD_{B'}(r, r') \; | \; \cY_B] \ge 1 - \exp(-c r^{c/\log \log r})$.
\end{enumerate}
\end{corollary}

We can now put everything together to construct the survival event we will need moving forward. 

\begin{lemma}
\label{L:putting-things-together}
For a block $B$, let
$$
\cS_B = \cY_B \cap \cD_B(L^{10}, m_B) \cap \bigcap_{B' : L^{10} < d_\sB(B, B') \le L^{15}} \cT_{B'} \cap \bigcap_{B' : L^{10} < d_\sB(B, B') \le m_B} \cR_{B, B'}.
$$
If $m_B \le L^{10}$, we omit the $\cD_B(L^{10}, m_B)$ and $\cR_{B, B'}$ events above. 
Then recalling the global event $\cG_\nu$, we have the following:
\begin{enumerate}[label=(\roman*)]
	\item On $\cS_B \cap \cG_\nu$, 	at all times in $[\tau_B + L^{20} - 1, \tau_B + L^{20}]$, there is a susceptible particle $a^*$ at location $v_B$ and no unremoved particles within distance $L^{15}$ of $v_B$. All particles other than $a^*$ that were coloured in a block in $D_\sB(B, L^{15})$ are removed before time $\tau_B + L^{20} - 1$.
	\item There exists $\delta=\de(L) > 0$ such that $\P [\cY_B] = \de$ for all $B \in \sB$. Moreover, $\P(\cS_B \; | \; \cY_B) \ge 1 - \exp(- L^{c /\log \log L})$.
	\item The event $\cS_B$ is $\sM_{D_\sB(B, m_B + L^{10} + 2)}$-measurable.
\end{enumerate}
\end{lemma}

\begin{proof}
	For part (i), we first check that the block $B$ is ignited on $\cG_\nu \cap \cS_B$. First, on $\cG_\nu$, any block not in $[\fD]$ gets coloured red and hence is ignited by Proposition \ref{P:SI-to-BR}.
	Next, the event $\cG_\nu$ implies the event $\cD_B(m_B)$. Since the event $\cD_B(L^{10}, m_B)$ occurs on $\cS_B$, by \eqref{E:transitivity} this implies the event $\cD_B(L^{10})$, so therefore $B$ gets coloured red as long as $B \notin [\fD(B, L^{10})]$. Corollary \ref{C:blue-prob} implies that $[\fD(B, L^{10})] = \emptyset$ on $\cY_B$, which contains $\cG_\nu \cap \cS_B$.
	
	Next, $\cL_B(B', [d(B, B')]^{3/2}) \sset \cR_{B, B'}$ for any $B, B'$, so the event $\cG_\nu \cap \cS_B$ implies $\cR_{B, B'}$ for all $B' \notin D_\sB(B, L^{10})$. Therefore by Lemma \ref{L:other-particles}.1, in the time interval $[\tau_B, \tau_B + L^{20}]$, no particles that originated in some $B'$ with $d_\sB(B, B') > L^{15}$ come within distance $d(B, B')/2 > L^{15}$ of $v_B$. Moreover, by Lemma \ref{L:other-particles}.2 and Lemma \ref{L:red-boxes}(iii, iv), all particles other than $a^*$ that originated in some $B'$ with $d_\sB(B, B') \le L^{15}$ are recovered by time $a^*$ and none of these come into contact with $a^*$.

Now we turn to (ii). First, $\P[\cY_B]$ does not depend on $B$ by construction. Moreover, we claim that $\cY_B$ is a finite intersection of finitely many independent positive probability events and hence $\P \cY_B > 0$. Indeed, let  $\tilde \cV_{B, B'}$ and $\tilde \cU_B$ be version of $\cV_{B', B}$ and $\cU_B$ where we omit any restrictions on ignition particles. The events $\tilde \cV_{B, B'}$ and $\tilde \cU_B$ are, respectively, $\sM_{B'}$ and $\sM_B$-measurable. Now, $\cY_B$ is then the intersection of all the $\tilde \cV_{B, B'}$ and $\tilde \cU_B$ with a collection of events restricting the behaviour of each ignition particle in $D_\sB(B, L^{10} + 2)$ which are independent of each other, all the $\tilde \cV_{B, B'}$, and $\tilde \cU_B$.

The bound on $\P(\cS_B \; | \; \cY_B)$ follows from the conditional bounds in Lemma \ref{L:AC-estimates} and Corollary \ref{C:blue-prob} along with a union bound.

Finally, the measurability follows from the `Moreover' in Lemma \ref{L:red-boxes} and the $\sM_{B'}$-measurability of the events $\cT_{B'}, \cR_{B, B'}$.
\end{proof}

Moving forward, we will also need more localized versions of $\cS_B$. This is the content of the following lemma.

\begin{lemma}
	\label{L:79}
For each block $B$ and $r > L^{15}$ define
$$
\cS_B^r := \cY_B \cap \cD_B(L^{10}, r \wedge m_B) \cap \bigcap_{B' : L^{10} < d(B, B') \le L^{15}} \cT_{B'} \cap \bigcap_{B' : L^{10} < d(B, B') < m_B \wedge r} \cR_{B, B'}.
$$
Again, if $m_B \le L^{10}$, we omit the $\cD_B(L^{10}, r \wedge m_B)$ and $\cR_{B, B'}$ events.
Then $\cS_B^r$ is measurable given either the $\sig$-algebra $\sF_{\tau_B + 2 r^{3/2}}$ or the $\sig$-algebra $\sM_{D_\sB(B, m_B + L^{10} + 2)}$ and we have the bound
$\P[\cS_B \mid \cS_B^r] \ge 1 - \exp(-  r^{c / \log \log r})$.
\end{lemma}

\begin{proof}
	The $\sM_{D_\sB(B, m_B + L^{10} + 2)}$-measurability follows as in Lemma \ref{L:putting-things-together}. The $\sF_{\tau_B + 2 r^{3/2}}$-measurability follows from the Lipschitz bound \eqref{E:first-lipschitz} since all of the events in the definition of $\cS_B^r$ that come from any fixed $B'$ concern times in the window $(-\infty, \tau_{B'} + r^{3/2}]$. The conditional probability bound follows in the exact same way as in Lemma \ref{L:putting-things-together}, using the transitivity $\cD_B(L^k, r) = \cD_B(L^k, r \wedge m_B) \cap \cD_B(r \wedge m_B, r)$ (Equation \eqref{E:transitivity}).  
\end{proof}
\subsection{The relationship with Section \ref{S:upper-bd-density}}

To move from the rare local survival event $\cY_B$ to a global survival event, we need to check that a version of the arguments from Section \ref{S:upper-bd-density} still goes through even after we condition on $\cY_B$. Our goal will be to show the following lemma, which is the conditional analogue of Proposition \ref{P:big-estimate}. Here all notation is as in Section \ref{S:upper-bd-density}.

\begin{lemma}
	\label{L:upper-bound-conditional}
Let $B \in \sB$, $s \ge L^{20}$. Let $z_B$ be the square closest to the center of $B$ and for $z \in \Z^2$ define $\Lambda_z := D(z + z_B, s^{49/100})$. Then there is an event $\tilde \cX = \tilde \cX_{B, s}$ such that:
\begin{itemize}
	\item  $\P[\tilde \cX \mid \cY_B] \ge 1 - \exp(-s^{c/\log \log s})$ and $\tilde \cX$ is $\sM_{D_\sB(B, m_B + 3s^{3/4})}$-measurable.
	\item  On the event $\tilde \cX \cap \cS_B  \cap \cG_\nu$, we have
	\begin{equation}
	\label{E:bara-sum*}
	\frac{1}{|\Lambda_z|} \sum_{B' \in D_\sB(B, s^{3/4})} \sum_{a \in H_{B'}^{++}} \mathbf{1}(\bar a(\tau_B + s') \in \Lambda_z, \bar a \notin \mathbf{R}_{\tau_B + s'}) \le 2L^{-20}
	\end{equation}
	for all $s' \in [s-1, s]$ and $z \in \Z^2$.
\end{itemize}
\end{lemma}

The event $\tilde \cX$ will simply be equal to the original event $\cX$ from Proposition \ref{P:big-estimate} but with some conditions in the local radius around $B$ removed. In particular, we have $\cX \sset \tilde \cX$.

\begin{proof}
	First, on $\cS_B \cap \cG_\nu$, all particles that originated in $D(B, L^{15})$ not equal to $a^*$ are recovered by time $\tau_B + s-1$ by Lemma \ref{L:putting-things-together}.1. Therefore the left-hand side of \eqref{E:bara-sum*} is less than or equal to
\begin{equation}
\label{E:sumsumsum}
\frac{1}{|\Lambda_z|}\Big(1 +  \sum_{B' : L^{15} < d_\sB(B, B') \le s^{3/4}} \sum_{a \in H_{B'}^{++}} \mathbf{1}(\bar a(\tau_B + s') \in \Lambda_z, \bar a \notin \mathbf{R}_{\tau_B + s'}) \Big)
\end{equation}
Here the $+ 1$ is for the one susceptible particle guaranteed by Lemma \ref{L:putting-things-together}. Therefore it suffices to find an event $\tilde \cX$ satisfying the first condition of the lemma, and such that \eqref{E:sumsumsum} $\le 2L^{-20}$ on $\tilde \cX \cap \cY_B \cap \cG_\nu$.  
 
Let $\tilde \sD = \{B' : L^{15} < d(B, B') \le s^{3/4}\}$, and define $\tilde \sD_1, \tilde \sD_2$ exactly as in the proof of Proposition \ref{P:big-estimate} but with $\tilde \sD$ in place of $\sD$. For $L^{20} \le s \le L^{5000}$, define
$$
\tilde \cX = \P\{\tilde \sD_1 \cup  \tilde \sD_2 = \emptyset\}
$$
and for $s > L^{5000}$, define $\tilde \cX = \tilde \cX' \cap \bigcap_{z \in \Z^2, |z| \le 2s} \tilde \cQ_z$, where
\begin{align*}
\tilde \cX' &= \bigcap_{B' : L^{15} < d(B, B') \le s^{3/4}} \cE(0, 2s, s/4),\\
\tilde \cQ_z &= \bigg\{\sum_{B'\in \tilde \sD_1 \cup \tilde \sD_2} \sum_{a\in H_{B'}^{++}} \mathbf{1}(G_{a}) \leq L^{-20} |\Lambda_z| \bigg \},
\end{align*}
where $G_a$ is as in the proof of Proposition \ref{P:big-estimate}. Exactly as in that proposition, the event $\tilde \cX$ implies that the double sum in \eqref{E:sumsumsum} is at most $L^{-20} |\Lambda_z|$, so the whole expression is at most $2L^{-20} |\Lambda_z|$. The measurability claim for $\tilde \cX$ also follows as in the proof of Proposition \ref{P:big-estimate}.

Now, conditional on $\cY_B$, the process of blue seeds still forms a $1$-dependent process that is stochastically dominated by a Bernoulli process of mean $\de_\nu$ by Lemma \ref{L:cYB-conditional}, so all estimates involving $\sD_1$ for the proof of Proposition \ref{P:big-estimate} go through verbatim for $\tilde \sD_1$. Moreover, the set $\tilde \sD_2$ is independent of $\cY_B$. Indeed, the event $\{B' \in \tilde \sD_2\}$ is $\sM_{D_\sB(B', 10 L^\de)}$-measurable by construction and hence $\tilde \sD_2$ is $\sM_{D_\sB(B, L^{14})^c}$-measurable whereas $\cY_B$ is $\sM_{D_\sB(B, L^{10} + 2)}$-measurable by Lemma \ref{L:red-boxes}. Therefore all estimates involving $\sD_2$ for the proof of Proposition \ref{P:big-estimate} go through verbatim for $\tilde \sD_2$. Finally, given $\sD_1, \sD_2$, all events in the definition of $\tilde \cX$ depend only on $M_{B'}, L^{15} < d(B, B') \le s^{3/4}$ and hence are independent of $\cY_B$. Therefore all remaining estimates in the proof of Proposition \ref{P:big-estimate} go through verbatim here, yielding the result.
\end{proof}

\section{The proof of Theorem \ref{T:main-2}}
\label{S:herd}

The results of Section \ref{S:upper-bd-density} guarantee that most particles in a region recover in an $O(1)$ amount of time after the infection front initially visits that region. The results of Section \ref{S:recovery-local} then guarantee that the infection must die out in this region soon afterwards, and so the infection will typically pass through a region in an $O(1)$ amount of time. The results of Section \ref{S:survival} ensure that a particle has some probability of surviving for an arbitrarily long time after the infection enters its block, which together with the results of Sections \ref{S:upper-bd-density} and \ref{S:recovery-local} indicates that some particles will survive forever. In summary, together the results of these three sections suggest the main content Theorem \ref{T:main-2}. However, putting them together is delicate. This is the goal of the present section.

\subsection{Events and scales for the proof}

For $n \in \N$, define time scales $s_n = L^{20} 2^{n-1}$. We will analyze the SIR process near a block $B$ in the increasing geometric time intervals $[\tau_B + s_n, \tau_B + s_{n+1}]$. We start by defining two events to guarantee herd immunity.

\begin{itemize}	
	\item $\cN_{n, B}$: For all $t \in [\tau_B + s_n, \tau_B + s_{n+1}]$ we have
	\begin{equation*}
	\bigcup_{B' \in D_\sB(B, m_B + s_n^{3/4})}  \{a \in H_{B'} : a \in {\bf I}_t, d(\bar a(t), B) \le s_n^{3/5} \} = \emptyset.
	\end{equation*}
	\item $\cM_{n, B}$: \quad We only care about this event when working on $\cY_B$. On $\cY_B$, there is a distinguished particle $a^* \in H_B$, see Definition \ref{D:survivor}. Let $\cM_{n, B}$ be the event where
	$$
	|a^*(t) - a^*(0)| \le s_n^{5/9}
	$$
	for all $0 \le t \le
	s_{n+1}$.
\end{itemize} 

\begin{lemma}
\label{L:herd-and-survive}
Fix any block $B$.
\begin{enumerate}[label=\arabic*.]
	\item On the event $\cG_\nu$, there are no infected particles $a$ with $d(\bar a (t), B) \le s_n^{3/5}$ for some $t \in [\tau_B + s_n,\tau_B + s_{n+1}]$ from any $H_{B'}$ with $d_\sB(B, B') > m_B + s_n^{3/4}$.
	\item On the event $\cG_\nu \cap \cN_{n, B}$, there are no infected particles $a$ with $d(\bar a (t), B) \le s_n^{3/5}$ for some $t \in [\tau_B + s_n,\tau_B + s_{n+1}]$.
In particular, on the event
	$$
	\operatorname{Herd}_{B, n} := \cG_\nu \cap \bigcap_{m \ge n} \cN_{m, B},
	$$
	there are no infected particles in the block $B$ after time $\tau_B + s_n$.
	\item On the event 
	$$
	\operatorname{Survive}_{B} := \cG_\nu \cap \cS_B \cap \bigcap_{n \ge 1} (\cN_{n, B} \cap \cM_{n, B}),
	$$
	there is a particle $a^* \in H_B$ such that $a^* \in \mathbf{S}_t$ for all $t > 0$.
\end{enumerate}
\end{lemma}

\begin{proof}
	For part $1$, for $B'$ with $d_\sB(B, B') > m_B + s_n^{3/4}$, the event $\cG_\nu$ is contained in the event $\cL_B(B', [d_\sB(B, B')]^{3/2}) \sset \cL_B(B', s_{n+1})$. Moreover, $d(B, B')/2 > s_n^{3/5}$, so by \eqref{E:LB-event} in Corollary \ref{C:no-far-particles}, no particles from $H_{B'}$ come within distance $s_n^{3/5}$ of $B$ in the interval $[\tau_B + s_n, \tau_B + s_{n+1}]$.
	Part $2$ follows from part $1$ and the definition of $\cN_{n, B}$.

For part $3$, the event $\cG_\nu \cap \cS_B$ guarantees that there is a susceptible particle $a^* \in H_B$ contained in the block $B$ at time $\tau_B + s_1 = \tau_B + L^{20}$. This uses Lemma \ref{L:putting-things-together}.1. On the event $\bigcap_{n \ge 1} \cM_{n, B}$, we have that $d(\bar a^*(t), B) \le 4 s_n^{5/9}$ for all $t \in [s_n, s_{n+1}]$. Since $4s_n^{5/9} < s_n^{3/5}$, part $2$ implies that $a^*$ will never encounter an infected particle after time $\tau_B + s_1$.
\end{proof}

The majority of this section is devoted to proving the following two propositions which give bounds on the behaviour of the events $\operatorname{Herd}_{B, n}$ and $\operatorname{Survive}_{B}$ on the event $\cG_\nu$.

\begin{prop}
\label{P:herd-immunity}
For all $B\in \sB$ and $n \ge 2$ we have
$$
\P( \operatorname{Herd}_{B, n} \;|\; \cG_\nu) \ge 1 - \exp(-{(2^n L)}^{c/ \log (n + \log L)}).
$$
\end{prop}

\begin{prop}
	\label{P:exists-delta}
There exists a constant $\alpha = \alpha(\nu) > 0$ such that almost surely, 
$$
\liminf_{m \to \infty} \frac{1}{m^2} \sum_{B : d_\sB(B, 0) \le r} \mathbf{1}(\operatorname{Survive}_B) \ge \al \mathbf{1}(\cG_\nu).
$$
More precisely, for all large enough $r$ we have
$$
\P \Big(\frac{1}{m^2} \sum_{B : d_\sB(B, 0) \le m} \mathbf{1}(\operatorname{Survive}_B) \le \al \; \Big| \; \cG_\nu \Big) \le C_\nu \exp(- m^{c/ \log \log m}).
$$
\end{prop}

\subsection{Probability bounds and the proof of Proposition \ref{P:herd-immunity}}
\label{SS:prob-bounds}

The key to bounding the probabilities in Proposition \ref{P:herd-immunity} and Proposition \ref{P:exists-delta} is understanding the events $\cN_{n, B}$, both conditionally on $\cG_\nu$ and conditionally on the rarer event $\cG_\nu \cap \cS_B$. We will also establish a degree of spatial independence for these events in order to prove Proposition \ref{P:exists-delta}. The first step is to find a way to apply the local recovery result in Proposition \ref{P:death-with-specifics}. This proposition requires us to start with a process that has a low density of particles locally. To ensure this, we require two more events:
\begin{itemize}
	\item $\cO_{n, B}, n \in \N$: For every $B'$ satisfying $m_B + s_n^{3/4} \ge d_\sB(B, B') > s_n^{3/4}$ and every $a \in H_{B'}^{++}$ we have
	$$
	d(a(t), B) \ge 2 s_n^{4/7} \quad \text{ for all } |t| \le s_{n+1} + L^2 d_\sB(B, B').
	$$
	Note that by the Lipschitz bound \eqref{E:first-lipschitz}, the event $\cO_{n, B}$ implies that
	$$
	d(\bar a(t), B) \ge 2 s_n^{4/7} \quad \text{ for all } t \in [\floor{\tau_B}, \tau_B + s_{n+1}].
	$$
	\item $\cP_{n, B}:$ We define this for all $n \ge 2$. Let $z_B$ be the square closest to the center of $B$ and for $z \in \Z^2$ define $\Lambda_{z, n} := D(z + z_B, s_n^{49/100})$.
	For every $z \in \Z^2$, there are at most
		$$
		L^{-20}|\Lambda_{z, n}|
		$$
		unrecovered particles in the ball $\Lambda_{z, n}$ that belong to $H_{B'}$ for some $B'$ with $d_\sB(B, B') \le s_n^{3/4}$ at time $\floor{\tau_B} + s_{n-1}$. The discretization of $\tau_B$ here is in preparation for a discrete martingale argument in Section \ref{S:proof83}.
\end{itemize}

We would like to say that the events $\cO_{n, B}, \cP_{n, B}$ are local, high-probability events. In the case of $\cO_{n, B}$, this is straightforward.
For this lemma recall that $\sF_t, t \ge 0$ is the filtration generated by the process of coloured particles up to time $t$, see Section~\ref{S:SI-colouring}.
\begin{lemma}
\label{L:cO}
Let $A_\sB(B, r_1, r_2) = \{B' \in \sB: r_1 \le d_\sB(B, B') \le r_2\}$ and $n \ge 1$.
The event $\cO_{n, B}$ is both $\sM_{A_\sB(B, s_n^{3/4}, m_B + s_n^{3/4})}$-measurable and $\sF_{\tau_B + s_{n+2} + 2\xi m_B}$-measurable, is independent of $\cY_{B}$, and satisfies
	\[
\P[\cO_{n, B}] \ge 1 - \exp(- c s_n^{1/2}) = 1 - \exp(-c' L^{10} 2^{n/2}).
	\]
\end{lemma}

\begin{proof}
Using the notation of Lemma \ref{L:far-particles}, we can write
$$
\cO_{n, B} \supset \bigcap_{B': m_B + s_n^{3/4} \ge d(B, B') > s_n^{3/4}} \cE_{B'}(0, s_{n+1} + L^2 d_\sB(B, B'), d(B, B')/8).
$$
 The $\sM_{A_\sB(B, s_n^{3/4}, m_B + s_n^{3/4})}$-measurability and independence from $\cY_{B}$ is immediate from the $\sM_{B'}$-measurability of each $\cE_{B'}$ and the independence from $\cY_{B}$ follows since $s_n^{3/4} \ge L^{15}$ and $\cY_B$ is $\sM_{B, L^{10} + 2}$-measurable (Lemma \ref{L:red-boxes}). The $\sF_{\tau_B + s_{n+2} + 2 \xi m_B}$-measurability follows from the Lipschitz bound \eqref{E:first-lipschitz}. The bound on $\P[\cO_{n, B}]$ follows from Lemma \ref{L:far-particles}. 
\end{proof}

While the event $\cP_{n, B}$ does not have a local definition, it contains a high-probability local version both conditionally on $\cG_\nu$ and conditionally on $\cG_\nu \cap \cS_B$. This next lemma follows immediately from Proposition \ref{P:big-estimate} and Lemma \ref{L:upper-bound-conditional}.

\begin{lemma}
	\label{L:local-Pbrief}
	Fix a scale $n$ and a block $B$. Then with notation as in Proposition \ref{P:big-estimate} and Lemma \ref{L:upper-bound-conditional} we have 
	$$
	\cG_\nu \cap \cX_{B, s_{n-1}} \sset \cP_{n, B} \qquad \text{ and } \qquad \cG_\nu \cap \cS_B \cap \tilde \cX_{B, s_{n-1}} \sset \cP_{n, B}.
	$$
\end{lemma}

We will use the events $\cO_{n, B}, \cP_{n, B}$ to help understand our main event of interest, $\cN_{n, B}$. 

\begin{lemma}
\label{L:NnB-conditional}

For $n \ge 2$, define the $\sig$-algebras
\begin{align*}
\sE_{n, B} &:= \sig(\sF_{\floor{\tau_B} + s_{n-1}}, \sM_{D_\sB(B, s_n^{3/4})}\cap \sF_{\tau_B + s_{n+1}}), \\
\quad \sI_{n, B} &:= \sig(\sF_{\floor{\tau_B} + s_{n-1}}, \sM_{D_\sB(B, s_n^{3/4})^c}) 
\end{align*}
For $n \ge 2$ and $B \in \sB$ there are events $\tilde \cN_{n, B}$ which are measurable given $\sE_{n, B}$ such that
$$
\tilde \cN_{n, B} \cap \cO_{n, B} \cap \cG_\nu \sset \cN_{n, B}.
$$
Moreover, we have the following bound:
\begin{align}
\label{E:sInB}
\P(\tilde \cN_{n, B} \mid \sI_{n, B}) &\ge (1 - \exp(- c L^4 2^{n/2}))\mathbf{1}(\cP_{n, B}).
\end{align}
\end{lemma}

\begin{proof}
	The idea is to appeal to Proposition \ref{P:death-with-specifics}. The details are as follows. 
	
Condition on $\sF_{\floor{\tau_B} + s_{n-1}}$, and consider the configuration of unrecovered particles at time $\floor{\tau_B} + s_{n-1}$. Split the set of unrecovered particles at time $\floor{\tau_B} + s_{n-1}$ into two groups:
\begin{itemize}[nosep]
	\item $P$, consisting of all unrecovered particles in $\bigcup \{H_{B'} : d_\sB(B, B') \le s^{3/4} \}$.
	\item $Q$, consisting of all other unrecovered particles.
\end{itemize}
Given $\sF_{\floor{\tau_B}+ s_{n-1}}$, the future trajectories of all particles in $P$ are independent continuous time random walks.

 Now, if we are on the $\sF_{\floor{\tau_B} + s_{n-1}}$-measurable event $\cP_{n, B}$ then the $P$-particles at time $\floor{\tau_B} + s_{n-1}$ satisfy
\begin{equation}
\label{E:P-bound}
|P(\tau_B + s_{n-1}) \cap B^{j \sqrt{s_{n-1}}, z}| \le 2L^{-20} |B^{j \sqrt{s_{n-1}}, z}|
\end{equation}
for every translate $z \in \Z^2, j \in \N$. The factor of $2$ enters here since we may not be able to exactly tile $B^{j \sqrt{s_{n-1}}, z}$ with sets of the form $B^{s_n^{49/100}, y}$.

In particular, if we shift time back by $\floor{\tau_B} + s_{n-1}$ and recenter at $z$, then the corresponding process of $P$-particles satisfies assumption \eqref{E:IC-assumption} with $\delta = 2 L^{-20}$ and $M = \sqrt{s_{n-1}}$. Therefore by the relationship between $L$ and $\nu$, \eqref{E:Lnu-relationship}, we are in the setting of Proposition \ref{P:death-with-specifics}, and so for every $z \in \Z^2$ we can define an event $\cB_{P, z}$ such that we have:
\begin{itemize}[nosep]
	\item $\P[\cB_{P, z} \mid \sF_{\floor{\tau_B}+ s_{n-1}}] \ge (1 - e^{-c L^{-6} \sqrt{s_{n-1}}}) \mathbf{1}(\cP_{n, B}) \ge 1 - e^{-c L^4 2^{n/2}}\mathbf{1}(\cP_{n, B}).$
	\item $\cB_{P,z}$ is measurable given only 
	$\sF_{\floor{\tau_B} + s_{n-1}}$ and
	$\sM_{B, s_n^{3/4}} \cap \sF_{\tau_B + s_{n+1}}$ (i.e. the trajectories of $P$-particles up to time $\tau_B + s_{n+1}$). 
	\item On the event $\cD_{Q, z} \cap \cB_{P, z}$, where 
	$$
	\mathcal D_{Q, z} = \{ \text{For all } t \in [\floor{\tau_B} + s_{n-1}, \tau_B + s_{n+1}], \text{ we have } Q(t) \cap B^{2\sqrt{s_{n-1}}, z} = \emptyset \}
	$$
	there are no infected particles in $B^{\sqrt{s_{n-1}}, z}$ at any $t \in [\tau_B + s_n, \tau_B + s_{n+1}]$. 
\end{itemize}
Now, we set
$$
\tilde \cN_{n, B} = \bigcap_{z : d(z, B) \le s_n^{3/5}} \cB_{P, z}.
$$
By the first and second bullet points above, the event $\tilde \cN_{n, B}$ satisfies the measurability claims in the lemma. It also satisfies \eqref{E:sInB} with $\sF_{\floor{\tau_B} + s_{n-1}}$ in place of $\sI_{n, B}$. As conditioning on the finer $\sig$-algebra $\sI_{n, B}$ only gives us additional information about the future trajectories of $Q$-particles and these are independent of the future trajectories of the $P$-particles, $\tilde \cN_{n, B}$ satisfies \eqref{E:sInB}. 

 To get the containment claim, it is enough to observe that
\[
\cO_{n, B} \cap \cG_\nu \sset \bigcap_{z : d(z, B) \le s_n^{3/5}} \cD_{Q, z}
\]
which follows from the definition of $\cO_{n, B}$ and Lemma \ref{L:herd-and-survive}.1.
\end{proof}

\begin{corollary}
	\label{C:avgNnB}
	For every $n \ge 2$ we have
	$$
	\P[\tilde \cN_{n, B} \mid \cG_\nu] \ge 1 - \exp(-{(2^n L)}^{c/ \log (n + \log L)}).
	$$
\end{corollary}

\begin{proof}
We have
\begin{align*}
\P[\tilde \cN_{n, B} \mid \cG_\nu] &\ge \P[\tilde \cN_{n, B} \cap \cP_{n, B}\mid \cG_\nu] 
\ge 2\P[\tilde \cN_{n, B} \cap \cP_{n, B}] - 1
\end{align*}
where in the second inequality we have  used the inequality $\P(A \mid B) \ge 1 - \P(A^c)/\P(B) = 1 + (\P(A)-1)/\P(B)$ along with the fact that $\P \cG_\nu \ge 1/2$. The bound then follows from averaging the bound in Lemma \ref{L:NnB-conditional}, applying the first containment in Lemma \ref{L:local-Pbrief}, and using the bound in Proposition \ref{P:big-estimate}.
\end{proof}

\begin{proof}[Proof of Proposition \ref{P:herd-immunity}]
This follows from the fact that $\tilde \cN_{n, B} \cap \cO_{n, B} \cap \cG_\nu \sset \cN_{n, B}$ (Lemma \ref{L:NnB-conditional}), Lemma \ref{L:cO}, Corollary \ref{C:avgNnB}, and a union bound. 
\end{proof}

We finish this subsection with two more event estimates that will be needed for the proof of the more difficult Proposition \ref{P:exists-delta}. The first lemma addresses the event $\cN_{n, 1}$ which was not addressed in Lemma \ref{L:NnB-conditional}.

\begin{lemma}
	\label{L:NnB-conditional-small}
	We have
	$$
	\cG_\nu \cap \cS_B \cap \cM_{1, B} \cap \cO_{1, B} \sset \cN_{1, B},
	$$
\end{lemma}

\begin{proof}
	On the event $\cG_\nu \cap \cS_B $ there are no unrecovered particles from $H_{B'}, B' \in D_\sB(B, L^{15})$ after time $\floor{\tau_B} + s_1$ except for the special particle $a^*$ located at $v_B$ by Lemma \ref{L:putting-things-together}(i).
	Also, on the event $\cO_{1, B}$, no particles from a block $B'$ with $s_1^{3/4} < d_\sB(B, B') \le s_1^{3/4} + m_B$ come within distance $2 s_1^{4/7}$ of $B$ 
	in the interval $[\floor{\tau_B}, \tau_B + s_2]$. The same holds for particles from $B'$ with $m_B + s_1^{3/4} < d(B, B')$ since we work on $\cG_\nu$ by the reasoning in Lemma \ref{L:herd-and-survive}.
	
	Therefore no infected particles $a$ satisfy $d(\bar a(t), B) \le s_1^{3/5}$ in the interval $[\tau_B + s_1, \tau_B + s_2]$ unless $a^*$ becomes infected during this time. However, this cannot happen since on $\cM_{1, B}$, we have $d(\bar a^*(t), B) < L^{20}$ in this time window.
\end{proof}

\begin{lemma}
	\label{L:srw-M}
	For $n \ge 1$ we have $\P[\cM_{n, B} \mid \cY_B] \ge 1 - \exp(-c L^{20/9} 2^{n/9}).$
\end{lemma}

\begin{proof}
On $\cY_B$ the particle $a^*$ performs a random walk conditioned not to move in the time interval $[0, L^{20}]$. Therefore conditional on $\cY_B$, for any $t$ the random variable $\max_{s \in [0, t]} |a^*(t) - a^*(0)|$ is stochastically dominated by the version of the random variable for an unconditioned random walk, and so the bound follows from Lemma \ref{L:rw-estimate}.
\end{proof}

\subsection{The proof of Proposition \ref{P:exists-delta}}
\label{S:proof83}

The proof of Proposition \ref{P:exists-delta} given the bounds in Section \ref{SS:prob-bounds} is more involved than the proof of Proposition \ref{P:herd-immunity} as we need will need to establish that far away events $\operatorname{Survive}_B$ behave independently. For this we set up a martingale argument.

Fix a large $r \in \N$, and let $n \in \N, n \ge 2$ be such that $s_n \le r$. For each $i \in \{0, \dots, 5s_{n-1} - 1\}$ and each $p \in \{0, 1, \dots, r-1\}^2$ we will define a martingale $J_{(\ell, x)} = J_{(\ell, x)}^{r, n, i, p}$. In the definition and various notation that follows, we will typically suppress the dependence on $r, n, i, p$ when these values do not play an important role.

To specify the index set of $(\ell, x)$, we first need a definition. For $x \in \Z^2$, define the block
$$
B^x = L (r x + p) + \{-L/2, \dots, L/2 - 1\}^2 \in \sB.
$$

Now let $K = \{x \in \Z^2 : B^x \in D_\sB(0, r^5)\}$, where here $0$ is the block containing the origin. The index set for $J$ is all $(\ell, x) \in I = I_{r, p}$, where
$$
I := \{\emptyset\} \cup I', \quad \text{and} \quad I' = I'_{r, p} = \{0, 1, \dots, r^5\} \times K.
$$
We totally order $I$ so that $\emptyset$ is the least element, and $I'$ has the lexicographic order. We note for later use that $|I| \le C r^{13}$.
For $(k, y) \in I'$, define
$$
\cT_{(k, y)} = \{\floor{\tau_{B^{y}}} + s_{n-1} = 5s_{n-1} k + i\} \cap \cS^{s_{n-1}^{2/3}/2}_{B^y}
$$
where $\cS^r_B$ is defined as in Lemma \ref{L:79}.
Finally, we define the martingale:
\begin{align*}
J_{(\ell, x)} &= \sum_{(k, y) \le (\ell, x)} \lf[\mathbf{1}(\cT_{(k, y)} \cap \tilde \cN_{n, \mathbf{y}}) - \P (\cT_{(k, y)} \cap \tilde \cN_{n, B^y} \mid  \sH_{(k, y)}) \rg]
\end{align*}
and let $J_\emptyset = 0$.
The sequence $J$ is a martingale with respect to $\sH_{(k, y)}$ if $\sH_{(k, y)}, (k, y) \in I$ is any filtration for which $J_{(\ell, x)}$ is $\sH_{(k, y)}$-measurable whenever $(\ell, x) < (k, y)$, i.e.
$$
\E[J_{(\ell, x)} \mid \sH_{(\ell, x)}] = J_{(\ell, x)^-},  
$$
where $(\ell, x)^-$ is the predecessor of $(\ell, x)$ in the total order on $I$. (Note that our indexing of the filtration here is not the standard one which would have $\sH_{(\ell, x)}$ labeled as $\sH_{(\ell, x)^-}$; we have done this to ease notation later on).

 We will work with the following filtration satisfying this. Set $\sH_\emptyset = \emptyset$ and for $(k, y) \in I_r$ define $\sH_{(k, y)}$ to be the $\sig$-algebra generated by:
\begin{itemize}
	\item For $(\ell, x) < (k, y)$, the events
	$
	\cT_{(k, y)} \cap \tilde \cN_{n, B^y}.
	$
	\item The $\sig$-algebra $\sF_{5 s_{n-1} k + i}$.
\end{itemize}
By construction, $J$ is a martingale with increments bounded by $1$. As a result, we can record the following concentration bound.

For this next corollary, we write $\bar J^{r, n, i, p}$ for the final state of the martingale.

\begin{corollary}
	\label{C:mart-conc}
	For any $r, n, i, p$ and any $\la > 0$ we have:
	$$
	\P \lf(|\bar J^{r, n, i, p}| \ge \la \rg) \le 2\exp (-c \la^2 r^{-13}).
	$$
	In particular, letting $J^{r, n} := \sum_{i \in [0, 5s_n-1], p \in [0, r-1]^2} \bar J^{r, n, i, p}$, we have
	$$
	\P \lf(|J^{r, n}| \ge \la \rg) \le 5 s_n r^2 \exp \lf( \frac{- c\la^2 r^{-13} }{ s_n^2 r^{4}}\rg) \le 5 r^3 \exp ( - c'\la^2 r^{-19} ).
	$$
\end{corollary}

\begin{proof}
The first inequality is simply Azuma's inequality and the second inequality follows by a union bound.	
\end{proof}

The usefulness of this martingale comes from the following observation. 
\begin{lemma}
\label{L:Elx-cond}
For any $(\ell, x) \in I'$ we have the following bound:
\begin{equation}
\label{E:TnBy}
\begin{split}
\P (\cT_{(\ell, x)}& \cap \tilde \cN_{n, B^x} \mid  \sH_{(\ell, x)}) \ge  (1 - \exp(- c L^{4} 2^{n/2}))\mathbf{1}(\cT_{(\ell, x)} \cap \cP_{n, B^{x}}).
\end{split}
\end{equation}
\end{lemma}

To prove the lemma we use the following simple fact.

\begin{lemma}
	\label{L:cond-exp-fact}
Suppose that $B$ is an event, and $\sF, \sG$ are $\sig$-algebras such that $B \in \sF$ and 
$$
A \in \sF \quad \implies \quad A \cap B \in \sG.
$$
Then for any event $C$ we have $\P(\P(B \cap C \mid \sG) \mid \sF) = \P(B \cap C \mid \sF)$.
\end{lemma}

\begin{proof}
	First observe that by letting $A$ be the whole space in the assumption of the lemma, we get that $B \in \sG$. Now let $A \in \sF$ be arbitrary. Then
\begin{align*}
\E \P(\P(B \cap C \mid \sG) \mid \sF) \mathbf{1}(A) &= \E \P(B \cap C \mid \sG) \mathbf{1}(A) \\
&= \E \P(C \mid \sG) \mathbf{1}(A \cap B) = \P(A \cap B \cap C)
\end{align*}
where the second equality uses that $B \in \sG$ and the third equality uses that $A \cap B \in \sG$. On the other hand, 
$$
\E \P(B \cap C \mid \sF) \mathbf{1}(A) = \P(A \cap B \cap C)
$$
by definition, completing the proof.
\end{proof}

\begin{proof}[Proof of Lemma \ref{L:Elx-cond}]	The basic idea of the proof is to apply Lemma \ref{L:NnB-conditional}, conditioning on the $\sig$-algebra $\sI_{n, B^x}$. To pass from conditioning on $\sI_{n, B^x}$ to $\sH_{(\ell, x)}$ we will use Lemma \ref{L:cond-exp-fact} with $B = \cT_{(\ell, x)}$. 
	
	First we claim that $\cT_{(\ell, x)}$ is $\sF_{5 s_{n-1} \ell + i}$-measurable and hence is also $\sH_{(\ell, x)}$-measurable. Indeed, by Lemma \ref{L:79} the event $\cS^{s_{n-1}^{2/3}/2}_{B^x}$ is $\sF_{\tau_{B^x} + s_{n-1}/\sqrt{2}}$-measurable, so the whole event $\cT_{(\ell, x)}$ is $\sF_{\floor{\tau_{B^x}} + s_{n-1}}$-measurable and hence is also 
	$\sF_{5 s_{n-1} \ell + i}$-measurable since 
	\begin{equation}
	\label{E:equality}
	\floor{\tau_{B^x}} + s_{n-1} = 5 s_{n-1} \ell + i
	\end{equation}
	 on the event $\cT_{(\ell, x)}$. Next, we claim that
	\begin{equation}
	\label{E:implication}
	A \in \sH_{(\ell, x)} \quad \implies \quad A \cap \cT_{(\ell, x)} \in \sI_{n, {B^x}}.
	\end{equation}
	It is enough to check this for a collection of events $A$ that generate the $\sig$-algebra $\sH_{(\ell, x)}$. First, since \eqref{E:equality} holds on $\cT_{(\ell, x)}$ and $\sF_{\floor{\tau_{B^x}} + s_{n-1}} \sset \sI_{n, B^x}$, the implication \eqref{E:implication} holds whenever $A \in \sF_{5 s_{n-1} \ell + i}$. 
	 Now suppose 
	$
	A = \cT_{(k, y)} \cap \tilde \cN_{n, B^y}
	$
	 for some $(k, y) < (\ell, x)$. If $k < \ell$ then on $\cT_{(k, y)} \cap \cT_{(\ell, x)}$ we have
	$$
	\tau_{B^y} + s_{n+1} < \floor{\tau_{B^x}} + s_{n-1},
	$$
	and so 
	$$
	\cT_{(\ell, x)} \cap \cT_{(k, y)} \cap \tilde \cN_{n, B^y} \in \sF_{\floor{\tau_{B^x}} + s_{n-1}} \sset \sI_{n, B^x}
	$$
	by the $\sE_{n, B^y}$-measurability of $\tilde \cN_{n, B^y}$ in Lemma \ref{L:NnB-conditional}. Alternately, suppose $k = \ell$ but $y \ne x$. In this case $\cT_{(\ell, x)} \cap \cT_{(k, y)}$ implies that
	$$
	\floor{\tau_{B^y}} + s_{n-1} = \floor{\tau_{B^x}} + s_{n-1}
	$$ 
	and since the points $r x + p, r y + p$ are distance more than $r \ge 2s_n^{3/4}$ apart we have that 
	$$
	\sM_{D_\sB(B^y, s_n^{3/4})} \sset \sM_{D_\sB(B^x, s_n^{3/4})^c} \sset \sI_{n, B^x}.
	$$ 
	Therefore again using the measurability claim in Lemma \ref{L:NnB-conditional}, we have $\cT_{(\ell, x)} \cap \cT_{(k, y)} \cap \tilde \cN_{n, B^y} \in \sI_{n, B^x}$, completing the proof of \eqref{E:implication}.

	Now, we can compute that
	\begin{equation}
	\label{E:conditioning-pf}
	\begin{split}
\P (\cT_{(\ell, x)} &\cap \tilde \cN_{n, B^x} \mid  \sI_{n, B^x}) \\
&= \mathbf{1}(\cT_{(\ell, x)}) \P (\tilde \cN_{n, B^x} \mid  \sI_{n, B^x}) \\
&\ge (1 - \exp(- c L^{4} 2^{n/2}))\mathbf{1}(\cT_{(\ell, x)} \cap \cP_{n, B^x}).
\end{split}
	\end{equation}
The equality here uses the $\sI_{n, B^x}$-measurability of $\cT_{(\ell, x)}$, which follows from \eqref{E:implication} when $A$ is the whole space. The inequality then follows from Lemma \ref{L:NnB-conditional}. Now take conditional expectations of both sides of \eqref{E:conditioning-pf} with respect to $\sH_{(\ell, x)}$. The right side is unchanged: it is $\sF_{\floor{\tau_{B^x}} + s_{n-1}}$-measurable by \eqref{E:equality} and the definition of $\sI_{n, B^x}$ and hence is also
$\sH_{(\ell, x)}$-measurable.
The left-hand side becomes $\P (\cT_{(\ell, x)} \cap \tilde \cN_{n, B^x} \mid  \sH_{(\ell, x)})$ by Lemma \ref{L:cond-exp-fact}.
\end{proof}

For this next corollary and for the remainder of this section, to simplify notation we write $\cS_{n, B} = \cS^{s_{n-1}^{2/3}/2}_{B}$ and $\tilde \cX_{n, B} = \tilde \cX_{B, s_{n-1}}$. 
\begin{corollary}
	\label{C:mart-summation}
On the event $\cG_\nu$ for any $r \in \N$ and $n \ge 2$ with $s_n \le r$ we have
\begin{equation}
\label{E:Cor-eqn}
\begin{split}
&\sum_{B \in D_\sB(0, r^5)} \mathbf{1}(\tilde \cN_{n, B}^c \cap \cS_{B}) \le - J^{r, n} + \\
&\sum_{B \in D_\sB(0, r^5)}\mathbf{1}(\cS_{n, B} \smin \cS_B) +  \exp(- c L^{4} 2^{n/2}) \mathbf{1}(\cS_B) + \mathbf{1}(\cS_B \smin \tilde \cX_{n, B}).
\end{split}
\end{equation}
\end{corollary}

\begin{proof}
Fix $r, n$. To simplify notation in the proof we write $\ep = \exp(- c L^{4} 2^{n/2})$. For any $r, n, i, p$ using Lemma \ref{L:Elx-cond} we have that $J^{r, n, i, p}$ is bounded above by
\begin{align*}
\sum_{(k, y) \in I_{r, p}'} \mathbf{1}(\cT_{(k, y)} \cap \tilde \cN_{n, B^y}) - (1 - \ep)\mathbf{1}(\cT_{(k, y)} \cap \cP_{n, B^{y}}),
\end{align*}
and so summing over $i, p$ we get that
\begin{align}
\label{E:triplesum}
J^{r, n} \le \sum_{B \in D_\sB(0, r^5)} \sum_{i=0}^{5 s_{n-1} - 1} \sum_{k = 0}^{r^5} \mathbf{1}(\cT_{B, k, i} \cap \tilde \cN_{n, B}) - (1 - \ep)\mathbf{1}(\cT_{B, k, i} \cap \cP_{n, B}),
\end{align}
where
$$
\cT_{B, k, i} =  \{\floor{\tau_B} + s_{n-1} = 5s_{n-1} k + i\} \cap \cS_{n, B}.
$$
Now, by \eqref{E:first-lipschitz} we necessarily have $\tau_B \le L^2 d_\sB(B, 0) \le L^2 r^5$ and so $\floor{\tau_B} + s_{n-1} \le 5 s_{n-1} r^5$. Therefore $\floor{\tau_B} + s_{n-1}$ must equal $5s_{n-1} k + i$ for some $i \in \{0, \dots, 5 s_{n-1} - 1\}, k \in \{0, \dots,r^5\}$, and so the triple sum in \eqref{E:triplesum} equals
\begin{align*}
\sum_{B : d(B, 0) \le m} \mathbf{1}(\cS_{n, B} \cap \tilde \cN_{n, B}) - (1 - \ep)\mathbf{1}(\cS_{n, B} \cap \cP_{n, B}).
\end{align*}
Next, $\cG_\nu \cap \tilde \cX_{n, B} \cap \cS_{B} \sset \cP_{n, B}$ by Lemma \ref{L:local-Pbrief} and $\cS_B \sset \cS_{n, B}$ by construction, so on $\cG_\nu$ we have
\begin{align}
\label{E:mart-bd}
J^{r, n} 
&\le \sum_{B \in D_\sB(0, r^5)} \mathbf{1}(\cS_{n, B} \cap \tilde \cN_{n, B}) - (1 - \ep)\mathbf{1}(\tilde \cX_{n, B} \cap \cS_B).
\end{align}
We now convert this into the desired bound on
$$
\sum_{B \in D_\sB(0, r^5)} \mathbf{1}(\tilde \cN_{n, B}^c \cap \cS_{B}).
$$ 
Indeed, again using that $\cS_{B} \sset \cS_{n, B}$ we can write
$$
\mathbf{1}(\tilde \cN_{n, B}^c \cap \cS_{B}) \le \mathbf{1}(\tilde \cN_{n, B}^c \cap \cS_{n, B}) =  \mathbf{1}(\cS_{n, B} \smin \cS_B) + \mathbf{1}(\cS_B) - \mathbf{1}(\cS_{n, B} \cap \tilde \cN_{n, B}),
$$
which after plugging in the bound in \eqref{E:mart-bd} gives that the left-hand side of \eqref{E:Cor-eqn} is bounded above by
\begin{align*}
- J^{r, n} + \sum_{B \in D_\sB(0, r^5)}\mathbf{1}(\cS_{n, B} \smin \cS_B) + \mathbf{1}(\cS_B) - (1-\ep)\mathbf{1}(\tilde \cX_{n, B} \cap \cS_B).
\end{align*}
Using the estimate $\mathbf{1}(\cS_B) - (1-\ep)\mathbf{1}(\tilde \cX_{n, B} \cap \cS_B) \le \ep \mathbf{1}(\cS_B) + \mathbf{1}(\cS_B \smin \tilde \cX_{n, B})$ then completes the proof.
\end{proof}

For this next lemma and the remainder of the section, we let $\de = \de_L > 0$ be as in Lemma \ref{L:putting-things-together} so that for all $B \in \sB$ we have
$$
\de = \P[\cY_B] \ge \P[\cS_B] \ge 2 \de/3.
$$
\begin{lemma}
	\label{L:conc-P-S}
	Fix $n, r \in \N$ with $r$ large and with $s_n \le r$ and let $\la > 8 \de L^{-1} \exp(- c 2^{cn /\log n})$. For $n \ge 2$ let
	$$
	W_{n, B} = \mathbf{1}(\cS_{n, B} \smin \cS_B) + \mathbf{1}(\cS_B \smin \tilde \cX_{n, B}) + \mathbf{1} (\cS_B \smin (\cM_{n, B} \cap \cO_{n, B})).
	$$
	Also define $W_{1, B} = \mathbf{1} (\cS_B \smin (\cM_{1, B} \cap \cO_{1, B}))$. Then
	$$
	\P \Big( \sum_{B \in D_\sB(0, r^5)} 
	W_{n, B} \ge \la r^{10} \Big) \le 4 r^2 \exp(- c \la^2 r^8).
	$$
	Similarly, for any $m \in \N$ and $\la < \de/4$ we have
	$$
	\P \Big( \sum_{B \in D_\sB(0, r^5)} \mathbf{1} (\cS_B) \le \la r^{10} \Big) \le 4 r^2 \exp(-c \la^2 r^8).
	$$
\end{lemma}

\begin{proof}
First let $n \ge 2$. For any $B \in D_\sB(0, r^5)$, since $2 r > m_B + s_n$ for large enough $r$, the events 
$$
\tilde \cX_{n, B}, \quad \cM_{n, B}, \quad \cO_{n, B}, \quad \cS_B, \quad  \cS_{n, B},
$$
are all measurable given $\sM_{D_\sB(B, 2 r)}$. This uses Lemma \ref{L:putting-things-together} for $\cS_B$, Lemma \ref{L:79} for $\cS_{n, B} = \cS^{s_{n-1}^{2/3}/2}_{B}$, Lemma \ref{L:upper-bound-conditional} for $\tilde \cX_{n, B} = \tilde \cX_{B, s_{n-1}}$, Lemma \ref{L:cO} for $\cO_{n, B}$, and is immediate from the definition for $\cM_{n, B}$.
 Moreover 
$$
\P(W_{n, B} \ge 1) \le C \de L^{-1} \exp(- 2^{cn /\log n}).
$$
for $L$ large enough.
This uses the fact that all events in the definition of $W_{n, B}$ are contained in $\cY_B$ along with the fact that $\P[\cY_B] = \de$ (Lemma \ref{L:putting-things-together}), the bound in Lemma \ref{L:79} on $\P[\cS_B \mid \cS_{n, B}]$, the bound on $\P[\tilde \cX_{n, B} \mid \cY_B]$ from Lemma \ref{L:upper-bound-conditional}, the bound on $\P[\cO_{n, B}]$ from Lemma \ref{L:cO} along with the independence of $\cO_{n, B}$ and $\cY_B$, and finally the bound on $\P[\cM_{n, B} \mid \cY_B]$ from Lemma \ref{L:srw-M}. The first bound in the lemma for $n \ge 2$ then follows from Lemma \ref{l:dependentPerc}. The $n=1$ case is identical except we do not need to worry about the first two $W_{n, B}$ terms.
 The second bound also follows from Lemma \ref{l:dependentPerc}, 
this time using the lower bound of $\P[\cS_B] \ge \de/2$. \end{proof}

\begin{proof}[Proof of Proposition \ref{P:exists-delta}]
	We will just prove the conditional probability bound, as the liminf claim then follows from the Borel-Cantelli lemma. It is also enough to prove the conditional probability bound when $m=r^5$ for sufficiently large $r \in \N$.
	
We work on the event $\cG_\nu$ throughout the proof. First, by the containments in Lemma \ref{L:NnB-conditional-small} and Lemma \ref{L:NnB-conditional} we can write
\begin{align}
\nonumber
\sum_{B \in D_\sB(0, r^5)} &\mathbf{1}(\operatorname{Survive}_B) \\
\label{E:Bd0B}
&\ge \sum_{B \in D_\sB(0, r^5)} \mathbf{1}\Big(\cS_B \cap \bigcap_{n\ge 1} (\cM_{n, B} \cap \cO_{n, B}) \cap \bigcap_{n\ge 2} \tilde \cN_{n, B}\Big).
\end{align}
Now let $n(r) \in \N$ be the largest value of $n$ such that $s_n \le r$ so that $n(r) \sim \log_2 r$ as $r \to \infty$. 
 By Lemmas \ref{L:cO} and \ref{L:srw-M} and Corollary \ref{C:avgNnB}, the probability that $\tilde \cN_{n, B} \cap \cO_{n, B} \cap \cM_{n, B}$ fails for some $B$ with $d(B, 0) \le r^5$ and $n \ge n(r)$ is at most
 $
\exp(-{(r L)}^{c/ \log \log (r L)})
 $
 for large enough $r$.
 
 Therefore with probability at least $1- \exp(-{(r L)}^{c/ \log \log (r L)})$, the right-hand side of \eqref{E:Bd0B} is equal to 
 \begin{align*}
 &\sum_{B \in D_\sB(0, r^5)} \mathbf{1}\Big(\cS_B \cap \bigcap_{n=1}^{n(r)} (\cM_{n, B} \cap \cO_{n, B}) \cap \bigcap_{n= 2}^{n(r)} \tilde \cN_{n, B}\Big) \\
 &\ge \sum_{B \in D_\sB(0, r^5)} \Big( \mathbf{1}(\cS_B) -\sum_{n=2}^{n(r)} \mathbf{1}(\cS_B \cap \tilde \cN_{n, B}^c) -\sum_{n=1}^{n(r)} \mathbf{1} (\cS_B \smin (\cM_{n, B} \cap \cO_{n, B}) \Big).
 \end{align*}
Now by Corollary \ref{C:mart-summation}, using the $W_{n, B}$ notation from Lemma \ref{L:conc-P-S}, this is bounded below by
 \begin{align*}
 & \sum_{B : d(0, B)\le r^5}\frac{3}{4}\mathbf{1}(\cS_B)
 - \sum_{n=2}^{n(r)} |J^{r, n}| - \sum_{n=1}^{n(r)} \sum_{B : d(0, B)\le r^5} W_{n, B}
 \end{align*}
 for large enough $L$. Then as long as $L$ is sufficiently large and $r$ is sufficiently large given $L$ we have the following bounds. The first of the three terms above is bounded below by $\de r^{10}/20$ with probability at least $1 - \exp(-c r)$ by Lemma \ref{L:conc-P-S}. Each of the summands in the second term is bounded above by $\de r^{10}/\log^2(r)$ with probability at least $1-\exp(-r^{1/2})$ by Corollary \ref{C:mart-conc}. Each of the summands in the final term is bounded above by $\de r^{10}/\log^2(r)$ with probability at least $1 - \exp(-c r)$ by Lemma \ref{L:conc-P-S}. Putting everything together implies that the above expression is bounded below by $\de r^{10}/40$ with probability at least $1-\exp(-r^{1/2})$ for large enough $r$, yielding the desired probability bound in Proposition \ref{P:exists-delta} with $\al = \de/40$.
\end{proof}
 
 \subsection{Proof of Theorem \ref{T:main-2}}

 At this point, we just need to gather together everything we have done in the last few sections to prove Theorem \ref{T:main-2}.
 
First, $\P\cG_\nu \to 1$ as $\nu \to 0$ by Corollary \ref{L:global-is-rare}. Next, $\cG_\nu$ is contained in the event $\cC$ in \eqref{E-event}, which by Theorem \ref{T:random-red-blue} and Proposition \ref{P:SI-to-BR} implies that on $\cG_\nu$ infinitely many blocked become ignited, and so $\mathbf{I}_t \ne \emptyset$ for all $t \ge 0$, giving point $1$.

We next prove point $3$. For a site $x \in \Z^2$, let $B(x)$ denote the block containing $x$, let $\sig_{B(x)}$ denote the final time an infected particle is in the block $B(x)$ and let $N(x)$ denote the smallest value of $n$ such that $\operatorname{Herd}_{B(x), n}$ succeeds. Then by Lemma \ref{L:herd-and-survive}.2, we have
$$
D_x \le \sig_{B(x)} - \tau_{B(x)} \le s_{N(x)} = L^{20} 2^{N(x)-1}
$$
and by Proposition \ref{P:herd-immunity} we have 
$$
\P(\sig_{B(x)} - \tau_{B(x)} > L^{20} 2^{n-1} \mid \cG_\nu) \le \P( \operatorname{Herd}_{B, n} \;|\; \cG_\nu) \ge 1 - \exp(-{(2^n L)}^{c/ \log (n + \log L)}), 
$$
which yields \eqref{E:Dxm} after simplification. By the Borel-Cantelli lemma there exists a random $D > 0$ such that for all $x \in \Z^2$, on $\cG_\nu$ we have
\begin{equation}
\label{E:sigBx}
\sig_{B(x)} - \tau_{B(x)} \le D + [\log (\|x\|_1 + 3)]^{C \log \log \log (\|x\|_1 + 3)},
\end{equation}
which yields \eqref{E:DxDbound}, completing the proof of point $3$. 

Next, the Lipschitz bound \eqref{E:Lipschitz} and \eqref{E:sigBx} implies that $\sig_{B(x)} \le 10 L \|x\|_1$ for all large enough $x$, and so $\mathbf{I}_t \cap B(0, ct) = \emptyset$ for all large enough $t$ for some $\nu$-dependent $c > 0$. Also, there exists $C > 0$ such that almost surely, $\mathbf{I}_t \sset B(0, Ct)$ for all large enough $t$ by the corresponding result for the SI model without recovery, \cite[Theorem 1]{kesten2005spread}. Together these results give point $2$.

Finally, let $\mathbf{S}_{t, r}$ denote the set of susceptible particles at time $t$ in some $H_B$ with $d(B, 0) \le r$. Standard random walk estimates (Lemma \ref{L:rw-estimate}) imply that for any $\ep, r > 0$, for all large enough $t$ we have
$$
\mathbf{S}_{t, t(r- \ep)} \sset \mathbf{S}_t \cap D(0, rt) \sset \mathbf{S}_{t, t(r + \ep)}.
$$
Combining this with Proposition \ref{P:exists-delta} and Lemma \ref{L:herd-and-survive}.3 yields point $4$.
 
 \bibliographystyle{alpha}
 
 \bibliography{SIR}

\end{document}